\documentclass[a4paper]{amsart}

\usepackage[utf8]{inputenc} 

\usepackage{amsmath, amssymb, amsthm, mathtools, mathrsfs} 
\usepackage{bm} 
\numberwithin{equation}{section} 

\usepackage{tikz-cd} 

\usepackage[top=2.3cm, bottom=2.3cm, left=2.3cm, right=2.3cm]{geometry} 
\usepackage{setspace} 
\usepackage[shortlabels]{enumitem} 

\usepackage{xcolor} 

\usepackage{hyperref} 
\definecolor{urlcolor}{rgb}{0,0,0}
\definecolor{linkcolor}{rgb}{.7,0.10,0.2}
\definecolor{citecolor}{rgb}{.12,.54,.11}
\hypersetup{
    breaklinks=true,     
    colorlinks=true,     
    urlcolor=urlcolor,   
    linkcolor=linkcolor, 
    citecolor=citecolor  
}

\usepackage{cleveref}

\crefname{section}{section}{sections}
\Crefname{section}{Section}{Sections}

\crefname{theorem}{theorem}{theorems}
\Crefname{theorem}{Theorem}{Theorems}

\crefname{lemma}{lemma}{lemmas}
\Crefname{lemma}{Lemma}{Lemmas}

\crefname{proposition}{proposition}{propositions}
\Crefname{proposition}{Proposition}{Propositions}

\crefname{corollary}{corollary}{corollaries}
\Crefname{corollary}{Corollary}{Corollaries}

\crefname{equation}{equation}{equations}
\Crefname{equation}{Equation}{Equations}

\usepackage[
    style=alphabetic,     
    doi=false,
    isbn=false,
    url=true,
    eprint=false,
    block=space,          
    maxcitenames=100,     
    maxbibnames=100,      
    sorting=anyt          
]{biblatex}
\newtheorem{theorem}{Theorem}[section]
\newtheorem{corollary}[theorem]{Corollary}
\newtheorem{proposition}[theorem]{Proposition}
\newtheorem{proposition-definition}[theorem]{Proposition-Definition}
\newtheorem{lemma}[theorem]{Lemma}

\newtheorem{definition}[theorem]{Definition}
\theoremstyle{remark}
\newtheorem{remark}{Remark}
\newtheorem{example}{Example}

\newcommand{\abs}[1]{\lvert #1 \rvert}  

\renewcommand{\subset}{\subseteq}  
\renewcommand{\setminus}{\smallsetminus}  

\newcommand{\BD}{\mathbb{D}}

\newcommand{\BQ}{\mathbb{Q}}

\newcommand{\BoA}{\mathbf{A}}

\newcommand{\BoC}{\mathbf{C}}

\newcommand{\BoG}{\mathbf{G}}

\newcommand{\BoN}{\mathbf{N}}

\newcommand{\BoQ}{\mathbf{Q}}

\newcommand{\BoZ}{\mathbf{Z}}
\newcommand{\bd}{\mathbf{d}}

\newcommand{\CA}{\mathcal{A}}

\newcommand{\CC}{\mathcal{C}}
\newcommand{\CD}{\mathcal{D}}

\newcommand{\CF}{\mathcal{F}}
\newcommand{\CG}{\mathcal{G}}
\newcommand{\CH}{\mathcal{H}}

\newcommand{\CL}{\mathcal{L}}
\newcommand{\CM}{\mathcal{M}}
\newcommand{\CN}{\mathcal{N}}

\newcommand{\FL}{\mathfrak{L}}
\newcommand{\FM}{\mathfrak{M}}

\newcommand{\FP}{\mathfrak{P}}

\newcommand{\SA}{\mathscr{A}}

\newcommand{\SC}{\mathscr{C}}

\newcommand{\SG}{\mathscr{G}}

\newcommand{\NN}{\mathrm{N}}
\newcommand{\MO}{\mathrm{MO}}

\newcommand{\rmB}{\mathrm{B}}  
\newcommand{\rmBPS}{\mathrm{BPS}}
\newcommand{\BPS}{\mathcal{BPS}}  
\newcommand{\GL}{\mathrm{GL}}  
\newcommand{\Hom}{\mathrm{Hom}}  
\newcommand{\Spec}{\mathrm{Spec}}  
\newcommand{\Sym}{\mathrm{Sym}}  

\newcommand{\Tr}{\mathrm{Tr}} 
\newcommand{\Jac}{\mathrm{Jac}} 
\newcommand{\ad}{\mathrm{ad}}
\newcommand{\rmb}{\mathrm{b}}
\newcommand{\rmc}{\mathrm{c}}
\newcommand{\cms}{/\!\!/}
\newcommand{\coker}{\mathrm{coker}}
\newcommand{\dd}{\mathbf{d}}
\newcommand{\ff}{\mathbf{f}}
\newcommand{\ee}{\mathbf{e}}
\newcommand{\pD}{{^{\frakp}\!\mathcal{D}}}
\newcommand{\Det}{\mathrm{Det}}
\newcommand{\Ext}{\mathrm{Ext}}
\newcommand{\id}{\mathrm{id}}  
\newcommand{\rmH}{\mathrm{H}}
\newcommand{\Heis}{\mathsf{Heis}}
\newcommand{\frakg}{\mathfrak{g}}
\newcommand{\Gr}{\mathrm{Gr}}
\newcommand{\frakh}{\mathfrak{h}}
\newcommand{\ICS}{\mathcal{IC}}
\newcommand{\JH}{\mathtt{JH}}
\newcommand{\frakL}{\mathfrak{L}}
\newcommand{\frakm}{\mathfrak{m}}
\newcommand{\rmm}{\mathrm{m}}
\newcommand{\nil}{\mathrm{nil}}  
\newcommand{\Nil}{\mathrm{Nil}}
\newcommand{\frakn}{\mathfrak{n}}
\newcommand{\op}{\mathrm{op}}
\newcommand{\Perv}{\mathrm{Perv}}  

\newcommand{\frakp}{\mathfrak{p}}

\newcommand{\pH}{{^{\mathfrak{p}}\!\mathcal{H}}}
\newcommand{\pt}{\mathrm{pt}}
\newcommand{\ptau}{{^{\mathfrak{p}}\!\tau}}
\newcommand{\re}{\mathrm{re}}
\newcommand{\red}{\mathrm{red}}
\newcommand{\rep}{\mathrm{rep}}
\newcommand{\frakS}{\mathfrak{S}}
\renewcommand{\Spec}{\mathrm{Spec}}  
\newcommand{\bfU}{\mathbf{U}}
\newcommand{\vir}{\mathrm{vir}}

\title[Degenerations of CoHAs of $2$-Calabi--Yau categories]{Degenerations of CoHAs of $\mathbf{2}$-Calabi--Yau categories}
\date{\today}

\author{Lucien Hennecart}

\address{Laboratoire Ami\'enois de Math\'ematique Fondamentale et Appliqu\'ee, CNRS UMR 7352, Universit\'e de Picardie Jules Verne, 33 rue Saint Leu, 80000 Amiens, France}\email{lucien.hennecart@u-picardie.fr}

\author{Shivang Jindal}
\address{Ecole Polytechnique Fédérale de Lausanne (EPFL), CH-1015 Lausanne, Switzerland}\email{shivang.jindal@epfl.ch}

\begin{document}

\begin{abstract}
By work of Davison and Meinhardt, the cohomological Hall algebra of a symmetric quiver with potential admits a geometrically defined filtration (the perverse filtration) whose associated graded is a supercommutative algebra. In the case of the triple quiver of a quiver with the canonical cubic potential, which corresponds to the preprojective algebra of the quiver via dimensional reduction, there is an additional filtration (the less perverse filtration), which is defined more generally for cohomological Hall algebras of suitably geometric $2$-Calabi--Yau categories in work of Davison. In this paper, we show that the degenerations of the cohomological Hall algebras of preprojective algebras and more generally $2$-Calabi--Yau categories with respect to the less perverse filtration is isomorphic to the enveloping algebra of the current Lie algebra of the BPS Lie algebra. This result applies in particular to CoHAs of local systems on Riemann surfaces and Higgs bundles on smooth projective curves. We extend this description to deformations of the cohomological Hall algebra obtained via torus actions on the arrows of the quiver and deformed canonical cubic potentials via the deformed dimensional reduction of Davison--P\u adurariu. We prove all our results at the level of sheafified CoHAs, which allows us to deduce similar statements for all versions of nilpotent CoHAs. Last, we use our results to compare the less perverse filtration on CoHAs of preprojective algebras with the order filtration on the Maulik--Okounkov Yangian, via the comparison isomorphism of Botta--Davison and Schiffmann--Vasserot.
\end{abstract}

\maketitle

\setcounter{tocdepth}{1}

\tableofcontents

\section{Introduction}

\subsection{Background}

Cohomological Hall algebras (CoHAs) of quivers with potentials were introduced around 2010 by Kontsevich and Soibelman \cite{kontsevich2011cohomological}. In parallel, Schiffmann and Vasserot introduced in \cite{schiffmann2013cherednik} the cohomological Hall algebra of the preprojective algebra of the Jordan quiver (i.e. of the commuting stack of pairs of matrices). Their construction is generalized to preprojective algebras of arbitrary quivers by Yang and Zhao in \cite{yang2018cohomological}. CoHAs are associative algebra structures on the (singular, critical or compactly supported) cohomology of the stack of objects in the category of representations of the quiver. The cohomology theory one choses depends on the context: for quivers with potentials, the critical cohomology is well-suited, while for preprojective algebras, one choses the Borel--Moore homology (a.k.a. dual compactly supported cohomology). When the potential vanishes, singular and critical cohomology coincide. It happens that the cohomological Hall algebra of the preprojective algebra of a quiver is isomorphic, via dimensional reduction \cite[Appendix~A]{davison2017critical}, to the cohomological Hall algebra of the triple quiver with its canonical cubic potential. These complementary approaches make CoHAs of preprojective algebras particularly well-suited to a detailed study. Further, CoHAs of preprojective algebras have applications to a broad range of questions in representation theory and enumerative geometry, in relation to the geometry of Nakajima quiver varieties and of symplectic surfaces.

Cohomological Hall algebras of quivers with potentials are further studied in \cite{davison2020cohomological}, where Davison and Meinhardt define the \emph{BPS Lie algebra} and prove the \emph{cohomological integrality isomorphism}. The latter takes the form of a PBW isomorphism for the cohomological Hall algebra of the quiver with potential. It gives a comparison of the underlying graded vector spaces of the cohomological Hall algebra and of the symmetric algebra of the current Lie algebra of the BPS algebra. Precisely, if $\CA$ is the cohomological Hall algebra and $\frakn$ the BPS Lie algebra, then there is a linear isomorphism $\CA\cong\Sym(\frakn[u])$.

Cohomological Hall algebras now play a major role in various subjects of mathematics, including enumerative geometry, algebraic geometry, representation theory and combinatorics. This role is explained by the access to the symmetries of the moduli problems given by CoHAs. Indeed, CoHAs act on the cohomology of related moduli spaces (sheaves on surfaces, quiver varieties). For example, CoHAs admit applications to a proof of the P=W conjecture of de Cataldo--Hausel--Migliorini in \cite{hausel2022p} (proven first by purely geometric means in \cite{maulikp}), to the study of the cohomology of Nakajima quiver varieties and the proof of the positivity of cuspidal polynomials of quivers in \cite{davison2023bps} (a further refinement of Kac polynomials of quivers and the corresponding Kac positivity conjecture). Following \cite{botta2023okounkov,schiffmann2023cohomological}, CoHAs also allow a computation of the Maulik--Okounkov Lie algebra defined in \cite{maulik2019quantum} via geometric R-matrices.

The examination of the structure of CoHAs is preliminary to their use and application to concrete problems. In this paper, we study further cohomological Hall algebras of $2$-Calabi--Yau categories, including that of preprojective algebras of quivers together with their deformations. The geometric examples include the category of Gieseker semistable coherent sheaves of fixed
normalised Hilbert polynomial on a polarised K3 surface and the category of semistable Higgs sheaves of fixed slope on a smooth
projective curve.  Although their precise structure remains unclear, we study the \emph{less perverse degeneration} of these algebras. We show that the associated graded with respect to the less perverse filtration (defined in \cite{davison2020bps}) is isomorphic to the enveloping algebra of the (twisted) current Lie algebra on the BPS Lie algebra. The heart of this paper is \Cref{theorem:main_degeneration_affinized_BPS}, which describes the Lie bracket on the associated graded of the preprojective CoHA with respect to the less perverse filtration, from which all other results follow via the use of \'etale slices and application of strict monoidal functors.

\subsection{Main results}
We give a brief outline of our main results.

Let $Q=(Q_0,Q_1)$ be a quiver and $\Pi_Q$ its preprojective algebra (see \Cref{subsection:preprojective_algebra_stacks}). We let $\BPS_{\Pi_Q}$ be the BPS sheaf (defined as in \cite{davison2020bps}) and $\frakn_{\Pi_Q}^+\coloneqq \rmH^*(\BPS_{\Pi_Q})$ its derived global sections. It has the structure of a Lie algebra, obtained by dimensional reduction from the BPS Lie algebra $\BPS_{\tilde{Q},\tilde{W}}$ of the triple quiver $\tilde{Q}$ with its canonical cubic potential $\tilde{W}$ \cite{davison2020cohomological}. The preprojective sheafified cohomological Hall algebra is denoted by $\SA_{\Pi_Q}^{\psi}$. It is a constructible sheaf on $\CM_{\Pi_Q}$, the moduli space of semisimple $\Pi_Q$-representations (\Cref{subsection:preprojective_algebra_stacks}). It has a structure of an algebra object for a certain monoidal structure on $\CD_{\rmc}^+(\CM_{\Pi_Q})$, the category of locally bounded below constructible complexes on $\CM_{\Pi_Q}$. We also let $\CA_{\Pi_Q}^{\psi}\coloneqq \rmH^*(\SA_{\Pi_Q}^{\psi})$. Here, $\psi$ is a bilinear form on $\BoN^{Q_0}$ whose symmetrization is the symmetrized Euler form of $Q$ (e.g. the Euler form of $Q$ is itself a valid choice for $\psi$) and encodes a twist of the geometric product on the cohomological Hall algebra by a sign. We refer to \Cref{subsection:cohas} for the definition of the twist. We let $\frakL$ be the less perverse filtration on $\CA_{\Pi_Q}^{\psi}$ \cite{davison2020bps}, see \Cref{subsection:support_lessperverse}.

\begin{theorem}[=\Cref{theorem:main_degeneration_affinized_BPS}+\Cref{corollary:coha=envelopingalgebra}]
\label{theorem:main1-preproj}
Let $Q$ be a quiver and $\Pi_Q$ be the preprojective algebra of $Q$. We let $u$ be the first Chern class of a positive determinant line bundle over the stack $\FM_{\Pi_Q}$ of representations of $\Pi_Q$ (\Cref{subsection:determinant_line_bundle_preproj}). Then, there is an isomorphism of algebras
\[
 \bfU(\frakn_{\Pi_Q}^{+,\rmBPS}\otimes\rmH^*_{\BoC^*}(\pt))\rightarrow \Gr^{\frakL}(\CA_{\Pi_Q}^{\psi})\,,
\]
where $\bfU$ denotes the universal enveloping algebra, $\frakn_{\Pi_Q}^{+,\rmBPS}\otimes\rmH^*_{\BoC^*}(\pt)$ has the Lie bracket ${^\frakp[-,-]}$ given by the $\abs{-}$-twisted $\rmH^*_{\BoC^*}(\pt)$-linear extension of the Lie bracket on $\frakn_{\Pi_Q}^{+,\rmBPS}$. This means that for any dimension vectors $\dd,\ee\in\BoN^{Q_0}$, $x\in\frakn_{\Pi_Q,\dd}^{+,\rmBPS}$, $y\in\frakn_{\Pi_Q,\ee}^{+,\rmBPS}$ and $m,n\in\BoN$,
\[
 {^\frakp[xu^m,yu^n]}=\frac{\abs{\dd}^m\abs{\ee}^n}{\abs{\dd+\ee}^{m+n}}[x,y]u^{m+n}
\]
where $\abs{\dd}=\sum_{i\in Q_0}r_i\dd_i$ if the determinant line bundle on $\FM_{\Pi_Q}$ is determined by $\GL_{\dd}\rightarrow\BoG_{\rmm}$, $(g_i)_{i\in Q_0}\mapsto\prod_{i\in Q_0}\det(g_i)^{r_i}$ and $[-,-]$ denotes the Lie bracket on the BPS Lie algebra $\frakn_{\Pi_Q}^{+,\rmBPS}$.
\end{theorem}

There is an extension of this theorem to the more general $2$-Calabi--Yau setting. When $\CC$ is a $2$-Calabi--Yau abelian category (\S\ref{subsection:2CY_Abelian_categories}), we define the BPS sheaf $\BPS_{\CC}$ as in \cite{davison2023bps}. It has the structure of a Lie algebra object in the category of perverse sheaves on the moduli space $\CM_{\CC}$ of semisimple objects in $\CC$ (a.k.a. the good moduli space of the stack $\FM_{\CC}$ of objects in $\CC$). We let $\frakn_{\CC}^{+,\rmBPS}\coloneqq\rmH^*(\BPS_{\CC})$. It has a Lie algebra structure. It is the \emph{BPS Lie algebra} of $\CC$. When $\CC$ is the category of representations of the preprojective algebra of a quiver $Q$, it coincides with the BPS Lie algebra of the triple quiver with its canonical cubic potential $(\tilde{Q},\tilde{W})$. It is proven in \cite{davison2023bps} that the BPS Lie algebra of $\CC$ has the structure of a generalized Kac--Moody Lie algebra. We let $\SA_{\CC}\in\CD^+_{\rmc}(\CM_{\CC})$ be the sheafified cohomological Hall algebra of $\CC$ and as usual, we fix a bilinear form $\psi$ on the numerical Grothendieck group $\pi_0(\FM_{\CC})$ of $\CC$ whose symmetrization is the Euler form of $\CC$ and we let $\SA_{\CC}^{\psi}$ be the cohomological Hall algebra of $\CC$ with the sign-twisted multiplication. We denote by $\frakL$ the less perverse filtration on $\rmH^*(\SA_{\CA}^{\psi})$ defined in terms of the good moduli space morphism $\JH_{\CC}\colon\FM_{\CC}\rightarrow\CM_{\CC}$ \cite{davison2021purity}. The CoHA product on $\rmH^*(\SA_{\CC}^{\psi})$ is compatible with the less perverse filtration, so that there is an induced algebra structure on the associated graded $\Gr^{\FL}(\rmH^*(\SA_{\CC}^{\psi}))=\rmH^*(\pH(\SA_{\CC}^{\psi}))$. The following theorem gives a description of this algebra. We refer to \Cref{lemma:abs_monoid_morphism} for the definition of the monoid morphism $\abs{-}\colon\pi_0(\FM_{\CA})\rightarrow\BoN$.

\begin{theorem}[=\Cref{theorem:less_perverse_2CY}+\Cref{corollary:enveloping_affinizedBPS}]
\label{theorem:main2-preproj}
 Let $\CC$ be a suitably geometric $2$-Calabi--Yau abelian category (i.e. satisfying the assumptions given in \S\ref{subsection:2CY_Abelian_categories}). Then, the perversely degenerated Lie bracket on $\pH(\SA_{\CC}^{\psi})$ preserves the subobject $\BPS_{\CC}\otimes\rmH^*_{\BoC^*}(\pt)$ and the Lie bracket ${^\frakp[-,-]}$ on $\BPS_{\CC}\otimes\rmH^*_{\BoC^*}(\pt)$ is given by the $\abs{-}$-twisted trivial $\rmH^*_{\BoC^*}(\pt)$-linear extension of the Lie bracket on $\BPS_{\CC}$. In particular, the subspace $\frakn_{\CC}^{+,\rmBPS}\otimes\rmH^*_{\BoC^*}(\pt)\subset \rmH^*(\SA_{\CC}^{\psi})$ is stable under the perversely degenerated Lie bracket, which is given by the $\abs{-}$-twisted $\rmH^*_{\BoC^*}(\pt)$-linear extension of the Lie bracket on $\frakn_{\CC}^{+,\rmBPS}$. This means that for any $a,b\in\pi_0(\FM_{\CC})$, $x\in\frakn_{\CC,a}^{+,\rmBPS}$, $y\in\frakn_{\CC,b}^{+,\rmBPS}$, $m,n\in\BoN$, and $u$ the first Chern class of a determinant line bundle over $\FM_{\CC}$,
 \[
  {^\frakp[xu^m,yu^n]}=\frac{\abs{a}^m\abs{b}^n}{\abs{a+b}^{m+n}}[x,y]u^{m+n}
 \]
where $[-,-]$ denotes the Lie bracket on the BPS Lie algebra $\frakn_{\CC}^{+,\rmBPS}$.

 In addition, there are isomorphisms of associative algebras objects
 \[
  \bfU(\BPS_{\CC}\otimes\rmH^*_{\BoC^*}(\pt))\rightarrow \pH(\SA_{\CC}^{\psi})
 \]
 in $\CD^+_{\rmc}(\CM_{\CC})$ and of algebras
 \[
  \bfU(\frakn_{\CC}^{+,\rmBPS}\otimes\rmH^*_{\BoC^*}(\pt))\rightarrow\Gr^{\frakL}(\rmH^*(\SA_{\CC}^{\psi}))\,,
 \]
 where $\pH(\SA_{\CC}^{\psi})\coloneqq\bigoplus_{i\in\BoZ}\pH^i(\SA_{\CC}^{\psi})[-i]$ is the total perverse cohomology sheaf of $\SA_{\CC}^{\psi}$.
\end{theorem}

The last layer of generality regarding the perversely degenerated cohomological Hall algebra of the preprojective algebra of a quiver consists in adding a torus action on the arrows of the quiver for which the preprojective relation is homogeneous. In this case, we denote by $\SA_{\Pi_Q}^{T',\psi}$ the corresponding sheafified cohomological Hall algebra. It is a complex of constructible sheaves in the $T'$-equivariant derived category $\CD_{\rmc,T'}(\CM_{\Pi_Q})$. The absolute cohomological Hall algebra $\CA_{\Pi_Q}^{T',\psi}\coloneqq\rmH^*(\SA_{\Pi_Q}^{T',\psi})$ is a $\rmH^*_{T'}(\pt)$-module.

\begin{theorem}[=\Cref{corollary:TprimeCoHA}+\Cref{corollary:Tprimeenveloping}]
Let $Q$ be a quiver and $\Pi_Q$ be the preprojective algebra of $Q$. We let $u$ be the first Chern class of a positive determinant line bundle over $\FM_{\Pi_Q}$. We let $T'$ be a torus acting on the arrows of $\overline{Q}$ such that the preprojective relation is homogeneous (see \Cref{subsection:preprojective_algebra_stacks} for the definition). Then, there is an isomorphism of algebras
\[
 \bfU(\frakn_{\Pi_Q}^{+,\rmBPS}\otimes\rmH^*_{T'\times\BoC^*}(\pt))\rightarrow \Gr^{\frakL}(\rmH^*(\SA_{\Pi_Q}^{T',\psi}))\,,
\]
where $\frakn_{\Pi_Q}^{+,\rmBPS}\otimes\rmH^*_{T'\times\BoC^*}(\pt)$ has the Lie bracket given by the $\abs{-}$-twisted $\rmH^*_{T'\times\BoC^*}(\pt)$-linear extension of the Lie bracket on $\frakn_{\Pi_Q}^{+,\rmBPS}$.
\end{theorem}

These theorems constitute a step towards a full understanding of the CoHA product on cohomological Hall algebras of $2$-Calabi--Yau categories. Before taking any perverse degeneration, the behaviour of the product takes, in known cases, a very interesting form involving avatars of vertex algebras. We refer the reader to \cite{davison2022affine} and \cite{jindal2024coha} for full determination of the entire cohomological Hall algebras for affine type A quivers, which involve the $W_{1+\infty}$ Lie algebra. We also refer to \cite{mellit2023coherent} for a study of the cohomological Hall algebras of finite length coherent sheaves on smooth surfaces. For \emph{nilpotent} CoHAs of affine quivers, we refer to \cite{diaconescu2025cohomological} for descriptions by generators and relations involving Yangians. Our degeneration results apply to the various versions of nilpotent CoHAs, is explained in \Cref{section:degenerations-nilpotent}.

\subsection{Ackowledgements}
The first author thanks CNRS Math\'ematiques for financial support via a PEPS JCJC 2025. We would like to thank Ben Davison for his careful reading and insightful comments on a preliminary version of this paper. The second author would also like to thank Ben Davison for various discussions about the contents of this paper.

\subsection{Conventions and Notations}
\begin{enumerate}
 \item If $\CF$ is a constructible complex on an algebraic variety $X$, we abbreviate $\rmH^*(\CF)\coloneqq\rmH^*(X,\CF)$ for the cohomology of $X$ with coefficients in $\CF$.
 \item If $\CF$ is a semisimple constructible sheaf on an algebraic variety $X$, we denote by $\FP$ the perverse filtration on $\rmH^*(\CF)$. We make an exception for the \emph{less} perverse filtration of cohomological Hall algebras of $2$-Calabi--Yau categories, which is denoted by $\FL$.
 \item If $\CF$ is a constructible sheaf on an algebraic variety $X$, we denote by $\pH(\CF)\coloneqq\bigoplus_{i\in\BoZ}\pH^i(\CF)[-i]\in\Perv^{\BoZ}(X)$ its total perverse cohomology sheaf, a $\BoZ$-graded perverse sheaf on $X$.
 \item If $X$ is an algebraic variety with an action of an algebraic group $G$, we denote by $X/G$ the quotient stack. That is, we omit the brackets to lighten the notation.
 \item If $f\colon X\rightarrow\BoA^1$ and $g\colon Y\rightarrow\BoA^1$ are two regular functions on some algebraic varieties or stacks $X,Y$, we denote by $f\boxplus g\colon X\times Y\rightarrow\BoA^1$ the function $(x,y)\mapsto f(x)+g(y)$.
 \item Let $\frakn$ be a Lie algebra graded by a monoid $M$ such that $\frakn_0=0$ and let $\abs{-}\colon M\rightarrow\BoZ$ be a monoid homomorphism such that $\abs{m}\neq 0$ for $m\neq 0$. Then, the $\abs{-}$-twisted Lie bracket on $\frakn[u]$ is defined by
 \[
  [xu^p,yu^q]=\frac{\abs{m}^p\abs{n}^q}{\abs{m+n}^{p+q}}[x,y]u^{p+q}
 \]
for $x\in\frakn_m$, $y\in\frakn_n$, $p,q\in\BoN$. The multiplication by $u$ on $\frakn[u]$ is a derivation, that is $u\cdot [x,y]=[u\cdot x,y]+[x,u\cdot y]$ for any $x,y\in\frakn[u]$. There is another derivation $\partial$ on $\frakn[u]$ defined by $\partial (xu^p)=\abs{m}xu^{p-1}$ for $x\in\frakn_m$ and $p\in\BoN$. We have $[\partial,u]=\abs{m}xu^m$.
\end{enumerate}

\section{Cohomological Donaldson--Thomas theory of quivers with potential}
We recall background material regarding cohomological Hall algebras of quivers with potentials \cite{kontsevich2011cohomological}, including the powerful relative/sheafified cohomological Hall algebras \cite{davison2020cohomological}.

\subsection{Constructible sheaves}
\label{subsection:constructible_sheaves}
The base field for all geometric objects (schemes, algebraic varieties, stacks, algebraic groups) is the field of complex numbers $\BoC$. Let $X$ be a finite type separated scheme over $\BoC$. We let $\CD_{\rmc}(X)$ be the derived category of constructible complexes of rational vector spaces over $X$. It is defined as the subcategory of $\CD(\mathrm{Sh}(X(\BoC)^{\mathrm{an}}))$, the derived category of the Abelian category of sheaves of $\BoQ$-vector spaces on the analytic space $X(\BoC)^{\mathrm{an}}$, of complexes with constructible cohomology sheaves. We let $\CD^{\rmb}_{\rmc}(X)$ be the full subcategory of bounded complexes and $\CD^+_{\rmc}(X)$ the subcategory of complexes that are bounded below. We let $\Perv(X)\subset \CD_{\rmc}(X)$ be the category of perverse sheaves. It is an Abelian category realized as the heart of the perverse t-structure $(\pD_{\rmc}^{\leq 0}(X),\pD_{\rmc}^{\geq 0}(X))$. We denote by $\ptau^{\leq i}, \ptau^{\geq i}$ the perverse truncation functors and $\pH^i=\ptau^{\leq i}\ptau^{\geq i}[i]\colon \CD_{\rmc}(X)\rightarrow\Perv(X)$ the perverse cohomology functors. If $\CF$ is a complex of constructible sheaves on $X$, we let $\pH(\CF)\coloneqq \bigoplus_{i\in\BoZ}\pH^i(\CF)[-i]$ be its total perverse cohomology. We see the total perverse cohomology $\pH$ as a functor $\pH\colon\CD_{\rmc}(X)\rightarrow\Perv^{\BoZ}(X)$ with target the abelian category of cohomologically graded perverse sheaves on $X$, with graded morphisms. In $\Perv^{\BoZ}(X)$, kernels, cokernels, etc. are computed componentwise. While not faithful, the restriction of the functor $\pH$ to the subcategory of $\CD^{+}_{\rmc}(X)$ of complexes that are bounded below is conservative, meaning that a morphism between constructible complexes is an isomorphism if and only if it is an isomorphism after applying $\pH$. This is easily seen by considering the long exact sequence of perverse cohomology sheaves, see \cite[Proposition~1.3.7]{beilinson2018faisceaux}.

If $X$ is an algebraic variety and $f\colon X\rightarrow\BoA^1$ a regular function, there is a perverse exact functor $\varphi_f\colon\CD^+_{\rmc}(X)\rightarrow\CD^+_{\rmc}(X)$, this is the \emph{vanishing cycle sheaf functor}. We refer the reader to \cite{davison2020cohomological} for the essential properties that it satisfies.

If $\CM$ is a monoid in the category of schemes with finite type separated connected components (see \Cref{example:monoid} for an example), $\oplus\colon\CM\times\CM\rightarrow\CM$ denotes the monoid structure (which we assume to be a finite morphism) and $\eta\colon \pt\rightarrow \CM$ is the unit, the formula $\CF\boxdot\CG\coloneqq\oplus_*(\CF\boxtimes\CG)$ defines a monoidal structure on $\CD_{\rmc}^+(\CM)$. The monoidal unit is given by $\eta_*\BoQ_{\pt}$. By finiteness of $\oplus$, the pushforward $\oplus_*$ is perverse exact and so preserves the subcategory $\Perv(\CM)$ of perverse sheaves. We obtain induced monoidal structures on $\Perv(\CM)$ and $\Perv^{\BoZ}(\CM)$.

Let $T'$ be a torus. A monoid in the category of schemes is \emph{$T'$-equivariant} if there is an action of $T'$ on $\CM$ such that the structure morphism $\oplus\colon\CM\times\CM\rightarrow\CM$ is $T'$-equivariant for the diagonal $T'$-action on $\CM\times\CM$. We let $\CM^{T'}\coloneqq\CM/T'$ be the quotient stack. Then, $\CM^{T'}$ is a monoid over $\rmB T'$, i.e. the structure map induces a morphism $\oplus^{T'}\colon\CM^{T'}\times_{\rmB T'}\CM^{T'}\rightarrow\CM^{T'}$. To ease the notation, we will often write $\oplus=\oplus^{T'}$. We consider the perverse t-structure on the $T'$-equivariant derived category $\CD^+_{\rmc}(\CM^{T'})$ so that a constructible sheaf $\CF$ on $\CM^{T'}$ is perverse if and only if its pullback $(\CM\rightarrow\CM^{T'})^*\CF$ to $\CM$ is perverse. In particular, if $Y\subset \CM$ is a smooth connected component, the shifted constant sheaf $\BoQ_{Y/T'}[\dim Y]$ is in $\Perv(\CM^{T'})$. Note that usually, perverse sheaves on Artin stacks (in particuler $\CM^{T'}$) are defined by descent from schemes so that if $\FM$ is a smooth stack, $\BoQ_{\FM}[\dim \FM]$ is perverse. Hence, our definition of perverse sheaves on $\CM^{T'}$ differs by a shift. The formula for $\CF\boxdot \CG$ induces monoidal structures on both $\CD^+_{\rmc}(\CM^{T'})$ and $\Perv(\CM^{T'})$.

We also consider the category $\Perv^{\BoZ}(\CM^{T'})\coloneqq\prod_{i\in\BoZ}\Perv(\CM^{T'})$ of cohomologically graded $T'$-equivariant perverse sheaves. The $i$th component of $\CF\in\Perv^{\BoZ}(\CM^{T'})$ is denoted by $\CF[-i]$. The formula for $\CF\boxdot\CG$ induces a monoidal structure on $\Perv^{\BoZ}(\CM^{T'})$.

The total perverse cohomology functor
\[
\begin{matrix}
 \pH&\colon&\CD^+_{\rmc}(\CM^{T'})&\rightarrow&\Perv^{\BoZ}(\CM^{T'})\\
 &&\CF&\mapsto&\bigoplus_{i\in\BoZ}\pH^i(\CF)[-i]
\end{matrix}
\]
is a monoidal functor since $\oplus$ is perverse exact. If $f\colon\CF\rightarrow\CG$ is a morphism between constructible sheaves in $\CD^+_{\rmc}(\CM)$, the morphism $\pH(f)\colon\pH(\CF)\rightarrow\pH(\CG)$ is called its \emph{perverse degeneration}. In this paper, we will study the perverse degenerations of Lie brackets on cohomological Hall algebras.

\begin{example}
\label{example:monoid}
A typical example of monoid $\CM$ appearing in the study of cohomological Hall algebras of quivers is $\CM=\bigoplus_{n\in\BoN}\BoC^n\cms\frakS_n$, where $\frakS_n$ acts on $\BoC^n$ by permuting the $n$ components (this is the monoid one gets for the study of the CoHA of the Jordan quiver -- the quiver with one vertex and one loop). The unit is the inclusion of the connected component $\BoC^0\cms\frakS_0=\pt\subset\CM$ and the structure morphism is induced by the $(\frakS_m\times\frakS_n)-\frakS_{m+n}$-equivariant isomorphisms $\BoC^{m}\times\BoC^n\rightarrow\BoC^{m+n}$ for any $m,n\in\BoN$, which give morphisms $\BoC^m\cms \frakS_m\times\BoC^n\cms \frakS_n\rightarrow\BoC^{m+n}\cms \frakS_{m+n}$ between the quotients.
\end{example}

If $\imath\colon(\CN,\oplus)\rightarrow(\CM,\oplus)$ is a \emph{saturated submonoid} in the sense that the diagram
\[\begin{tikzcd}
	{\CN\times\CN} & \CN \\
	{\CM\times\CM} & \CM
	\arrow["\oplus", from=1-1, to=1-2]
	\arrow["{\imath\times\imath}"', from=1-1, to=2-1]
	\arrow["\imath", from=1-2, to=2-2]
	\arrow["\oplus", from=2-1, to=2-2]
\end{tikzcd}\]
is Cartesian, then the exceptional pullback $\imath^!\colon \CD^+_{\rmc}(\CM)\rightarrow\CD^+_{\rmc}(\CN)$ is strict monoidal (by base-change). If $f\colon \CM\rightarrow\CN$ is a morphism of monoids in the category of schemes, then $f_*$ is a strict monoidal functor. In particular, if $\CN=\pt$ is the trivial monoid, the derived global section functor $f_*=\rmH^*$ is strict monoidal. If $\CM$ is a commutative monoid, i.e. $\oplus\circ\mathrm{sw}=\oplus$ where $\mathrm{sw}$ is the operator swapping the two factors of $\CM\times\CM$, then the monoidal structures $\boxdot$ on the various categories $\CD^+_{\rmc}(\CM)$, $\Perv(\CM)$, $\Perv^{\BoZ}(\CM)$ and their $T'$-equivariant versions are symmetric.

\begin{remark}
\label{remark:twist_monoidal_structure}
If $\theta\colon\pi_0(\CM)\times\pi_0(\CM)\rightarrow\BoZ$ is a bilinear form on the monoid $\pi_0(\CM)$ of connected components of $\CM$, we may twist the monoidal structure $\boxdot$ on $\CD^+_{\rmc}(\CM)$ by $\theta$ in the following way. For $\CF,\CG\in\CD^+_{\rmc}(\CM)$, we let $\CF\boxdot_{\theta}\CG\coloneqq \bigoplus_{a,b\in\pi_0(\CM)}(\oplus_{a,b})_*(\CF_a\boxtimes\CG_b)[\theta(a,b)]$. The $\theta$-twisted monoidal structure $\boxdot_{\theta}$ is considered when studying CoHAs of non-symmetric quivers, in which case $\theta$ is the antisymmetrized Euler form of $Q$ (\Cref{subsection:cohas}). For symmetric quivers, $\theta=0$ and it is not necessary to twist the monoidal structures.
\end{remark}

\subsection{Associative and Lie algebras in categories of perverse sheaves}

As in \cite{davison2020cohomological}, the cohomological Hall algebras we consider will be defined as algebra objects in monoidal categories of constructible complexes. We briefly define how these algebras are defined.

Let $\CM$ be a monoid in the category of schemes as in  \Cref{subsection:constructible_sheaves} and $\CD^+_{\rmc}(\CM)$ the category of locally bounded below (i.e. bounded below on each connected component of $\CM$) complexes of constructible sheaves with the monoidal structure $\boxdot$ defined in \Cref{subsection:constructible_sheaves} induced by the binary structural operation on $\CM$. An associative algebra object in $\CD^+_{\rmc}(\CM)$ is a complex $\CF\in\CD^+_{\rmc}(\CM)$ together with a morphism $m\colon \CF\boxdot\CF\rightarrow\CF$ called the \emph{multiplication} and a unit morphism $u\colon\eta_{*}\BoQ_{\pt}\rightarrow\CF$ where $\eta\colon\pt\rightarrow\CM$ is the unit of the monoid $\CM$. They satisfy the associativity axiom which is the commutativity of the following diagram:
\[\begin{tikzcd}
	{\CF\boxdot\CF\boxdot\CF} & {\CF\boxdot\CF} \\
	{\CF\boxdot\CF} & \CF
	\arrow["{\id_{\CF}\boxdot m}", from=1-1, to=1-2]
	\arrow["{m\boxdot \id_{\CF}}"', from=1-1, to=2-1]
	\arrow["m", from=1-2, to=2-2]
	\arrow["m"', from=2-1, to=2-2]
\end{tikzcd}\]
and the unitality axiom, given by the commutativity of the diagram
\[\begin{tikzcd}
	{\CF\cong\CF\boxdot\BoQ_{\pt}} & {\CF\boxdot\CF} \\
	& \CF
	\arrow["{\id_{\CF}\boxdot u}", from=1-1, to=1-2]
	\arrow["{\id_{\CF}}"', from=1-1, to=2-2]
	\arrow["m", from=1-2, to=2-2]
\end{tikzcd}\]
and of the similar diagram with $\CF\boxdot\BoQ_{\pt}$ replaced by $\BoQ_{\pt}\boxdot\CF$.

We now assume that $\CM$ is commutative, so that $\CD^+_{\rmc}(\CM)$ is symmetric. For $\CF,\CG\in\CD^+_{\rmc}(\CM)$, we let $T_{\CF,\CG}\colon\CF\boxdot\CG\rightarrow\CG\boxdot\CF$ be the braiding isomorphism. A Lie algebra object is an object $\CL\in\CD^+_{\rmc}(\CM)$ together with a morphism $c\colon\CL\boxdot\CL\rightarrow\CL$ (the \emph{Lie bracket}) that is antisymmetric: $c\circ T_{\CL,\CL}\cong -c$, and that satisfies the Jacobi identity. The Jacobi identity can be expressed as follows: $c\circ (c\boxdot \id_{\CL})+c\circ (c\boxdot \id_{\CL})\circ T_{1,2,3}+c\circ (c\boxdot \id_{\CL})\circ T_{1,3,2}=0$ where $T_{1,2,3}$ and $T_{2,3,1}$ are the cyclic permutation automorphisms of $\CL^{\boxdot 3}$, which may be defined in terms of the braiding operator $T_{\CL,\CL}$ as $T_{1,2,3}=(\id_{\CL}\boxdot T_{\CL,\CL})\circ(T_{\CL,\CL}\boxdot \id_{\CL})$ and $T_{1,3,2}=(T_{\CL,\CL}\boxdot\id_{\CL})\circ(\id_{\CL}\boxdot T_{\CL,\CL})$.

If we apply a monoidal functor to an associative algebra or a Lie algebra object, then we obtain a similar object in the target category. As noticed in \Cref{subsection:constructible_sheaves}, the derived global section functor
\[
 \rmH^*\colon\CD_{\rmc}^+(\CM)\rightarrow \CD^+_{\rmc}(\pt)
\]
is strict monoidal. We may identify $\CD_{\rmc}^+(\pt)$ with the category of $\BoZ$-graded vector spaces with the symmetric monoidal structure involving the Koszul sign rule. Therefore, any algebra or Lie algebra in $\CD_{\rmc}^+(\CM)$ produces a $\BoZ$-graded algebra after taking derived global sections.

\subsection{Quivers with potentials} \label{subsection:quivers-with-potentials}


Let $Q$ be a quiver with set of vertices $Q_0$ and set of arrows $Q_1$ and $\BoC[Q]$ be the path algebra of $Q$. A \emph{potential} for $Q$ is an element $W\in\BoC[Q]$ that is a linear combination of (oriented) cycles. Sometimes, in the literature, a potential is defined as an element of the quotient $\BoC[Q]/[\BoC[Q],\BoC[Q]]_{\mathrm{vec}}$ of the path algebra by the subspace linearly generated by commutators (a basis of which being given by non-based oriented cycles in $Q$). However, it is more convenient for our purposes to consider directly a lift in $\BoC[Q]$ instead of just the class in $\BoC[Q]/[\BoC[Q],\BoC[Q]]_{\mathrm{vec}}$. We recall the definition of the \emph{cyclic derivatives} $\frac{\partial}{\partial a}$ of a potential with respect to arrows $a\in Q_1$  by linearly extending the formula on a cyclic word $a_1\hdots a_r\in\BoC[Q]$ with $a_i\in Q_1$,
\[
 \frac{\partial a_1\hdots a_r}{\partial a}=\sum_{\substack{1\leq i\leq r\\a_i=a}}a_{i+1}\hdots a_ra_1\hdots a_{i-1}\,.
\]
The Jacobi algebra $\Jac(Q,W)$ of the quiver with potential $(Q,W)$ is defined as the quotient of the path algebra $\BoC[Q]$ of $Q$ by the two-sided ideal generated by all cyclic derivatives $\left\langle \frac{\partial W}{\partial a}\colon a \in Q_1\right\rangle$.

A particular case of interest is given by the tripled quiver with the canonical cubic potential, defined as follows. Let $Q$ be a quiver. Then, we define the triple quiver $\tilde{Q}=(Q_0,\tilde{Q}_1)$. It has the same set of vertices as $Q$ but $\tilde{Q}_1=Q_1\sqcup Q_1^*\sqcup Q_0$ where $Q_1^*$ is the set of opposite arrows and $Q_0$ is seen as the set of loops at vertices. For $\alpha\in Q_1$, we denote by $\alpha^*$ the opposite arrow and for $i\in Q_0$, we denote by $\omega_i$ the loop at the $i$th vertex. The canonical cubic potential is $\tilde{W}\coloneqq \left(\sum_{i\in Q_0}\omega_i\right)\left(\sum_{\alpha\in Q_1}[\alpha,\alpha^*]\right)$. For example, if $Q$ is the Jordan quiver (one vertex with a single loop), then $\tilde{Q}$ is the one-vertex three-loop quiver and, letting $\alpha,\alpha^*,\omega$ be labels for the three loops, $\tilde{W}=\omega[\alpha,\alpha^*]=\omega \alpha\alpha^*-\omega \alpha^*\alpha$.

In this paper, we will always assume that all potentials $W$ are such that the critical locus $\mathrm{crit}(\Tr(W))$ is contained in the zero locus $\Tr(W)^{-1}(0)$.

\subsection{Moduli stacks and moduli spaces}


Let $Q=(Q_0,Q_1)$ be a quiver. We denote by $s(\alpha)$ and $t(\alpha)$ the source and target of an arrow $\alpha\in Q_1$. We let $\FM_{Q}$ be the stack of representations of $Q$. It decomposes into connected components $\FM_{Q}\simeq\bigsqcup_{\dd\in\BoN^{Q_0}}\FM_{Q,\dd}$ where $\FM_{Q,\dd}$ is the stack of $\dd$-dimensional representations of $Q$. It can be presented as a quotient stack $\FM_{Q,\dd}=X_{Q,\dd}/\GL_{\dd}$ where $X_{Q,\dd}=\prod_{\alpha\in Q_1}\Hom(\BoC^{\dd_{s(\alpha)}},\BoC^{\dd_{t(\alpha)}})$ and $\GL_{\dd}=\prod_{i\in Q_0}\GL_{\dd_i}$ acts by changing bases at vertices.

The stack $\mathfrak{Exact}_Q$ of extensions of representations of $Q$ can be defined in a similar way. Namely, for $\dd,\dd'\in\BoN^{Q_0}$, we fix a direct sum decomposition $\BoC^{\dd+\dd'}=\BoC^{\dd}\oplus\BoC^{\dd'}$ and we let $X_{Q,\dd,\dd'}$ be the subspace of elements $(x_{\alpha})_{\alpha\in Q_1}\in X_{Q,\dd+\dd'}$ such that for any arrow $\alpha\in Q_1$, $x_{\alpha}(\BoC^{\dd_{s(\alpha)}})\subset\BoC^{\dd_{t(\alpha)}}$. We denote by $P_{\dd,\dd'}\subset\GL_{\dd+\dd'}$ the stabilizer of $\BoC^{\dd}\subset\BoC^{\dd+\dd'}$, a parabolic subgroup. Then, there is an equivalence $\mathfrak{Exact}_{Q}\simeq\bigsqcup_{\dd,\dd'\in\BoN^{Q_0}}X_{Q,\dd,\dd'}/P_{\dd,\dd'}$.

The inclusion $X_{Q,\dd,\dd'}\rightarrow X_{Q,\dd+\dd'}$ is $P_{\dd,\dd'}$ equivariant, where the levi subgroup $P_{\dd,\dd'}$ acts via inclusion $P_{\dd,\dd'} \subset \GL_{\dd+\dd'}$, and thus induces a proper and representable morphism of stacks $p_{\dd,\dd'}\colon\FM_{Q,\dd,\dd'}\rightarrow\FM_{Q,\dd+\dd'}$. The vector bundle $X_{Q,\dd,\dd'}\rightarrow X_{Q,\dd}\times X_{Q,\dd'}$ (projection on the diagonal blocks) is equivariant with respect to the projection on the Levi $P_{\dd,\dd'}\rightarrow\GL_{\dd}\times\GL_{\dd'}$ and therefore induces a smooth morphism $q_{\dd,\dd'}\colon\FM_{Q,\dd,\dd'}\rightarrow\FM_{Q,\dd}\times\FM_{Q,\dd'}$. This morphism is even a vector bundle stack \cite[Definition~0.1]{heinloth2003coherent}, as it can be presented as the total space of a complex of vector bundles  over $\FM_{Q,\dd}\times\FM_{Q,\dd'}$ concentrated in degrees $[-1,0]$ (the RHom-complex).

For $\dd\in\BoN^{Q_0}$, we let $\CM_{Q,\dd}\coloneqq X_{Q,\dd}\cms \GL_{\dd}=\Spec(\BoC[X_{Q,\dd}]^{\GL_{\dd}})$ be the good moduli space for $\FM_{Q,\dd}$. It is in this case an affine GIT quotient. We let $\JH_{Q,\dd}\colon \FM_{Q,\dd}\rightarrow\CM_{Q,\dd}$ be the good moduli space morphism. On $\BoC$-points, it takes a $\dd$-dimensional representation of $Q$ and sends it to its semisimplification (i.e. its associated graded with respect to a Jordan--H\"older filtration). We let $\CM_{Q}\coloneqq\bigsqcup_{\dd\in\BoN^{Q_0}}\CM_{Q,\dd}$, and $\JH_{Q}\coloneqq\bigsqcup_{\dd\in\BoN^{Q_0}}\JH_{Q,\dd}$. The direct sum morphism $\oplus\colon\FM_{Q}\times\FM_{Q}\rightarrow\FM_{Q}$ induces, by universality of the good moduli space morphism, a direct sum morphism $\oplus\colon\CM_Q\times\CM_Q\rightarrow\CM_{Q}$. It is a finite morphism \cite{meinhardt2019donaldson}.

If $T'$ is an auxiliary torus acting on the arrows of $Q$, it acts on the representation spaces $X_{Q,\dd}$ ($\dd\in\BoN^{Q_0}$), and this action commutes with the $\GL_{\dd}$-action. We let $\FM_{Q}^{T'}\coloneqq \FM_Q/T'$ and $\CM_{Q}^{T'}\coloneqq\CM_Q/T'$ be the quotient stacks. Since $\JH_Q$ is $T'$-equivariant, it induces a morphism $\JH_Q^{T'}\colon \FM_{Q}^{T'}\rightarrow\CM_{Q}^{T'}$. The stack $\FM_{Q}^{T'}$ and the moduli space $\CM_{Q}^{T'}$ have natural structure morphisms to $\rmB T'$.

We let $\langle-,-\rangle$ be the Euler form of $Q$. If $\dd,\dd'\in\BoN^{Q_0}$ are dimension vectors, $\langle\dd,\dd'\rangle=\sum_{i\in Q_0}\dd_i\dd'_i-\sum_{\alpha\in Q_1}\dd_{s(\alpha)}\dd_{t(\alpha)}$. The stack $\FM_{Q,\dd}^{T'}$ has relative dimension $-\langle\dd,\dd\rangle$ over $\rmB T'$. We let $(-,-)$ be the \emph{symmetrized Euler form}: $(\dd,\dd')=\langle\dd,\dd'\rangle+\langle\dd',\dd\rangle$.

Let $(Q,W)$ be a quiver with potential. We assume that the potential $W$ is invariant for the $T'$-action. Then, the trace of $W$ acting on representations of $Q$ gives a regular function $\Tr(W)\colon\FM_{Q}^{T'}\rightarrow\BoA^1$.

\subsection{Determinant line bundles}
\label{subsection:determinant_line_bundle_Q}
A determinant line bundle on $\FM_Q$ is a line bundle $\CL$ such that for any $\BoC$-point $x\in\FM_Q$, the pullback $\mathrm{act}_x^*\colon\rmH^*(\FM_Q)\rightarrow\rmH^*(\mathrm{B}\BoG_{\mathrm{m}})$ in cohomology of the action map $\mathrm{act}_x\colon\mathrm{B}\BoG_{\mathrm{m}}\rightarrow\FM_Q$, induced by the inclusion of scalar automorphisms of $x$, is such that $\mathrm{act}_x^*(c_1(\CL))\neq 0$. We assume in addition that $\CL$ is compatible with short exact sequences. Namely, in the correspondence
\[
 \FM_Q\times\FM_Q\xleftarrow{q}\mathfrak{Exact}_{Q}\xrightarrow{p}\FM_Q,
\]
we assume that $p^*\CL\cong q^*(\CL\boxtimes\CL)$.

The line bundle $\CL$ induces a morphism $\mathrm{Det}\colon\FM_Q\rightarrow\mathrm{B}\BoG_{\mathrm{m}}$ which we call the \emph{determinant}. We denote by $\otimes\colon \mathrm{B}\BoG_{\mathrm{m}}\times\mathrm{B}\BoG_{\mathrm{m}}\rightarrow\mathrm{B}\BoG_{\mathrm{m}}$ the morphism induced by the multiplication of $\BoG_{\mathrm{m}}$. Then, the diagram
\[\begin{tikzcd}
	{\FM_Q\times\FM_Q} & {\mathfrak{Exact}_{Q}} & {\FM_Q} \\
	{\mathrm{B}\BoG_{\mathrm{m}}\times\mathrm{B}\BoG_{\mathrm{m}}} && {\mathrm{B}\BoG_{\mathrm{m}}}
	\arrow["{\mathrm{Det}\times\mathrm{Det}}"', from=1-1, to=2-1]
	\arrow["q"', from=1-2, to=1-1]
	\arrow["p", from=1-2, to=1-3]
	\arrow["\mathrm{Det}", from=1-3, to=2-3]
	\arrow["\otimes"', from=2-1, to=2-3]
\end{tikzcd}\]
commutes.

Since $\FM_{Q,\dd}=X_{Q,\dd}/\GL_{\dd}$ is a quotient stack of an affine algebraic variety $X_{Q,\dd}$, a line bundle over $\FM_Q$ is given by a one-dimensional representation of $\GL_{\dd}$, that is by a group homomorphism $\GL_{\dd}\rightarrow\BoG_{\mathrm{m}}$ of the form $(g_i)_{i\in Q_0}\mapsto\prod_{i\in Q_0}\det(g_i)^{m_{\dd,i}}$ for some $m_i\in\BoZ\setminus\{0\}$. The compatibility with short exact sequences means that $m_{\dd,i}$ does not depend on $\dd$ and we will simply write $m_i=m_{\dd,i}$. The non-vanishing of $\mathrm{act}_x^*(c_1(\CL))$ implies that $\sum_{i\in Q_0}\dd_im_i\neq 0$ for any $\dd\in\BoN^{Q_0}$ and so either all of the $m_i$'s are either positive or negative. Conversely, it is easily seen that any choice of $Q_0$-tuple $(m_i)_{i\in Q_0}$ in $(\BoN\setminus{0})^{Q_0}$ or $(-\BoN\setminus{0})^{Q_0}$ defines a positive determinant line bundle on $\FM_Q$.

\subsection{Ext-quivers}
\label{subsection:ext_quivers}
Let $\CC$ be an Abelian category. Let $\underline{S}=\{S_1,\hdots,S_r\}$ be a set of pairwise non-isomorphic objects of $\CC$ (which we do not assume to be simple for the definition, but which will be in practice). The Ext-quiver of $\underline{S}$ is the quiver $Q_{\underline{S}}$ with set of vertices $\underline{S}$ and having $\dim\Ext^1(S_i,S_j)$ arrows from $i$ to $j$ for $1\leq i,j\leq r$. If $M=\bigoplus_{j=1}^rS_j^{m_j}$ is a semisimple object of $\CC$ (where the $S_i$ are pairwise non-isomorphic simple objects of $\CC$), we may associate to $M$ the Ext-quiver $(Q_{\underline{S}},\underline{m}=(m_1,\hdots,m_r))$ of the set of its simple direct summands together with the dimension vector given by the multiplicities of the simple summands of $M$.

Let $Q=(Q_0,Q_1)$ be a quiver. We let $\CC=\mathrm{Rep}(Q)$ be the category of finite dimensional complex representations of $Q$.

We denote by $\zeta\colon \BoN^{r}\rightarrow\BoN^{Q_0}\,,(m_1,\hdots,m_r)\mapsto\sum_{j=1}^rm_j\dim S_j$ the morphism of monoids induced by a set $\underline{S}$.

\begin{lemma}
\label{lemma:compatibility_Eulerforms_extquiver}
 Let $\underline{S}$ be a collection of pairwise non-isomorphic simple representations of $Q$. Then, $\zeta^*\langle-,-\rangle_{Q}=\langle-,-\rangle_{Q_{\underline{S}}}$.
\end{lemma}
\begin{proof}
 This is a straightforward calculation.
\end{proof}

\subsection{Cohomological Hall algebras}
\label{subsection:cohas}

Let $Q$ be a quiver and $T'$ an auxiliary torus acting on the arrows of $Q$. We recall in this section how the cohomological Hall algebras for quivers with potentials are constructed following \cite{davison2020cohomological}, which is a sheafification over the good moduli space of the construction of Kontsevich and Soibelman \cite{kontsevich2011cohomological}. It is an algebra structure on the constructible complex $\SA_{Q}^{T'}\coloneqq (\JH_{Q}^{T'})_{*}\BoQ_{\FM_{Q}^{T'}}^{\vir}\in\CD_{\rmc}^+(\CM_{Q}^{T'})$, where for $\dd\in\BoN^{Q_0}$, $(\BoQ_{\FM_Q^{T'}}^{\vir})_{\FM_{Q,\dd}^{T'}}=\BoQ_{\FM_{Q,\dd}^{T'}}[-\langle\dd,\dd\rangle]$ is the perverse constant sheaf (\Cref{subsection:constructible_sheaves}).

We consider the commutative induction diagram:
\[\begin{tikzcd}
	{\FM_{Q,\dd}^{T'}\times_{\rmB T'}\FM_{Q,\dd'}^{T'}} & {\mathfrak{Exact}_{Q,\dd,\dd'}^{T'}} & {\FM_{Q,\dd+\dd'}^{T'}} \\
	{\CM_{Q,\dd}^{T'}\times_{\rmB T'}\CM_{Q,\dd'}^{T'}} && {\CM_{Q,\dd+\dd'}^{T'}}
	\arrow["{\JH_{Q,\dd}\times\JH_{Q,\dd'}}"', from=1-1, to=2-1]
	\arrow["{q_{\dd,\dd'}}"', from=1-2, to=1-1]
	\arrow["{p_{\dd,\dd'}}", from=1-2, to=1-3]
	\arrow["{\JH_{Q,\dd+\dd'}}", from=1-3, to=2-3]
	\arrow["{\oplus_{\dd,\dd'}}"', from=2-1, to=2-3]
\end{tikzcd}\]
The morphism $p_{\dd,\dd'}$ is proper, and therefore $(p_{\dd,\dd'})_*\cong (p_{\dd,\dd'})_!$. The natural pullback morphism
\[
 \BoQ_{\FM_{Q,\dd+\dd'}^{T'}}\rightarrow (p_{\dd,\dd'})_*\BoQ_{\mathfrak{Exact}_{Q,\dd,\dd'}^{T'}}
\]
therefore dualizes, thanks to the smoothness of all stacks under consideration, to a morphism
\begin{equation}
\label{equation:pushforward}
 (p_{\dd,\dd'})_*\BoQ_{\mathfrak{Exact}_{Q,\dd,\dd'}^{T'}}[2(-\langle\dd,\dd\rangle-\langle\dd',\dd'\rangle-\langle\dd',\dd\rangle)]\rightarrow\BoQ_{\FM_{Q,\dd+\dd'}^{T'}}[-2\langle\dd+\dd',\dd+\dd'\rangle]\,.
\end{equation}
The pullback by $q_{\dd,\dd'}$ is a morphism of constructible complexes
\begin{equation}
\label{equation:pullbackcohaproduct}
 \BoQ_{\FM_{Q,\dd}^{T'}\times_{\rmB T'}\FM_{Q,\dd'}^{T'}}\rightarrow(q_{\dd,\dd'})_*\BoQ_{\mathfrak{Exact}_{Q,\dd,\dd'}^{T'}}\,.
\end{equation}
 By composing $(\JH_{Q,\dd+\dd'})_*[\langle\dd,\dd\rangle+\langle\dd',\dd'\rangle+2\langle\dd',\dd\rangle]$ applied to \Cref{equation:pushforward} with $(\oplus_{\dd,\dd'})_*\circ (\JH_{Q,\dd}\times\JH_{Q,\dd'})_*[-\langle\dd,\dd\rangle-\langle\dd',\dd'\rangle]$ applied to \Cref{equation:pullbackcohaproduct}, we obtain the CoHA multiplication
 \begin{equation}
 \label{equation:coha_multiplication}
  \frakm_{\dd,\dd'}\colon (\JH_{Q,\dd})_*\BoQ_{\FM_{Q,\dd}^{T'}}^{\vir}\boxdot(\JH_{Q,\dd'})_*\BoQ_{\FM_{Q,\dd'}^{T'}}^{\vir}\rightarrow(\JH_{Q,\dd+\dd'})\BoQ_{\FM_{Q,\dd+\dd'}^{T'}}^{\vir}[\langle\dd',\dd\rangle-\langle\dd,\dd'\rangle]\,.
 \end{equation}
By putting together all $\frakm_{\dd,\dd'}$ for $\dd,\dd'\in\BoN^{Q_0}$, we obtain the CoHA product on $\SA_{Q}^{T'}$. If the Euler form $\langle-,-\rangle$ of $Q$ is symmetric, then the shift on the r.h.s. of the CoHA multiplication \Cref{equation:coha_multiplication} vanishes and we get an ordinary degree $0$ multiplication. Then, $\SA_{Q}^{T'}$ is an algebra object in $(\CD^+_{\rmc}(\CM_Q^{T'}),\boxdot)$. If $Q$ is not symmetric, it is necessary to twist the monoidal structure $\boxdot$ as explained in \Cref{remark:twist_monoidal_structure}. The twist $\theta$ is given by the antymmetrized Euler form $\theta(\dd,\dd')=\langle\dd,\dd'\rangle-\langle\dd',\dd\rangle$ of $Q$.

Let $W$ be a potential that is $T'$-invariant. Then, by applying the vanishing cycle sheaf functor $\varphi_{\Tr(W)}$ for the trace function $\Tr(W)\colon\FM_{Q}^{T'}\rightarrow\BoA^1$ to $\frakm$, we obtain the sheafified critical CoHA $\SA_{Q,W}^{T'}$. For this, we rely on the fact that the Jordan--H\"older morphism $\JH$ is approachable by proper maps \cite[\S4.1]{davison2020cohomological}, so that it commutes with the vanishing cycle sheaf functor $\varphi_{\Tr(W)}$. It is again an algebra object in the monoidal category $(\CD^+_{\rmc}(\CM_Q^{T'}),\boxdot)$. If $W=0$, $\SA_{Q,W}^{T'}=\SA_{Q}^{T'}$.

It is customary to twist the multiplication $\frakm\colon\SA_{Q,W}^{T'}\boxdot\SA_{Q,W}^{T'}\rightarrow\SA_{Q,W}^{T'}$ by a compatible system of signs to make the identification with enveloping algebras and quantum group related objects cleaner. The signs involved appear in \cite{kontsevich2011cohomological}, \cite{efimov2012cohomological}, \cite{davison2020cohomological}, \cite{davison2023bps} and the twist is defined as follows.

We let $\psi\colon\BoN^{Q_0}\times\BoN^{Q_0}\rightarrow\BoZ/2\BoZ$ be a bilinear form such that for any $\dd,\dd'\in\BoN^{Q_0}$,
\begin{equation}
\label{equation:equality_psi}
 \psi(\dd,\dd')+\psi(\dd',\dd)=\tau(\dd,\dd')
\end{equation}
where $\tau(\dd,\dd')=\langle\dd,\dd'\rangle+\langle\dd,\dd\rangle\langle\dd',\dd'\rangle\pmod{2}$ is a bilinear form (the quadratic term $\langle\dd,\dd\rangle\langle\dd',\dd'\rangle$ becomes linear after reduction modulo $2$). Then, we define the $\psi$-twisted product $\frakm^{\psi}$ on $\SA_{Q,W}^{T'}$ by $\frakm_{\dd,\dd'}^{\psi}\coloneqq (-1)^{\psi(\dd,\dd')}\frakm_{\dd,\dd'}$ for $\dd,\dd'\in\BoN^{Q_0}$. We denote by $\SA_{Q,W}^{T',\psi}$ the $\psi$-twisted cohomological Hall algebra. Note that the space of bilinear forms $\psi$ satisfying \Cref{equation:equality_psi} is stable under translation by \emph{symmetric} bilinear forms.

If $\tilde{Q}$ is the triple quiver, then the Euler form of $\tilde{Q}$ satisfies $\langle\dd,\dd\rangle_{\tilde{Q}}\equiv 0\pmod 2$ for any $\dd\in\BoN^{Q_0}$ so that $\tau$ can be identified with the symmetrized Euler form of $Q$ modulo $2$. Therefore, a choice for $\psi$ is given by the Euler form of $Q$ itself.

If $Q'$ is the Ext-quiver of a collection of simple representations of $Q$ and $\zeta\colon \BoN^{Q'_0}\rightarrow\BoN^{Q_0}$ is the morphism of monoids that it induces, then if $\psi$ is a bilinear form satisfying \Cref{equation:equality_psi} for $Q$, $\psi\circ \zeta$ is a bilinear form on $\BoN^{Q'_0}$ satisfying \Cref{equation:equality_psi} for $Q'$. This follows from the fact that $\langle-,-\rangle_{Q'}=\langle\zeta(-),\zeta(-)\rangle_Q$ (\Cref{lemma:compatibility_Eulerforms_extquiver}).

We define $\CA_{Q,W}^{T'}\coloneqq \rmH^*(\SA_{Q,W}^{T'})$ and $\CA_{Q,W}^{T',\psi}\coloneqq \rmH^*(\SA_{Q,W}^{T',\psi})$ the absolute cohomological Hall algebra and its $\psi$-twisted version.

\subsection{BPS Lie algebras and associated graded with respect to the perverse filtration}

Let $Q$ be a symmetric quiver and $W\in \BoC[Q]$ a potential. By \cite{davison2020cohomological}, the complex $\SA_{Q,W}^{T'}\in\CD^+_{\rmc}(\CM_{Q}^{T'})$ has no perverse cohomology in degrees $\leq 0$, that is $\pH^i(\SA_{Q,W}^{T'})=0$ for $i\leq 0$. The \emph{BPS sheaf} is defined as the lowest possibly non-vanishing perverse cohomology of $\SA_{Q,W}^{T'}$, that is $\BPS_{Q,W}^{T'}\coloneqq\pH^1(\SA_{Q,W}^{T'})$. There is an adjunction morphism $\BPS_{Q,W}^{T'}[-1]=\ptau^{\leq 1}\SA_{Q,W}^{T'}\rightarrow\SA_{Q,W}^{T'}$.

Since $\SA_{Q}^{T'}\in\CD^+_{\rmc}(\CM_Q^{T'})$ is a semisimple complex, there is an isomorphism of constructible complexes $\SA_{Q}^{T'}\cong\pH(\SA_{Q}^{T'})$ with the total perverse cohomology. Then, since $\SA_{Q,W}^{T'}=\varphi_{\Tr(W)}\SA_{Q}^{T'}$ and the functor $\varphi_{\Tr(W)}$ of vanishing cycles is perverse t-exact, there is an isomorphism of constructible complexes $\SA_{Q,W}^{T'}\cong \pH(\SA_{Q,W}^{T'})$. We define the perverse filtration of $\CA_{Q,W}^{T'}$ by $\FP^i\CA_{Q,W}^{T'}\coloneqq \rmH^*(\ptau^{\leq i}\SA_{Q,W}^{T'})\subset\CA_{Q,W}^{T'}$. It is compatible with the CoHA multiplication. The $\BoZ$-graded perverse sheaf $\pH(\SA_{Q,W}^{T'})\in\Perv^{\BoZ}(\CM_Q^{T'})$ has the associative algebra structure given by $\pH(\frakm)$. By taking the derived global sections, we obtain the perverse degeneration $\rmH^*(\pH(\SA_{Q,W}^{T'}))=\Gr^{\FP}\CA_{Q,W}^{T'}$ of the cohomological Hall algebra of $(Q,W)$.

The following theorem already appears in the literature (\cite[Theorem~3.4]{davison2022affine}, extending \cite[\S1.7]{davison2020cohomological} from the $T'=1$ case). It regards the description of the associated graded of the CoHA of symmetric quivers with potentials with respect to the perverse filtration $\FP$. We recall its statement and proof for completeness and comparison with the case of the \emph{less} perverse filtration $\FL$, defined for triple quivers with the canonical cubic potential, that we define study later (from \Cref{section:less_perverse}).

\begin{theorem}
\label{theorem:asso_graded_perverse_filt}
 Let $Q$ be a symmetric quiver and $W\in\BoC[Q]$ a potential. Let $T'$ be an auxiliary torus acting on the arrows of $Q$ so that $W$ is $T'$-invariant. Then, the cohomological Hall algebra product on $\SA_{Q,W}^{T',\psi}$ induces a Lie bracket on $\BPS_{Q,W}^{T'}[-1]=\pH^{1}(\SA_{Q,W}^{T',\psi})[-1]$. The associated graded of $\SA_{Q,W}^{T',\psi}$ with respect to the perverse filtration is a super-commutative algebra in $\Perv^{\BoZ}(\CM_{Q}^{T'})$. There is an isomorphism of algebra objects
 \[
  \Sym_{\boxdot}(\BPS_{Q,W}^{T'}[-1]\otimes\rmH^*_{\BoC^*}(\pt))\rightarrow \pH(\SA_{Q,W}^{T',\psi})\,.
 \]
\end{theorem}
\begin{proof}

We first prove that $\pH(\SA_{Q,W}^{T',\psi})$ is commutative. By the definition of the CoHA product at the sheaf level, it suffices to consider the case $W=0$. Indeed, if the perversely degenerated CoHA $\pH(\SA_{Q}^{T',\psi})$ is commutative, the same is true for $\pH(\SA_{Q,W}^{T',\psi})$ by applying the perverse exact vanishing cycle sheaf functor $\varphi_{\Tr(W)}$ since $\pH(\varphi_{\Tr(W)}\frakm_{\dd,\dd'}^{\psi})=\varphi_{\Tr(W)}\pH(\frakm_{\dd,\dd'}^{\psi})$.

It also suffices to consider the case $T'=1$ by faithfulness of the forgetful functor $\Perv^{\BoZ}(\CM_{Q}^{T'})\rightarrow\Perv^{\BoZ}(\CM_Q)$. Thus, we only need to prove that the perversely degenerated CoHA $\pH(\SA_{Q}^{\psi})$ is commutative.

We let $^{\frakp}[-,-]$ be the commutator Lie bracket on $\pH(\SA_{Q}^{\psi})$. We would like to show that it vanishes. For a contradiction, if it does not vanish, there exists $i,j\in\BoN$ such that the morphism
\[
 C\colon \pH^i(\SA_Q^{\psi})\boxdot \pH^j(\SA_{Q}^{\psi})\rightarrow\pH^{i+j}(\SA_{Q}^{\psi})
\]
induced by the Lie bracket does not vanish. Since $C$ is a morphism between semisimple perverse sheaves, there exists $x\in \CM_Q$ so that $\imath_x^*C$ does not vanish. The point $x$ corresponds to a semisimple representation $M\cong \bigoplus_{j=1}^rM_j^{m_j}$ of $Q$, where the $M_j$ are pairwise non-isomorphic simple representations of $Q$. We let $\underline{M}\coloneqq \{M_1,\hdots,M_r\}$ be the set of simple direct summands of $M$ and we denote by $Q_{\underline{M}}$ its Ext-quiver. There are embeddings of monoids
\[
\begin{matrix}
 \imath_{\underline{M}}&\colon&\BoN^r&\rightarrow&\CM_Q\\
 &&(m_j)_{1\leq j\leq r}&\mapsto&\bigoplus_{j=1}^rM_j^{\oplus m_j}
\end{matrix}
\]
and
\[
 \begin{matrix}
  \imath_{\Nil}&\colon&\BoN^r&\rightarrow&\CM_{Q_{\underline{M}}}\\
  &&\underline{m}=(m_j)_{1\leq j\leq r}&\mapsto&0_{\underline{m}}
 \end{matrix}\,.
\]
There is an isomorphism of algebras $\imath_{\underline{M}}^*\pH(\SA_{Q}^{\psi})\cong \imath_{\Nil}^*\pH(\SA_{Q_{\underline{M}}}^{\zeta^*\psi})$ where $\zeta\colon\BoN^r\rightarrow\BoN^{Q_0}$ is the monoid homomorphism in \Cref{subsection:ext_quivers}. This comes from the \'etale local description of $\FM_{Q}$ around the image of $\imath_{\underline{M}}$ by $\FM_{Q_{\underline{M}}}$ around the image if $\imath_{\Nil}$.

Moreover, by $\BoC^*$-equivariance, $\imath_{\Nil}^*\SA_{Q}^{\psi}\cong\rmH^*(\SA_{Q_{\underline{M}}}^{\psi})$ is the absolute cohomological Hall algebra for $Q_{\underline{M}}$. Since it is supercommutative \cite{efimov2012cohomological}, its Lie bracket vanishes, and therefore so does the Lie bracket of $i^*_{\underline{M}}\pH(\SA_Q^{\psi})$. It follows that $\imath_x^*C$ vanishes, and we reached the contradiction. This proves that $\pH(\SA_Q^{\psi})$ is commutative.

We now prove that $\pH(\SA_{Q,W}^{T',\psi})$ is isomorphic to $\Sym_{\boxdot}(\BPS_{Q,W}^{T'}[-1]\otimes\rmH^*_{\BoC^*}(\pt))$ as a supercommutative algebra.

Again, since $\varphi_{\Tr(W)}$ is strict monoidal for $\boxdot$ and is perverse $t$-exact, it suffices to prove the statement for $W=0$. There is a natural morphism $\BPS_{Q}^{T'}[-1]\rightarrow \SA_{Q}^{T',\psi}$. By using the $\rmH^*_{\BoC^*}(\pt)$-action on $\SA_Q^{T',\psi}$ induced by the choice of a determinant line bundle, we obtain the morphism $\BPS_{Q}^{T'}[-1]\otimes\rmH^*_{\BoC^*}(\pt)\rightarrow \SA_{Q}^{T',\psi}$. After taking the perverse degeneration, this becomes a morphism of Lie algebra objects in $\Perv^{\BoZ}(\CM_Q^{T'})$. The Lie brackets on both sides vanish. By the universal property of the universal enveloping algebra, we obtain a morphism of algebras $\Sym_{\boxdot}(\BPS_Q^{T'}[-1]\otimes\rmH^*_{\BoC^*}(\pt))\rightarrow\pH(\SA_Q^{T',\psi})$. To prove it is an isomorphism, it suffices to prove it after applying the (conservative) forgetful functor $\Perv^{\BoZ}(\CM_Q^{T'})\rightarrow\Perv^{\BoZ}(\CM_Q)$. Then, it is a consequence of the cohomological integrality isomorphism \cite[Theorem~C]{davison2020cohomological}.
\end{proof}

 The commutator Lie bracket on $\SA_{Q,W}^{T',\psi}$ induced by the associative multiplication $\frakm^{\psi}$ induces a Lie bracket on $\BPS_{Q,W}^{T'}[-1]$, which by taking derived global sections, gives a Lie bracket on $\rmBPS_{Q,W}^{T'}:= \mathrm{H}^{*}(\BPS_{Q,W}^{T'}[-1])$. The subspace $\rmBPS_{Q,W}^{T'}$ with this Lie bracket is called \emph{BPS Lie algebra}. We may also use the more convenient notation $\frakn_{Q,W}^{T',\rmBPS,+}\coloneqq\rmBPS_{Q,W}^{T'}$ when convenient.

\begin{corollary}
\label{corollary:equivariant_pwb_supercommutativity}
We assume that $\rmBPS_{Q,W}^{T'}$ is a free $\rmH^*_{T'}(\pt)$-module or that $\rmH^*(\rmBPS_{Q,W})$ carries a pure mixed Hodge structure. Let $\FP$ be the perverse filtration on the absolute cohomological Hall algebra $\CA_{Q,W}^{T',\psi}=\rmH^*(\SA_{Q,W}^{T',\psi})$. Then, there is an isomorphism of algebras $\Sym_{\rmH^*_{T'}(\pt)}(\rmBPS_{Q,W}^{T'}[u])\cong\Gr^{\FP}(\rmH^*(\SA_{Q,W}^{T',\psi}))=\Gr^{\FP}(\CA_{Q,W}^{T',\psi})$ where $\BPS_{Q,W}^{T'}[u]=\BPS_{Q,W}^{T'}\otimes_{\BoQ}\BoQ[u]$ under an identification $\rmH^*_{\BoC^*}(\pt)\cong\BoQ[u]$.
\end{corollary}
\begin{proof}
The corollary follows by taking the derived global sections of \Cref{theorem:asso_graded_perverse_filt}. The freeness or purity of $\rmBPS_{Q,W}^{T'}$ over $\rmH^*_{T'}(\pt)$ is used to identify $\Sym_{\rmH^*_{T'}(\pt)}(\rmBPS_{Q,W}^{T'}\otimes\rmH^*_{\BoC^*(\pt)})$ with $\rmH^*(\Sym_{\boxdot}(\BPS_{Q,W}^{T'})\otimes\rmH^*_{\BoC^*}(\pt))$ via the equivariant K\"unneth formula.
\end{proof}

\Cref{corollary:equivariant_pwb_supercommutativity} applies for the triple quiver with the canonical cubic potential, as the BPS Lie algebra is known to carry a pure mixed Hodge structure, by \cite{davison2016integrality}.

\section{The cohomological Hall algebra of preprojective algebras}

\subsection{Preprojective algebras and moduli stacks}
\label{subsection:preprojective_algebra_stacks}

Let $Q=(Q_0,Q_1)$ be a quiver. We denote by $\overline{Q}=(Q_0,\overline{Q}_1=Q_1\sqcup Q_1^{\op})$ the double quiver. It has the same set of vertices as $Q$ and $Q_1^{\op}$ is the set of opposite arrows, that is the set of $\alpha^*\colon j\rightarrow i$ for $\alpha\colon i\rightarrow j$. We denote by $\rho\coloneqq\sum_{\alpha\in Q_1}[\alpha,\alpha^*]\in\BoC[\overline{Q}]$ the preprojective relation. The \emph{preprojective algebra} of $Q$ is the quotient $\BoC[\overline{Q}]/\langle\rho\rangle$ of the path algebra of $\overline{Q}$ by the two-sided ideal generated by the preprojective relation.

We denote by $\FM_{\Pi_Q}$ the stack of representations of the preprojective algebra. Its set of connected components is $\BoN^{Q_0}$ and we have a decomposition $\FM_{\Pi_Q}=\bigsqcup_{\dd\in\BoN^{Q_0}}\FM_{\Pi_Q,\dd}$ where for $\dd\in\BoN^{Q_0}$, $\FM_{\Pi_Q,\dd}$ is the stack of representations of $\Pi_Q$ with dimension vector $\dd$. It may be realized as the stacky Hamiltonian reduction of the stack of representations of $\overline{Q}$. Namely, the moment map induces a $\GL_{\dd}$-equivariant morphism $\mu_{\dd}\colon X_{\overline{Q},\dd}\rightarrow\mathfrak{gl}_{\dd}$, $(x_{\alpha},x_{\alpha^*})_{\alpha\in Q_1}\mapsto \sum_{\alpha\in Q_1}[x_{\alpha},x_{\alpha^*}]$, where $\GL_{\dd}$ acts on $\mathfrak{gl}_{\dd}$ by conjugation. Then, $\FM_{\Pi_Q,\dd}\simeq \mu_{\dd}^{-1}(0)/\GL_{\dd}$. We also define the stack of exact sequences of representations of the preprojective algebra $\mathfrak{Exact}_{\Pi_Q}=\bigsqcup_{\dd,\dd'\in\BoN^{Q_0}}\mathfrak{Exact}_{\Pi_Q,\dd,\dd'}$ where $\mathfrak{Exact}_{\Pi_Q,\dd,\dd'}=(X_{\overline{Q},\dd,\dd'}\cap\mu_{\dd+\dd'}^{-1}(0))/\GL_{\dd,\dd'}$ and $\GL_{\dd,\dd'}\subset\GL_{\dd+\dd'}$ is the parabolic subgroup of invertible matrices preserving a $\dd$-dimensional subspace. There are morphisms $p_{\dd,\dd'}\colon\mathfrak{Exact}_{\Pi_Q,\dd,\dd'}\rightarrow \FM_{\Pi_Q,\dd+\dd'}$ and $q_{\dd,\dd'}\colon\mathfrak{Exact}_{\Pi_Q,\dd,\dd'}\rightarrow\FM_{\Pi_Q,\dd}\times\FM_{\Pi_Q,\dd'}$ respectively sending a short exact sequence to its middle term and and its extreme terms.

We denote by $\CM_{\Pi_Q}=\bigsqcup_{\dd\in\BoN^{Q_0}}\CM_{\Pi_Q,\dd}$ the good moduli space of $\FM_{\Pi_Q}$. It can be realized as an affine GIT quotient $\CM_{\Pi_Q,\dd}\cong \mu_{\dd}^{-1}(0)\cms \GL_{\dd}\coloneqq\Spec(\BoC[\mu_{\dd}^{-1}(0)]^{\GL_{\dd}})$. It parametrizes semisimple $\dd$-dimensional $\Pi_Q$-representations. We denote by $\JH_{\Pi_Q,\dd}\colon \FM_{\Pi_Q,\dd}\rightarrow\CM_{\Pi_Q,\dd}$ the good moduli space morphism (also known as \emph{semisimplification morphism} or \emph{Jordan--H\"older} morphism in this case).

If $T$ is an auxiliary torus acting on the arrows of $\overline{Q}$, we say that the preprojective relation $\rho$ is \emph{homogeneous} if for any vertex $i\in Q_0$, $e_i\rho e_i$ is homogeneous of some weight, where $e_i\in\BoC[\overline{Q}]$ is the idempotent at the $i$th vertex. If we denote by $w_{\alpha}\in\Hom(T',\BoG_{\mathrm{m}})$ the weight of the $T'$-action on the arrow $\alpha\in \overline{Q}_1$, this means that for $\alpha\in Q_1$, the weight $w_{\alpha}w_{\alpha^*}$ of $\alpha\alpha^*$ only depends on $t(\alpha)$.

If $T'$ is an auxiliary torus acting on the arrows on $\overline{Q}$ so that the preprojective relation $\rho$ is homogeneous, there is an induced action of $T'$ on $\FM_{\Pi_Q}$ and on $\CM_{\Pi_Q}$. We denote by $\FM_{\Pi_Q}^{T'}$ and $\CM_{\Pi_Q}^{T'}$ the corresponding quotient stacks. The Jordan--H\"older morphism $\JH_{\Pi_Q}$ is $T'$-equivariant, thus inducing a morphism $\JH_{\Pi_Q}^{T'}\colon \FM_{\Pi_Q}^{T'}\rightarrow\CM_{\Pi_Q}^{T'}$.

\subsection{Determinant line bundles}
\label{subsection:determinant_line_bundle_preproj}
We define a determinant line bundle over $\FM_{\Pi_Q}$ exactly as in \Cref{subsection:determinant_line_bundle_Q} by replacing $\FM_Q$ by $\FM_{\Pi_Q}$ and $\mathfrak{Exact}_Q$ by $\mathfrak{Exact}_{\Pi_Q}$. Namely, this is a line bundle $\CL$ over $\FM_{\Pi_Q}$ such that in the correspondence $\FM_{\Pi_Q}\times\FM_{\Pi_Q}\xleftarrow{q}\mathfrak{Exact}_{\Pi_Q}\xrightarrow{p}\FM_{\Pi_Q}$, $q^*(\CL\boxtimes\CL)\cong p^*\CL$. We assume that $\CL$ is \emph{positive} in the sense that for any $\BoC$-point $x\in\FM_{\Pi_Q}$, for the morphism $\mathrm{act}_x\colon\rmB\BoG_{\rmm}\rightarrow\FM_{\Pi_Q}$ given by the inclusion of scalar automorphism of $x$, $\mathrm{act}_x^*c_1(\CL)\neq 0$.

\begin{lemma}
\label{lemma:weight_det_lb_preproj}
 There exists a monoid homomorphism $\abs{-}\colon \BoN^{Q_0}\rightarrow\BoZ$ such that for any $\BoC$-point $x\in\FM_{\Pi_Q,\dd}$, the line bundle $\mathrm{act}_x^*\CL$ on $\rmB\BoG_{\rmm}$ is given by the weight $\abs{\dd}$-representation of $\BoG_{\rmm}$.
\end{lemma}
\begin{proof}
We write $\FM_{\Pi_Q,\dd}\simeq \mu_{\dd}^{-1}(0)/\GL_{\dd}$. There is a morphism of stacks $u_{\dd}\colon\mu_{\dd}^{-1}(0)\times\rmB\BoG_{\rmm}\rightarrow\FM_{\Pi_Q,\dd}$. The pullback $u_{\dd}^*\CL$ is a line bundle over $\mu_{\dd}^{-1}(0)\times\rmB\BoG_{\rmm}$. It is equivalent to the datum of a $\BoZ$-weight on a line bundle over $\mu_{\dd}^{-1}(0)$. We denote by $\abs{\dd}$ this weight. The compatibility of $\CL$ with short exact sequences implies that $\abs{-}\colon\BoN^{Q_0}\rightarrow\BoZ$ is a monoid homomorphism. This concludes.
\end{proof}

\subsection{Dimensional reduction}
\label{subsection:dimensionalreduction}

Recall the triple quiver with potential $(\widetilde{Q},\widetilde{W})$ associated to a quiver $Q$ from Section \ref{subsection:quivers-with-potentials}. Let $\mathrm{Jac}(\widetilde{Q},\widetilde{W})$ be the Jacobian algebra associated to the quiver with potential $(\widetilde{Q},\widetilde{W})$. There is an isomorphism of algebras $\mathrm{Jac}(\widetilde{Q},\widetilde{W}) \cong \Pi_Q[\omega]$, where $\Pi_Q$ is the preprojective algebra of $Q$ and $\omega=\sum_{i\in Q_0}\omega_i$ is a central element: the partial derivatives with respect to the $\omega_i$ impose the preprojcetive relation while the partial derivatives with respect to the $\alpha$ and $\alpha^*$ arrows impose respectively the commutation of $\omega$ with $\alpha^*$ and $\alpha$ respectively.

We have the following dimensional reduction isomorphism  due to Davison \cite{davison2017critical} relating the CoHA of the preprojective algebra and the of triple quiver with its canonical cubic potential.  Let \[ \pi: \mathfrak{M}^{T'}_{\tilde{Q},\bd} \rightarrow \mathfrak{M}^{T'}_{\overline{Q},\bd}\] be the natural forgetful morphism, where $\overline{Q}$ is the doubled quiver. We have a product decomposition
\begin{equation}\mathfrak{M}^{T'}_{\tilde{Q},\bd} \simeq \mathfrak{M}^{T'}_{\overline{Q},\bd}\times_{\mathrm{B}\GL_{\dd}} \mathfrak{M}^{T'}_{L, \bd} \end{equation} where $L$ is the quiver with vertices $Q_0$ and loops $\omega_i$ at each vertex $i$. We let $\BoQ_{\FM_{\Pi_Q,\dd}^{T'}}^{\vir}=\BoQ_{\FM_{\Pi_Q,\dd}^{T'}}[-(\dd,\dd)_Q]$ and $\BoQ_{\FM_{\tilde{Q},\dd}^{T'}}^{\vir}=\BoQ_{\FM_{\tilde{Q},\dd}^{T'}}[-\langle\dd,\dd\rangle_{\tilde{Q}}]$ is defined as in \Cref{subsection:cohas}. Then applying dimension reduction \cite{davison2017critical}, where we consider the $\mathbb{C}^{*}$ action scaling the loops, gives an isomorphism of complexes
\begin{equation} \label{equation:dimension_reduction_sheaf}
    \mathbb{D}\mathbf{Q}^{\vir}_{\mathfrak{M}^{T'}_{\Pi_Q,\bd}} \simeq  \pi_{*} \varphi_{\mathrm{Tr}(\tilde{W})} \mathbf{Q}^{\vir}_{\mathfrak{M}^{T'}_{\tilde{Q},\bd}}.\end{equation}

We have a commutative square
\[\begin{tikzcd}
	{\FM_{\tilde{Q}}^{T'}} & {\FM_{\overline{Q}}^{T'}} \\
	{\CM_{\tilde{Q}}^{T'}} & {\CM_{\overline{Q}}^{T'}}
	\arrow["\pi", from=1-1, to=1-2]
	\arrow["{\JH_{\tilde{Q}}^{T'}}"', from=1-1, to=2-1]
	\arrow["{\JH_{\overline{Q}}^{T'}}", from=1-2, to=2-2]
	\arrow["{\pi'}"', from=2-1, to=2-2]
\end{tikzcd}\]
where $\pi'$ is the morphism induced by $\pi$ between the good moduli spaces.

Taking cohomology gives an isomorphism of $\mathbf{N}^{Q_0} \times \mathbf{Z}$-graded $\rmH_{T'}(\pt)$-modules
\begin{equation} \label{equation: dimension_reduction_cohomology}
\textrm{DR}_{\bd}: \rmH^*(\mathfrak{M}_{\tilde{Q},\bd}^{T'}, \varphi_{\Tr(\tilde{W})}\BoQ_{\FM_{\tilde{Q},\dd}^{T'}}^{\vir}) \simeq \rmH^*(\mathfrak{M}_{\Pi_Q,\bd}^{T'},\BD\BoQ^{\vir}_{\FM_{\Pi_Q,\dd}^{T'}}) .
\end{equation}

\subsection{Cohomological Hall algebra structure}
Let $Q$ be a quiver. We consider an auxiliary torus $T'$ acting on the arrows of $\overline{Q}$ so that the preprojective relation is homogeneous. We let $\SA_{\Pi_Q}^{T'}\coloneqq (\JH_{\Pi_Q}^{T'})_{*}\BD\BoQ_{\FM_{\Pi_Q}^{T'}}^{\vir}\in\CD^+_{\rmc}(\CM_{\Pi_Q}^{T'})$. Using the closed immersion $\CM_{\Pi_Q}^{T'}\rightarrow\CM_{\overline{Q}}^{T'}$, which induces a fully faithful functor $\CD^+_{\rmc}(\CM_{\Pi_Q}^{T'})\rightarrow\CD^+_{\rmc}(\CM_{\overline{Q}})$, we may consider $\SA_{\Pi_Q}^{T'}$ as an object in $\CD^+_{\rmc}(\CM_{\overline{Q}}^{T'})$ when convenient.

There is a unique extension of the $T'$ action to the arrows of $\tilde{Q}$ so that the canonical cubic potential $\tilde{W}$ is $T'$-invariant.

We define the CoHA structure on $\SA_{\Pi_Q}^{T'}$ from the CoHA structure on $\SA_{\tilde{Q},\tilde{W}}^{T'}$ of the CoHA of the triple quiver with cubic potential via dimensional reduction. It is defined so that above dimensional reduction isomorphism \Cref{equation:dimension_reduction_sheaf} is an isomorphism of algebras: we have
\[ \SA_{\Pi_Q}^{T'} \simeq \pi'_*\SA^{T'}_{\tilde{Q},\tilde{W}}.\]
We will as usual twist the CoHA product on both sides by a bilinear form $\psi$ as in \Cref{subsection:cohas}. We denote by $\SA_{\Pi_Q}^{T',\psi}$ and $\SA_{\tilde{Q},\tilde{W}}^{T',\psi}$ the resulting twisted algebras.

By taking derived global sections, we obtain the cohomological Hall algebra structure on $\CA_{\Pi_Q}^{T'}\coloneqq \rmH^*(\SA_{\Pi_Q}^{T'})$. Taking the $\psi$-twist into account, we obtain the algebra $\CA_{\Pi_Q}^{T',\psi}$.

\subsection{Support Lemma and Less Perverse filtration}
\label{subsection:support_lessperverse}
We saw that the dimensional reduction morphism (\Cref{equation:dimension_reduction_sheaf})	 is an isomorphism of complexes. Combined with the integrality Theorem \ref{theorem:asso_graded_perverse_filt}, this gives the integrality theorem for the pushforward $(\JH_{\overline{Q}}^{T'})_{*} \mathbb{D}\mathbf{Q}^{\vir}_{\mathfrak{M}_{\mathrm{\Pi}_Q}^{T'}}$, where we may treat $\mathbb{D}\mathbf{Q}^{\vir}_{\mathfrak{M}_{\mathrm{\Pi}_Q}^{T}}$ as a complex on $\mathfrak{M}_{\overline{Q}}^{T'}$ via the closed immersion $\FM_{\Pi_Q}^{T'}\rightarrow\FM_{\overline{Q}}^{T'}$. By forgetting the additional loops $\omega_i$ at vertices, we obtain a commutative diagram
\[\begin{tikzcd}[ampersand replacement=\&]
	{\mathfrak{M}_{\tilde{Q}}^{T'}} \& {\mathfrak{M}_{\overline{Q}}^{T'}}\\
	{\mathcal{M}_{\tilde{Q}}^{T'}} \& {\mathcal{M}_{\overline{Q}}^{T'}}
	\arrow["\pi", from=1-1, to=1-2]
	\arrow["{\JH_{\tilde{Q}}^{T'}}"', from=1-1, to=2-1]
	\arrow["{\JH_{\overline{Q}}^{T'}}", from=1-2, to=2-2]
	\arrow["{\pi^{\prime}}"', from=2-1, to=2-2]
\end{tikzcd}\]

The dimensional reduction isomorphism together with the integrality gives an isomorphism of complexes

\begin{equation}
\label{equation:dim_red_coh_int}
(\JH_{\overline{Q}}^{T'})_{*} \mathbb{D}\mathbf{Q}^{\vir}_{\mathfrak{M}^{T'}_{\mathrm{\Pi}_Q}} \cong \mathrm{Sym}_{\boxdot}\left( \pi^{\prime}_{*}\mathcal{BPS}_{\tilde{Q},\tilde{W}}^{T'}[-1] \otimes \rmH_{\BoC^{*}}(\pt)\right)
\end{equation}

The sheaf $\pi^{\prime}_{*}\mathcal{BPS}_{\tilde{Q},\tilde{W}}^{T'}[-1]$ is perverse sheaf on $\mathcal{M}_{\overline{Q}}^{T'}$. This is a consequence of the following support lemma, proved in \cite{davison2016integrality}.

\begin{lemma}[Support lemma {\cite[Lemma 4.1]{davison2016integrality}}]\label{lemma:support_lemma_preprojective}

The support of the sheaf $\mathcal{BPS}_{\tilde{Q},\tilde{W}}$ is contained in the closed subvariety of $\mathcal{M}_{\tilde{Q}}$ consisting of representations where for each vertex $i \in Q_0$, the actions of all the loops $\omega_i$ are given by a single scalar multiple of the identity. In particular, there exists a semisimple perverse sheaf $\mathcal{BPS}_{\Pi_Q}$ on $\mathcal{M}_{\overline{Q}}$ such that \[\mathcal{BPS}_{\tilde{Q},\tilde{W}} \cong l_{*}(\mathcal{BPS}_{\Pi_Q} \boxtimes \mathbf{Q}^{\vir}_{\mathbb{A}^1}) \] where $l: \mathcal{M}_{\overline{Q}} \times \mathbb{A}^1 \rightarrow \mathcal{M}_{\tilde{Q}}$ is the closed embedding given by sending a representation $\rho$ of $\overline{Q}$ and a scalar $\lambda \in \mathbb{A}^1$ to the representation $\tilde{\rho}$ of $\tilde{Q}$ defined by $\tilde{\rho}(a) = \rho(a)$ for $a \in \overline{Q}_1$ and $\tilde{\rho}(\omega_i) = \lambda \cdot \mathrm{id}$ for each vertex $i \in Q_0$.
\end{lemma}

In the presence of an auxiliary torus $T'$, all complexes and perverse sheaves appearing in \Cref{lemma:support_lemma_preprojective} are $T'$-equivariant, and we let $\BPS_{\Pi_Q}^{T'}\in\Perv(\CM_{\Pi_Q}^{T'})$ be the corresponding perverse sheaf, which satisfies $\BPS_{\tilde{Q},\tilde{W}}^{T'}\cong l_*(\BPS_{\Pi_Q}^{T'}\boxtimes\BoQ_{\BoA^1}^{\vir})$.

By \Cref{equation:dim_red_coh_int} and \Cref{lemma:support_lemma_preprojective}, we have an isomorphism of complexes
\begin{align} \label{equation:integrality_preprojective}
(\JH_{\overline{Q}}^{T'})_{*}\mathbb{D}\mathbf{Q}^{\vir}_{\mathfrak{M}_{\Pi_Q}^{T'}} \simeq \Sym_{\boxdot}\left(\mathcal{BPS}_{\mathrm{\Pi}_Q}^{T'} \otimes \rmH_{\BoC^{*}}(\pt)\right)
\end{align}

Since $\mathcal{BPS}_{\Pi_Q}^{T'}$ is a semisimple $T'$-equivariant perverse sheaf, the complex of constructible sheaves \eqref{equation:integrality_preprojective} splits, which means that
\[ (\JH_{\overline{Q}}^{T'})_{*}\mathbb{D}\mathbf{Q}^{\vir}_{\mathfrak{M}_{\Pi_Q}^{T'}} \simeq \bigoplus_{i \in 2\mathbb{Z}_{\geq 0}} \bigoplus_{j \in S_i} \mathcal{IC}_{\overline{Z}_j}(\mathcal{L}_j[\dim(Z_j)])[-i].\]
where $Z_j \subset \mathcal{M}_{\Pi_Q,\bd}$ are $T'$-invariant locally closed irreducible smooth subvarieties of $\mathcal{M}_{\Pi_Q,\bd}$ and $\mathcal{L}_j$ are $T'$-equivariant simple local systems on $Z_j$.

For any $i$, applying the natural transformation ${}^{\mathfrak{p}}\!(\tau^{\leq i}) \rightarrow \id$, gives a split morphism
\[ {}^{\mathfrak{p}}\!\tau^{\leq i} ((\JH_{\overline{Q}}^{T'})_{*}\mathbb{D}\mathbf{Q}^{\vir}_{\mathfrak{M}_{\Pi_Q}^{T'}} ) \rightarrow (\JH_{\overline{Q}}^{T'})_{*}\mathbb{D}\mathbf{Q}^{\vir}_{\mathfrak{M}_{\Pi_Q}^{T'}}.\]
Therefore, the morphism
\[ \mathfrak{L}^{i} \mathcal{A}_{\mathrm{\Pi}_Q,\bd}^{T'} := \rmH^*( \mathcal{M}_{\mathrm{\Pi}_Q,\bd}^{T'}, {}^{\mathfrak{p}}\!\tau^{\leq i}(\JH^{T'}_{\overline{Q},\bd})_{*}\mathbb{D}\mathbf{Q}^{\vir}_{\mathfrak{M}_{\mathrm{\Pi}_Q,\bd}^{T'}} ) \rightarrow \mathcal{A}_{\mathrm{\Pi}_Q,\bd}^{T'}\]
induced by taking derived global sections is an injection, giving an increasing filtration, called the \textit{Less Perverse Filtration} which starts from degree 0, i.e we have

\[\mathfrak{L}^{\bullet}\mathcal{A}_{\mathrm{\Pi}_Q,\bd}^{T'} = (0 \subset  \mathfrak{L}^{0} \mathcal{A}_{\mathrm{\Pi}_Q,\bd}^{T'} \subset  \mathfrak{L}^{1} \mathcal{A}_{\mathrm{\Pi}_Q,\bd}^{T'}  \subset \cdots \mathcal{A}_{\mathrm{\Pi}_Q,\bd}^{T'}).\]

It is proved in \cite[Proposition 5.1]{davison2020bps} that the Less Perverse filtration is compatible with the CoHA multiplication on $\mathcal{A}_{\mathrm{\Pi}_Q}^{T'}$. This follows formally from the fact that the CoHA multiplication $\frakm\colon \SA_{\Pi_Q}^{T'}\boxdot\SA_{\Pi_Q}^{T'}\rightarrow\SA_{\Pi_Q}^{T'}$ at the sheaf level is a morphism of constructible sheaves and standard properties of t-structures. Thus, we have an induced graded algebra structure on the associated graded algebra \[\mathrm{Gr}^{\mathfrak{L}}(\mathcal{A}_{\mathrm{\Pi}_Q}^{T'}) := \bigoplus_{i \geq 0} \mathfrak{L}^{i} \mathcal{A}_{\mathrm{\Pi}_Q}^{T'}/ \mathfrak{L}^{i-1} \mathcal{A}_{\mathrm{\Pi}_Q}^{T'}.\]

The algebra structure on $\Gr^{\FL}(\CA_{\Pi_Q}^{T'})$ is denoted by $^{\frakp}\frakm$. The commutator Lie bracket induced by the algebra structure on $\Gr^{\FL}(\CA_{\Pi_Q}^{T'})$ is denoted by $^{\frakp}\![-,-]$.

The main goal of this paper is to precisely describe the algebra $\mathrm{Gr}^{\mathfrak{L}}(\mathcal{A}_{\mathrm{\Pi}_Q}^{T'})$ in terms of the BPS Lie algebra of $\Pi_Q$.

\section{Coproduct on CoHAs and affinized BPS Lie algebras}

In this section we define a coproduct structure on the CoHA of the triple quiver with the canonical cubic potential, and use it to define the affinized BPS Lie algebra (\Cref{subsection:affinized_BPS_Lie_algebra}). Note that CoHAs of arbitrary quivers with potentials carry the \emph{localized coproduct} of Davison \cite{davison2017critical} and the advantage of the triple quiver with its canonical potential is that as outlined in \cite{davison2022affine}, one can define an actual, non-localized, cocommutative coproduct, leading to the definition of the \emph{affinized BPS Lie algebra}.

\subsection{Coproduct for the CoHA of the triple quiver}

Let $Q$ be a quiver. Recall (\Cref{subsection:dimensionalreduction}) that there is an isomorphism of algebras $\Jac(\widetilde{Q},\widetilde{W}) \cong \Pi_Q[\omega]$ between the Jacobi algebra of the triple quiver with its canonical cubic potential and the polynomial algebra over the preprojective algebra.

Let $U \subset \BoC$ be any analytic open subset. Then, we define  $\mathfrak{M}^{U}_{\tilde{Q},\bd} \subset \mathfrak{M}_{\tilde{Q},\bd}$ to be the open (analytic) substack of representations, such that the generalized eigenvalues of $\omega_i$ for all $i \in Q_0$ are all inside $U$. We can similarly define the coarse moduli space $\mathcal{M}_{\tilde{Q},\bd}^{U}$ as an analytic open subvariety of $\CM_{\tilde{Q},\dd}$. Let $\mathfrak{M}^{U}_{\tilde{Q}} = \coprod \mathfrak{M}^{U}_{\tilde{Q},\bd}$ and $\mathcal{M}_{\tilde{Q}}^{U} = \coprod \mathcal{M}_{\tilde{Q},\bd}^{U}$. Given a short exact sequence of representations \[0 \rightarrow \rho_1 \rightarrow \rho_2 \rightarrow \rho_3 \rightarrow 0,\] the eigenvalues of $\rho_1(\omega)$ and $\rho_3(\omega)$ are in $U$ if and only if the eigenvalues of $\rho_2(\omega)$ are in $U$. This means that the full subcategory of representations of $\tilde{Q}$ such that the eigenvalues of the extra loops at vertices are in $U$ forms a Serre subcategory. From \cite[Section 3]{davison2020cohomological}, we have that for any open $U \subset \mathbf{C}$,  \[ \mathcal{A}^{U}_{\tilde{Q},\tilde{W}} := \bigoplus_{\bd \in \mathbf{N}^{Q_0}} \mathrm{H}^{*}(\mathfrak{M}^{U}_{\tilde{Q},\bd},(\varphi_{\mathrm{Tr}(\tilde{W})})_{\mid_{U}})\] has a structure of cohomological Hall algebra.
Given two open sets $U_1$ and $U_2$ such that $U_1 \subset U_2$, we have an open immersion $\mathfrak{M}^{U_1}_{\tilde{Q}} \hookrightarrow \mathfrak{M}^{U_2}_{\tilde{Q}}$ and thus the restriction morphism in cohomology gives the morphism \[ \rho_{U_2,U_1}: \mathcal{A}^{U_2}_{\tilde{Q},\tilde{W}} \rightarrow \mathcal{A}^{U_1}_{\tilde{Q},\tilde{W}}.\] Then using the functorial properties of vanishing cycles and proper base change, it is checked in \cite[Proposition 6.1.3]{jindalthesis} that this morphism is compatible with the CoHA structures on both sides, i.e., it is a morphism of algebras. When $U_2=\mathbf{C}$ is the full affine line, we denote this morphism by $\rho_{U_1}:= \rho_{\mathbf{C},U_1}$.

When the set $U$ is an open ball (in particular, $U$ is contractible) then as a consequence of the support lemma, it is proved in \cite[Proposition 6.1.4]{jindalthesis} that the restriction morphism is an isomorphism of algebras, i.e we have an isomorphism of algebras \[\rho_{U}: \mathcal{A}_{\tilde{Q},\tilde{W}} \cong \mathcal{A}^{U}_{\tilde{Q},\tilde{W}} .\]

Let $U_1$ and $U_2$ be two disjoint open subsets of $\mathbf{C}$. Then, we have a closed embedding of stacks \[\mathfrak{M}^{U_1}_{\tilde{Q}} \times \mathfrak{M}^{U_2}_{\tilde{Q}} \hookrightarrow \mathfrak{M}^{U_1 \cup U_2}_{\tilde{Q}}\] given by taking direct sum of representations. Suppose $\rho \in \mathfrak{M}^{U_1 \cup U_2}_{\Jac(\tilde{Q},\tilde{W})}\subset \mathfrak{M}^{U_1 \cup U_2}_{\tilde{Q}}$. Then $\rho$ is a representation of $\Pi_Q$ with endomorphism $\rho(\omega)$ such that eigenvalues of $\rho(\omega)$ are in $U_1 \cup U_2$. Since $\rho(\omega)$ commutes with the $\Pi_Q$-action, the generalized eigenspaces of $\rho(\omega)$ are $\Pi_Q$-submodules. Thus, we have a decomposition \[\rho \cong \bigoplus_{z \in U_1 \cup U_2} \rho_z,\] where $\rho_z$ is the generalized eigenspace of $\rho(\omega)$ with eigenvalue $z$. Since $U_1$ and $U_2$ are disjoint, we have that $\rho = \rho_{U_1} \oplus \rho_{U_2}$, where $\rho_{U_i} = \bigoplus_{z \in U_i} \rho_z$. This shows that the above closed embedding restricts to an equivalence of stacks \[\mathfrak{M}^{U_1}_{\Jac(\tilde{Q},\tilde{W})} \times \mathfrak{M}^{U_2}_{\Jac(\tilde{Q},\tilde{W})} \simeq \mathfrak{M}^{U_1 \cup U_2}_{\Jac(\tilde{Q},\tilde{W})}. \]
Using the Thom--Sebastiani isomorphism and deformed dimensional reduction theorem, it is proved in \cite[Proposition 6.1.5]{jindalthesis} that the above equivalence of stacks induces an isomorphism of algebras \[\mathcal{A}^{U_1 \cup U_2}_{\tilde{Q},\tilde{W}} \cong \mathcal{A}^{U_1}_{\tilde{Q},\tilde{W}} \otimes \mathcal{A}^{U_2}_{\tilde{Q},\tilde{W}}.\]

Now let $U_1$ and $U_2$ be any two disjoint open balls in $\mathbf{C}$. Then we have the following commutative diagram of algebra morphisms:
\[\begin{tikzcd}
	{\mathcal{A}_{\tilde{Q},\tilde{W}}} & {\mathcal{A}_{\tilde{Q},\tilde{W}} \otimes \mathcal{A}_{\tilde{Q},\tilde{W}}} \\
	{\mathcal{A}^{U_1 \cup U_2}_{\tilde{Q},\tilde{W}}} & {\mathcal{A}^{U_1}_{\tilde{Q},\tilde{W}} \otimes \mathcal{A}^{U_2}_{\tilde{Q},\tilde{W}}}
	\arrow["\Delta", dashed, from=1-1, to=1-2]
	\arrow["{\rho_{U_1 \cup U_2}}"', from=1-1, to=2-1]
	\arrow["{\simeq }"', from=2-1, to=2-2]
	\arrow["{\rho_{U_1}^{-1} \otimes \rho_{U_2}^{-1}}"', from=2-2, to=1-2]
\end{tikzcd}\]
defining the coproduct $\Delta$ as the unique morphism making the diagram commute. It is proved in \cite[Proposition 6.1.10, 6.1.11]{jindalthesis} that this coproduct is coassociative and independent of the choice of open balls $U_1$ and $U_2$. Intuitively, since $U_1$ and $U_2$ can be exchanged via a homotopy, one can expect the coproduct to be cocommutative. It is indeed proved in \cite[Proposition 6.1.12]{jindalthesis} that the coproduct $\Delta$ is cocommutative up to a sign, i.e we have
\[\Delta_{\bd,\mathbf{e}} = (-1)^{\sum_{a \in Q_1} \bd_{s(a)}\mathbf{e}_{t(a)}+ \mathbf{e}_{s(a)}\bd_{t(a)}} \sigma(\Delta_{\mathbf{e},\bd}) .\] where $\sigma \colon \mathcal{A}_{\tilde{Q},\tilde{W}} \otimes \mathcal{A}_{\tilde{Q},\tilde{W}} \rightarrow \mathcal{A}_{\tilde{Q},\tilde{W}} \otimes \mathcal{A}_{\tilde{Q},\tilde{W}}$ is the braiding morphism defined by $\sigma(a \otimes b) = b \otimes a$. We define the \emph{twisted coproduct} \[\Delta^{\chi}: \mathcal{A}_{\tilde{Q},\tilde{W}} \to \mathcal{A}_{\tilde{Q},\tilde{W}} \otimes \mathcal{A}_{\tilde{Q},\tilde{W}} \] by $\Delta^{\chi}_{\bd,\mathbf{e}} = (-1)^{\chi(\bd,\mathbf{e})} \Delta_{\bd,\mathbf{e}}$ where $\chi$ is the euler form of the quiver $Q$. It is immediate to check that the twisted coproduct $\Delta^{\chi}$ is coassociative and cocommutative. Furthemore, since $\Delta$ is compatible with the CoHA multiplication then as checked in \cite[Proposition 6.1.13]{jindalthesis}, $\Delta^{\chi}$ is compatible with the twisted CoHA multiplication defined in \Cref{equation:equality_psi} for the particular case $\psi=\chi$. Let us denote the twisted CoHA by $\mathcal{A}_{\tilde{Q},\tilde{W}}^{\psi}$.

\subsection{Affinized BPS Lie algebra} \label{subsection:affinized_BPS_Lie_algebra}

Let $\mathcal{P}$ be the set of primitive elements of the cocommutative coalgebra $(\mathcal{A}_{\tilde{Q},\tilde{W}}^{\psi},\Delta^{\chi})$, i.e., \[\mathcal{P} := \{ x \in \mathcal{A}_{\tilde{Q},\tilde{W}}^{\psi} \mid \Delta^{\chi}(x) = x \otimes 1 + 1 \otimes x \}.\] Since $\Delta^{\chi}$ is compatible with the CoHA multiplication on $\mathcal{A}_{\tilde{Q},\tilde{W}}^{\psi}$, it follows that $\mathcal{P}$ is closed under the commutator Lie bracket $[a,b] = ab - ba$. Thus, $\mathcal{P}$ has a structure of Lie algebra. We call this Lie algebra \emph{the affinized BPS Lie algebra} associated to the quiver $Q$ and denote it by $\mathfrak{g}_{\tilde{Q},\tilde{W}}^{\mathrm{aff}}$.

It follows from the support lemma (\Cref{lemma:support_lemma_preprojective}) that $\mathfrak{n}_{\Pi_Q}^{+,\rmBPS} \otimes \rmH^{*}_{\BoC^{*}}(\mathrm{\pt})\subset \mathfrak{g}_{\tilde{Q},\tilde{W}}^{\mathrm{aff}}$, where $\mathfrak{n}_{\Pi_Q}^{+,\rmBPS}$ is the BPS Lie algebra defined in \cite{davison2020bps}. By the Milnor--Moore theorem, we have injections of algebras \[\bfU(\mathfrak{n}_{\Pi_Q}^{+,\rmBPS} \otimes \rmH_{\BoC^{*}}(\mathrm{\pt})) \subset \bfU(\mathfrak{g}_{\tilde{Q},\tilde{W}}^{\mathrm{aff}}) \hookrightarrow \mathcal{A}_{\tilde{Q},\tilde{W}}^{\psi}.\] By the integrality theorem, this composition is an isomorphism of algebras and hence we have a canonical isomorphism of graded vector spaces \[ \mathfrak{n}_{\Pi_Q}^{+,\rmBPS} \otimes \rmH_{\BoC^{*}}(\mathrm{\pt}) \cong \mathfrak{g}^{\mathrm{aff}}_{\tilde{Q},\tilde{W}} \] and an isomorphism of algebras \[ \bfU(\mathfrak{n}_{\Pi_Q}^{+,\rmBPS} \otimes \rmH_{\BoC^{*}}(\mathrm{\pt})) \cong \mathcal{A}_{\tilde{Q},\tilde{W}}^{\psi}.\]
Note that the Lie bracket on $\mathfrak{n}_{\Pi_Q}^{+,\rmBPS} \otimes \rmH_{\BoC^{*}}(\mathrm{\pt})$ is not usually the trivial $\rmH^*_{\BoC^*}(\pt)$-linear extension of the Lie bracket on $\frakn_{\Pi_Q}^+$ \cite{davison2022affine,jindal2024coha}. It is not yet known what the actual Lie bracket on the affinized BPS Lie algebra is. See however \Cref{proposition:trivial_extension_real_subquiver} for a statement regarding the degree $0$ cohomological piece. In this paper, we determine the less perverse degeneration of this Lie bracket.

\section{Less perverse degeneration of the affinized BPS Lie algebra}
\label{section:less_perverse}
In order to compute the less perversely degenerated affinized BPS Lie algebra, we recall the definitions of raising and lowering operators on the cohomological Hall algebras of quivers with potentials and preprojective algebras.

\subsection{Raising operators}
 Let $(Q,W)$ be a quiver with potential, $\CL$ a determinant line bundle over $\FM_Q$, and $u\coloneqq c_1(\CL)\in \rmH^*(\FM_Q)$ the first Chern class of $\CL$. We consider the morphism $l\coloneqq\Det\times\id\colon\FM_Q\rightarrow\mathrm{B}\BoG_{\mathrm{m}}\times\FM_Q$. The function $\Tr(W)$ on the source coincides with the pullback of the function $0\boxplus\Tr(W)$ on the target, giving a natural pullback morphism in vanishing cycle cohomology
 \[
  l^*\colon \rmH^*(\mathrm{B}\BoG_{\mathrm{m}})\otimes\rmH^*(\FM_Q,\varphi_{\Tr(W)})\rightarrow\rmH^*(\FM_Q,\varphi_{\Tr(W)})\,,
 \]
inducing an action of $\rmH^*(\mathrm{B}\BoG_{\mathrm{m}})\cong\BoQ[u]$ on $\CA_{Q,W}^{(\psi)}$.

The following is proven in \cite[Proposition~4.1]{davison2022affine}.

\begin{proposition}
Let $(Q,W)$ be a quiver with potential. The morphism $u\bullet-\coloneqq l^*(u\otimes -)\colon \CA_{Q,W}^{(\psi)}\rightarrow\CA_{Q,W}^{(\psi)}$ is a derivation.
\end{proposition}

When we consider the triple quiver with its canonical cubic potential $(\tilde{Q},\tilde{W})$, we obtain a derivation of $\CA_{\Pi_Q}^{(\psi)}$. It may be defined directly in terms of $\FM_{\Pi_Q}$ by adapting the definition, replacing $l\colon\FM_Q\rightarrow\rmB\BoG_{\rmm} \times \FM_Q$ by $l'\coloneqq\Det\times\id\colon \FM_{\Pi_Q}\rightarrow\rmB\BoG_{\rmm}\times\FM_{\Pi_Q}$ and using that $l'^*(\BoQ_{\rmB\BoG_{\rmm}}\boxtimes\BD\BoQ_{\FM_{\Pi_Q}}^{\vir})\cong\BD\BoQ_{\FM_{\Pi_Q}}^{\vir}$. Here, $\Det\colon\FM_{\Pi_Q}\rightarrow\rmB\BoG_{\rmm}$ is the morphism describing the determinant line bundle $\CL$.

\subsection{Lowering operators}
\label{subsection:lowering_operators}
Let $(Q,W)$ be a quiver with potential. We define a morphism of stacks
\[
 A\colon \mathrm{B}\BoG_{\mathrm{m}}\times\FM_Q\rightarrow\FM_Q
\]
given by $(\CL',\CF)\mapsto \CL'\otimes \CF$ at the level of $X$-points for $X$ a scheme, where $\CL'$ is a line bundle over $X$ and $\CF$ is a flat family of $Q$-modules over $X$. This makes the stack $\FM_Q$ a module over $\mathrm{B}\BoG_{\mathrm{m}}$. The pullback in cohomology
\[
 A^*\colon\rmH^*(\FM_Q,\varphi_{\Tr(W)})\rightarrow \rmH^*(\mathrm{B}\BoG_{\mathrm{m}})\otimes \rmH^*(\FM_Q,\varphi_{\Tr(W)})
\]
makes $\CA_{Q,W}^{(\psi)}$ a comodule for the coalgebra $\rmH^*(\mathrm{B}\BoG_{\mathrm{m}})$. We write for $\alpha\in \CA_{Q,W}^{(\psi)}$, $A^*(\alpha)=\sum_{n\geq 0}u^n\otimes\alpha_n$, defining the coefficient $\alpha_n\in\CA_{Q,W}^{(\psi)}$. The following is \cite[Proposition~4.2]{davison2022affine}.

\begin{proposition}
 The morphism
 \[
\begin{matrix}
  \partial_u&\colon&\CA_{Q,W}^{(\psi)}&\rightarrow&\CA_{Q,W}^{(\psi)}\\
  &&\alpha&\mapsto&\alpha_1
\end{matrix}
 \]
is a derivation of of $\CA_{Q,W}^{(\psi)}$.
\end{proposition}

Again, when $(\tilde{Q},\tilde{W})$ is the triple quiver with its canonical cubic potential, we obtain a derivation on $\CA_{\Pi_Q}^{(\psi)}$ which can be defined directly in terms of $\FM_{\Pi_Q}$.

\begin{lemma}
 \label{lemma:vanishing_delta_BPS_Lie_algebra}
Let $Q$ be a symmetric quiver and $W\in\BoC[Q]$ a potential. Then, $\partial_u(\rmBPS_{Q,W})=0$.
\end{lemma}
\begin{proof}
After pushing the vanishing cycle sheaf on $\FM_Q$ to the good moduli space, the pullback morphism $\BoQ_{\FM_Q}^{\vir}\rightarrow A_*(\BoQ_{\rmB\BoG_{\rmm}}\boxtimes\BoQ_{\FM_Q}^{\vir})$ gives a morphism of contructible complexes
\[
 \SA_{Q,W}^{(\psi)}\rightarrow \rmH^*(\mathrm{B}\BoG_{\mathrm{m}})\otimes \SA_{Q,W}^{(\psi)}.
\]
After taking derived global sections, this morphism gives the morphism $A^*$ in critical cohomology. By formal properties of t-structures, the morphism $A^*$ respects the respective perverse filtrations on $\CA_{Q,W}^{(\psi)}$ and $\rmH^*(\mathrm{B}\BoG_{\mathrm{m}})\otimes\CA_{Q,W}^{(\psi)}$. Identifying $\rmH^*(\mathrm{B}\BoG_{\mathrm{m}})$ with $\BoQ[u]$, $u^n$ sits in perverse degree $2n$. Therefore, $A^*(\rmBPS_{Q,W})\subset \FP^{\leq 1}(\rmH^*(\mathrm{B}\BoG_{\mathrm{m}})\otimes\CA_{Q,W}^{(\psi)})=\rmBPS_{Q,W}$. In particular, for $x\in \rmBPS_{Q,W}\subset\CA_{Q,W}^{(\psi)}$, $A^*(x)=x_0$ has no terms with positive $u$-degree. This implies that $\partial_u(x)=0$.

\end{proof}

\subsection{Heisenberg algebra action}

We define the Heisenberg Lie algebra $\Heis$ as the Lie algebra over $\BoQ$ with basis $\{p,q,c\}$ and Lie bracket given by $[q,p]=c$, $[c,p]=[c,q]=0$.

We let $Q$ be a quiver and $\CL$ a determinant line bundle over $\FM_Q$. Since $\FM_Q\simeq X_{Q,\dd}/\GL_{\dd}$ and $X_{Q,\dd}$ is an affine scheme, the restriction $\CL_{\dd}$ of $\CL$ to $\FM_{Q,\dd}$ is given by a one-dimensional representation of $\GL_{\dd}$, that is by some representation $\GL_{\dd}\rightarrow\rmB\BoG_{\rmm}$, $(g_i)_{i\in Q_0}\mapsto \prod_{i\in Q_0}g_i^{r_{\dd,i}}$ for some integers $r_{\dd,i}\in\BoZ$. The compatibility with exact sequences (\Cref{subsection:determinant_line_bundle_Q}) ensures that $r_{\dd,i}$ does not depend on $\dd$ and also by the non-degeneracy of $\CL$ (\Cref{subsection:determinant_line_bundle_Q}) that either $r_{i}>0$ for all $i$, or $r_{i}<0$ for all $i$. For $\dd\in\BoN^{Q_0}$, we let $\abs{\dd}\coloneqq\sum_{i\in Q_0}\dd_i r_{i}$.

\begin{proposition}[{\cite[Proposition~4.3]{davison2022affine}}]
 \label{proposition:heisenberg_algebra_action}
 For any $\dd\in\BoN^{Q_0}$, there is an action of $\Heis$ on $\CA_{Q,W}^{(\psi)}$ where $p$ acts via $u\bullet-$ and $q$ by $\partial_u$. The central charge of the $\Heis$-action on $\CA_{Q,W,\dd}^{(\psi)}$ is $\abs{\dd}$. 
\end{proposition}

We have the following lemma that generalizes \Cref{lemma:vanishing_delta_BPS_Lie_algebra}.
\begin{lemma} \label{lemma: heisqlessperversecompatibility}
The action of $\partial_u$ decreases the Less perverse degree by $2$ and action of $u \bullet -$ increases the Less perverse degree by $2$, i.e we have

\[ \partial_u (\mathfrak{L}^{i}(\mathcal{A}_{\Pi_Q,\dd})) \subset \mathfrak{L}^{i-2}(\mathcal{A}_{\Pi_Q,\dd}) \] and 
\[ u \bullet (\mathfrak{L}^{i}(\mathcal{A}_{\Pi_Q,\dd})) \subset \mathfrak{L}^{i+2}(\mathcal{A}_{\Pi_Q,\dd}) \,. \]
\end{lemma}

\begin{proof}
The morphism $A^{*}$ in \Cref{subsection:lowering_operators} is defined by taking derived global sections on the morphism of constructible complexes
\[ a\colon(\JH_{\tilde{Q}})_{*}\varphi_{\mathrm{Tr}(W)} (\BoQ^{\mathrm{vir}}_{\mathfrak{M}_{\tilde{Q}}}) \rightarrow (\JH)_{*} A_{*}(\BQ_{\mathrm{B}\BoG_{\mathrm{m}}} \boxtimes \varphi_{\mathrm{Tr}(W)}(\mathbb{Q}^{\mathrm{vir}}_{\mathfrak{M}_{\tilde{Q}}})). \]

Since the morphism $A$ commutes with the semisimplification, we have a commutative diagram

\[\begin{tikzcd}
	{\mathrm{B}\BoG_{\mathrm{m}} \times \mathfrak{M}_{\tilde{Q}}} & {\mathfrak{M}_{\tilde{Q}}} \\
	{\mathfrak{M}_{\tilde{Q}}} & {\mathcal{M}_{\tilde{Q}}} \\
	{\mathfrak{M}_{\overline{Q}}} & {\mathcal{M}_{\overline{Q}}}
	\arrow["A", from=1-1, to=1-2]
	\arrow["{{\pi_{2}}}", from=1-1, to=2-1]
	\arrow["{{\JH_{\tilde{Q}}}}", from=1-2, to=2-2]
	\arrow["{{\JH}_{\tilde{Q}}}", from=2-1, to=2-2]
	\arrow["\pi", from=2-1, to=3-1]
	\arrow["{\pi^{\prime}}", from=2-2, to=3-2]
	\arrow["{\JH_{\overline{Q}}}", from=3-1, to=3-2]\,.
\end{tikzcd}\]

Thus we have 
\[ (\JH_{\tilde{Q}})_{*} A_{*}(\BQ_{\mathrm{B}\BoG_{\mathrm{m}}} \boxtimes \varphi_{\mathrm{Tr}(W)}(\mathbb{Q}^{\mathrm{vir}}_{\mathfrak{M}_{\tilde{Q}}})) \cong \rmH^*(\mathrm{B}\BoG_{\mathrm{m}},\mathbb{Q}) \otimes (\JH_{\tilde{Q}})_{*}\varphi_{\mathrm{Tr}(W)} (\BQ^{\mathrm{vir}}_{\mathfrak{M}_{\tilde{Q}}}) \]

Pushing forward by the forgetful morphism $\pi^{\prime}$ and using dimensional reduction (\Cref{equation:dimension_reduction_sheaf}), we get an isomorphism of complexes
\[ \pi^{\prime}_{*}(\JH_{\tilde{Q}})_{*} A_{*}(\BQ_{\mathrm{B}\BoG_{\mathrm{m}}} \boxtimes \varphi_{\mathrm{Tr}(W)}(\mathbb{Q}^{\mathrm{vir}}_{\mathfrak{M}_{\tilde{Q}}})) \cong \rmH^*(\mathrm{B}\BoG_{\mathrm{m}},\mathbb{Q}) \otimes  (\JH_{\overline{Q}})_{*}\mathbb{D}\mathbf{Q}^{\vir}_{\mathfrak{M}_{\Pi_Q,\bd}} \]

Thus, the morphism \[ \ptau^{\leq i} \pi^{\prime}_{*} a \colon \ptau^{\leq i} (\JH_{\overline{Q}})_{*}\mathbb{D}\mathbf{Q}^{\vir}_{\mathfrak{M}_{\Pi_Q,\bd}}  \rightarrow \ptau^{\leq i} \pi^{\prime}_{*}( (\JH_{\overline{Q}})_{*} A_{*}(\BQ_{\mathrm{B}\BoG_{\mathrm{m}}} \boxtimes \varphi_{\mathrm{Tr}(W)}(\mathbb{Q}^{\mathrm{vir}}_{\mathfrak{M}_{\tilde{Q}}}))) \] factors through

\[ \oplus_{j \geq 0} \rmH^{2j}(\mathrm{B}\BoG_{\mathrm{m}},\mathbb{Q}) \otimes \tau^{\leq i-2j} \JH_{*}\mathbb{D}\mathbf{Q}^{\vir}_{\mathfrak{M}_{\Pi_Q}} \,. \]	

Taking derived sections, the first claim follows. The second claim is similar and is proved in \cite[Proposition 5.7]{botta2023okounkov}.

\end{proof}

\begin{lemma}
 \label{lemma:injectivity}
 The action of $\partial_u$ on $\frakn_{\Pi_Q}^{+,\rmBPS}\otimes_{\BoQ}(u\BoQ[u])$ is injective. Moreover, for $\dd\in\BoN^{Q_0}$, $x\in\rmBPS_{Q,W,\dd}$ and $n\in\BoN$, $\partial_u(xu^n)=n\abs{\dd}xu^{n-1}$.
\end{lemma}
\begin{proof}
The derivation $\partial_u$ sends $\frakn_{\Pi_Q}^{+,\rmBPS}\otimes_{\BoQ}\BoQ[u]_{\mathrm{deg}\leq n}$ to $\frakn_{\Pi_Q}^{+,\rmBPS}\otimes_{\BoQ}\BoQ[u]_{\mathrm{deg}\leq n-1}$ for any $n\geq 1$ and therefore, for the injectivity, it suffices to show that for $x\in\frakn_{\Pi_Q,\dd}^+$ and $n\geq 1$, writing $\partial_u(xu^n)=\sum_{j=0}^{n-1}x'_ju^j$, the coefficient $x'_{n-1}$ is non-zero. One may indeed prove by induction on $n\geq 1$ that $\partial_u(xu^{n})=n\abs{\dd}xu^{n-1}$, which would conclude. For $n=1$, the formula may be written $\partial_u(xu)=\abs{\dd}x$, which follows from $\partial_u(xu)=[\partial_u,u]x+u\partial_u(x)=[\partial_u,u]x=\abs{\dd}x$, using \Cref{proposition:heisenberg_algebra_action}, since $\partial_u(x)=0$ (\Cref{lemma:vanishing_delta_BPS_Lie_algebra}). The formula for general $n$ follows by induction, \Cref{proposition:heisenberg_algebra_action} and
\[
 \partial_u(xu^n)=[\partial_u,u](xu^{n-1})+u\bullet \partial_u(xu^{n-1})\,.
\]

\end{proof}

\subsection{The less perverse associated graded of the preprojective CoHA}

\begin{theorem}
\label{theorem:main_degeneration_affinized_BPS}
Let $Q$ be a quiver and $\mathfrak{n}_{\Pi_Q}^{+,\rmBPS}$ the BPS Lie algebra for the preprojective algebra $\Pi_Q$ of $Q$. We let $\mathfrak{n}_{\Pi_Q}^{+,\rmBPS}[u]\coloneqq \mathfrak{n}_{\Pi_Q}^{+,\rmBPS}\otimes \rmH^*_{\BoC^*}(\pt)$ be the affinized Lie algebra where $u$ is the first Chern class of a determinant line bundle $\CL$ on $\FM_{\Pi_Q}$. Then, $\mathfrak{n}_{\Pi_Q}^{+,\rmBPS}[u]\subset\CA_{\Pi_Q}^{\psi}$ is a subspace that does not depend on the choice of $\CL$, it is stable under the less perversely degenerated Lie bracket ${^{\frakp}[-,-]}$ on $\rmH^*(\SA_{\Pi_Q}^{\psi})=\CA_{\Pi_Q}^{\psi}$ and for any $\dd,\ee\in\BoN^{Q_0}$, $x\in \mathfrak{n}_{\Pi_Q,\dd}^{+,\rmBPS}$, $y\in \mathfrak{n}_{\Pi_Q,\ee}^{+,\rmBPS}$, $m,n\in\BoN$, we have
\[
 ^{\frakp}[xu^m,yu^n]=\frac{\abs{\dd}^m\abs{\ee}^n}{\abs{\dd+\ee}^{m+n}}[x,y]u^{m+n}\,,
\]
where we denote by $[-,-]$ the Lie bracket on the BPS Lie algebra $\frakn_{\Pi_Q}^{+,\rmBPS}$.
\end{theorem}

\begin{proof}
The existence of the affinized BPS Lie algebra structure on $\mathfrak{n}^{+,\rmBPS}_{\Pi_Q}[u]$, that is that $\mathfrak{n}^{+,\rmBPS}_{\Pi_Q}[u]\subset\CA_{\Pi_Q}^{\psi}$ is realised as the space of primitive elements for a cocommutative coproduct on $\CA_{\Pi_Q}^{\chi}$ (Section \ref{subsection:affinized_BPS_Lie_algebra}). This implies that $\mathfrak{n}_{\Pi_Q}^{+,\rmBPS}[u]$ is stable under the Lie bracket on $\CA_{\Pi_Q}^{\psi}$ and therefore under the perversely degenerated Lie bracket ${^{\frakp}[-,-]}$. In particular, this also proves independence on the determinant line bundle, since the set $\mathfrak{n}^{+,\rmBPS}_{\Pi_Q}[u]$ is still primitive for the coproduct $\Delta^{\chi}$. Let $x,y\in \mathfrak{n}_{\Pi_Q}^{+,\rmBPS}$ and $m,n\in\BoN$. Then, $xu^m$ and $yu^n$ respectively sit in less perverse degrees $2m$ and $2n$. Therefore, the Lie bracket $[xu^m,yu^n]$ sits in cohomological degree $\leq 2(m+n)$ and by Lemma \ref{lemma: heisqlessperversecompatibility}, the perversely degenerated Lie bracket $^\frakp[xu^m,yu^n]$ sits in less perverse degree $\leq 2(m+n)$. This means that there exists $z\in \frakn_{\Pi_Q}^{+,\rmBPS}$ such that $^\frakp[xu^m,yu^n]=zu^{m+n}$. It remains to identify $z$ with $\frac{\abs{\dd}^m\abs{\ee}^n}{\abs{\dd+\ee}^{m+n}}[x,y]\in\frakn_{\Pi_Q}^{+,\rmBPS}$, by induction on $m+n$.

If $m=n=0$, then $z=[x,y]$ by definition of the BPS Lie algebra. If $m>0$ or $n>0$, then $m+n>0$ and $zu^{m+n}$ is determined by $\partial_u (zu^{m+n})$ by \Cref{lemma:injectivity}. Since $\partial_u$ acts as a derivation, $\partial_u(zu^{m+n})={^{\frakp}[\partial_ u(xu^m),yu^n]}+{^{\frakp}[xu^m,\partial_u(yu^n)]}$. We have by \Cref{lemma:injectivity}
\[
 \partial_u(zu^{m+n})=(m+n)\abs{\dd+\ee}zu^{m+n-1}\,,
\]
and by \Cref{lemma:injectivity} and induction on $m+n$,
\[
 {^{\frakp}[\partial_ u(xu^m),yu^n]}={^{\frakp}[m\abs{\dd}xu^{m-1},yu^n]}=m\abs{\dd}\frac{\abs{\dd}^{m-1}\abs{\ee}^n}{\abs{\dd+\ee}^{m+n-1}}[x,y]u^{m+n-1}
\]
and
\[
 {^{\frakp}[xu^m,\partial_u(yu^n)]}={^{\frakp}[xu^m,\abs{\ee}nyu^{n-1}]}=n\abs{\ee}\frac{\abs{\dd}^m\abs{\ee}^{n-1}}{\abs{\dd+\ee}^{m+n-1}}[x,y]u^{m+n-1}\,.
\]
By comparing the formulas, we obtain $z=\frac{\abs{\dd}^m\abs{\ee}^n}{\abs{\dd+\ee}^{m+n}}[x,y]$ as wanted.
\end{proof}

\begin{corollary}
\label{corollary:different_psi_twist}
Let $\psi$ be a bilinear form as in \Cref{subsection:cohas}. Then, the subspace $\frakn_{\Pi_Q}^{+,\rmBPS}\otimes\rmH^*_{\BoC^*}(\pt)\subset\CA_{\Pi_Q}^{\psi}$ is closed under the less perversely degenerated Lie bracket, which is given by the formula in \Cref{theorem:main_degeneration_affinized_BPS}.
\end{corollary}
\begin{proof}
As noted in \Cref{subsection:cohas}, the bilinear forms $\psi$ and the Euler form $\chi$ differ by a symmetric bilinear form, which we denote by $\phi$. Then, for any $\dd,\ee\in\BoN^{Q_0}$, $x\in\frakn_{\Pi_Q,\dd}^{+,\rmBPS}[u]$ and $y\in\frakn_{\Pi_Q,\ee}^{+,\rmBPS}[u]$, $[x,y]_{\psi}=(-1)^{\phi(\dd,\ee)}[x,y]_{\chi}$. This proves the stability of $\frakn_{\Pi_Q}^{+,\rmBPS}[u]$ under $^{\frakp}[-,-]_{\psi}$ and also the formula for the Lie bracket by using \Cref{theorem:main_degeneration_affinized_BPS}.
\end{proof}

\begin{corollary}
\label{corollary:coha=envelopingalgebra}
For $\dd\in\BoN^{Q_0}$, we set $u'_{\dd}=\frac{u}{\abs{\dd}}$. Then, the Lie bracket on $\frakn_{\Pi_Q}^{+,\rmBPS}[u']\coloneqq\bigoplus_{\dd\in\BoN^{Q_0}}\frakn_{\Pi_Q}^{+,\rmBPS}[u'_{\dd}]$ is the $\BoQ[u']$-linear extension of the Lie bracket on $\frakn_{\Pi_Q}^{+,\rmBPS}$.
\end{corollary}
\begin{proof}
This follows directly from \Cref{theorem:main_degeneration_affinized_BPS}.
\end{proof}

\begin{corollary} \label{corollary:PBW_affinized_BPS}
 In the situation of \Cref{theorem:main_degeneration_affinized_BPS}, there is an isomorphism of algebras
 \[
  \bfU(\frakn_{\Pi_Q}^{+,\rmBPS}[u])\rightarrow (\rmH^*(\SA_{\Pi_Q}^{\psi})=\CA_{\Pi_Q}^{\psi},{^{\frakp}\mathfrak{m}^{\psi}})=\Gr^{\FL}\CA_{\Pi_Q}^{\psi}
 \]
where ${^{\frakp}\mathfrak{m}^{\psi}}$ denotes the perversely degenerated CoHA product on $\rmH^*(\SA_{\Pi_Q}^{\psi})$ and $\frakn_{\Pi_Q}^{+,\rmBPS}[u]$ has the $\abs{-}$-twisted $\BoQ[u]$-linear extension of the Lie bracket on $\frakn_{\Pi_Q}^{+,\rmBPS}$, defined by the formula in \Cref{theorem:main_degeneration_affinized_BPS}.
\end{corollary}
\begin{proof}
By \Cref{theorem:main_degeneration_affinized_BPS} and \Cref{corollary:different_psi_twist}, the morphism $\frakn_{\Pi_Q}^{+,\rmBPS}[u]\rightarrow \Gr^{\FL}\CA_{\Pi_Q}^{\psi}$ is a morphism of Lie algebras where both sides have the less perversely degenerated Lie bracket $^\frakp[-,-]$. By the universal property of the enveloping algebra, we obtain a natural morphism of algebras $\bfU(\frakn_{\Pi_Q}^{+,\rmBPS}[u])\rightarrow \Gr^{\mathfrak{L}}\CA_{\Pi_Q}^{\psi}$. By the PBW isomorphism for enveloping algebras and the PBW isomorphism for the perversely degenerated CoHA, we obtain a commutative diagram in which the diagonal arrows are isomorphisms of vector spaces:
\[\begin{tikzcd}
	& {\Sym(\frakn_{\Pi_Q}^{+,\rmBPS}[u])} \\
	{\bfU(\frakn_{\Pi_Q}^{+,\rmBPS}[u])} & {} & {\Gr^{\FL}\CA_{\Pi_Q}^{\psi}}
	\arrow["{\mathrm{PBW}}"', from=1-2, to=2-1]
	\arrow["{\mathrm{PBW}}", from=1-2, to=2-3]
	\arrow[from=2-1, to=2-3]
\end{tikzcd}\]
Therefore, $\bfU(\frakn_{\Pi_Q}^{+,\rmBPS}[u])\rightarrow \Gr^{\FL}\CA_{\Pi_Q}^{(\psi)}$ is an isomorphism. This concludes.
\end{proof}

\subsection{The degree zero affinized BPS Lie algebras}
In this section, we determine the degree zero affinized BPS Lie algebra $\rmH^0(\BPS_{\Pi_Q})\otimes\rmH^*_{\BoC^*}(\pt)$ (without taking any degeneration). We show that it is a trivial extension of the positive part of the Kac--Moody algebra of the real subquiver of $Q$.

\begin{proposition}
\label{proposition:trivial_extension_real_subquiver}
 Let $Q$ be a quiver. Then, the degree zero affinized BPS Lie algebra $\mathfrak{n}_{Q^{\re}}^+[u]\coloneqq \mathfrak{n}^+_{Q^{\re}}\otimes\rmH^*_{\BoC^*}(\pt)\subset \mathfrak{n}_{\Pi_Q}^{+,\rmBPS}\otimes\rmH^*_{\BoC^*}(\pt)\subset\CA_{\Pi_Q}^{\psi}$ is stable under the Lie bracket and for any $x,y\in\mathfrak{n}_{Q^{\re}}^+$, $m,n\in\BoN$, $[xu^m,yu^n]=[x,y]u^{m+n}$, that is the Lie bracket on $\mathfrak{n}_{Q^{\re}}^+[u]$ is the trivial $\abs{-}$-twisted $\rmH^*_{\BoC^*}(\pt)$-linear extension of the Lie bracket on $\mathfrak{n}_{Q^{\re}}^+$.
\end{proposition}
\begin{proof}
Let $x\in\rmH^0(\BPS_{\Pi_Q,\dd})\cong\frakn_{Q^{\re},\dd}^{+}$, $y\in\rmH^0(\BPS_{\Pi_Q,\ee})\cong\frakn_{Q^{\re},\ee}^{+}$ and $m,n\in\BoN$. Then, $xu^m$ and $xu^n$ respectively sit in cohomological degrees $2m$ and $2n$. Therefore, the Lie bracket $[xu^m,yu^n]$ sits in cohomological degree $2(m+n)$. Additionally, by \Cref{theorem:main_degeneration_affinized_BPS},
\[
 [xu^m,yu^n]=\sum_{j=0}^{m+n-1}a_{j}u^{j}+\frac{\abs{\dd}^m\abs{\ee}^n}{\abs{\dd+\ee}^{m+n}}[x,y]u^{m+n}
\]
with $a_{j}\in\rmH^*(\BPS_{\Pi_Q})\cong \frakn_{\Pi_Q}^{+,\rmBPS}$. Since $\frakn_{\Pi_Q}^{+,\rmBPS}$ is concentrated in nonpositive cohomological degrees, $a_ju^j$ sits in cohomological degrees $<2(m+n)$ for $0\leq j<m+n$. Therefore, we must have $\sum_{j=0}^{m+n-1}a_{j}u^{j}=0$, which concludes.
\end{proof}

\section{Less perverse degeneration for the sheafified affinized BPS Lie algebra}

\subsection{$2$-Calabi--Yau abelian categories}
\label{subsection:2CY_Abelian_categories}
In \cite{davison2022bps}, we defined and studied the BPS Lie algebra of $2$-Calabi--Yau abelian categories satisfying some geometric assumptions. The BPS Lie algebra is defined as the positive part of a generalized Kac--Moody Lie algebra at both the sheaf (over the good moduli space) and absolute levels. We recall here our main geometric assumptions from \cite{davison2022bps}.

We let $\CC$ be an abelian category which is a finite length abelian category of the homotopy category of a $\BoC$-linear dg-category $\SC$. We assume that the stack of object $\FM_{\CC}$ of $\CC$ is a $1$-Artin open substack of the stack of objects $\FM_{\SC}$ of $\SC$ \cite[\S5.1.1]{davison2022bps}. We further assume

\begin{enumerate}
 \item The morphism $p\colon \mathfrak{Exact}_{\CC}\rightarrow\FM_{\CC}$ from the stack of short exact sequences to the stack of objects of $\CC$ keeping the middle term is proper and representable \cite[Assumption~1]{davison2022bps}.
 \item Each connected component of $\FM_{\CC}$ has the resolution property \cite{totaro2004resolution}, and the RHom complex over $\FM_{\CC}\times\FM_{\CC}$ is quasi-isomorphic to a complex of vector bundles over each connected component \cite[Assumption~2]{davison2022bps}.
 \item The stack $\FM_{\CC}$ has a good moduli space $\JH_{\CC}\colon \FM_{\CC}\rightarrow\CM_{\CC}$ and $\CM_{\CC}$ is a scheme \cite[Assumption~3]{davison2022bps}.
 \item The direct sum morphism $\oplus\colon \CM_{\CC}\times\CM_{\CC}\rightarrow\CM_{\CC}$ is finite \cite[Assumption~4]{davison2022bps}.
 \item For any collection of pairwise non-isomorphic simple objects $\underline{\CF}\coloneqq(\CF_1,\hdots,\CF_{r})$ of $\CC$, the full dg-subcategory of $\SC$ generated by $\underline{\CF}$ carries a right 2-Calabi--Yau structure \cite[Assumption~5]{davison2022bps}.
 \item There is a positive determinant line bundle over $\FM_{\CC}$. That is, there is a line bundle $\CL$ over $\FM_{\CC}$ such that for any $\BoC$-point $x\in\FM_{\CC}$, the morphism $\mathrm{act}_x\colon\mathrm{B}\BoG_{\mathrm{m}}\rightarrow\FM_{\CC}$ given by the inclusion of scalar automorphisms of $x$, $\mathrm{act}_x^*(c_1(\CL))\neq 0$.
\end{enumerate}

Under these assumptions, one can define and study the cohomological Hall algebra $\CA_{\CC}$ of $\CC$ and also its sheafified version $\SA_{\CA}$, as an algebra structure on the complex of constructible sheaves $(\JH_{\CC})_*\BD\BoQ_{\FM_{\CC}}^{\vir}$. As usual and explained in \cite[\S8.1]{davison2023bps}, we twist the CoHA product by a bilinear form $\psi\colon\pi_0(\FM_{\CC})\times\pi_0(\FM_{\CC})\rightarrow\BoZ/2$ such that the symmetrization of $\psi$ is the Euler form $(-,-)_{\SC}$. We denote by $\SA_{\CC}^{\psi}$ and $\CA_{\CC}^{\psi}$ the twisted algebras. Moreover, there exists a Lie algebra object $\BPS_{\CC}\in\Perv(\CM_{\CC})$ together with a canonical morphism of Lie algebras $\BPS_{\CC}\rightarrow\SA_{\CC}^{\psi}$. Using the $\rmH^*_{\BoC^*}(\pt)$-action on $\SA_{\CC}^{\psi}$ induced by the choice of the determinant line bundle, we obtain the morphism $\BPS_{\CC}\otimes\rmH^*_{\BoC^*}(\pt)\rightarrow\SA_{\CC}^{\psi}$. By evaluating symmetric tensors via the cohomological Hall algebra product, the \emph{cohomological integrality isomorphism} holds:
\[
 \SA_{\CC}^{\psi}\cong\Sym_{\boxdot}(\BPS_{\CC}\otimes\rmH^*_{\BoC^*}(\pt))\,.
\]
We refer to \cite{davison2022bps,davison2023bps} for details.

\subsection{Local neighbourhood theorem for $2$-Calabi--Yau categories}
\label{subsection:local_neighbourhood_theorem}

We recall the local neighbourhood theorem for $2$-Calabi--Yau categories from \cite{davison2021purity}. Let $\CC$ be an abelian category satisfying the assumptions of \Cref{subsection:2CY_Abelian_categories}. Any closed point $x\in\CM_{\CC}$ corresponds to a  semisimple object $\CF=\bigoplus_{j=1}^r\CF_{j}^{m_j}$ where the $\CF_{j}$ are pairwise non-isomorphic simple objects of $\CC$. We let $\overline{Q}$ be the Ext-quiver of $\{\CF_i\colon 1\leq i\leq r\}$ and $Q$ a half of $\overline{Q}$ (\Cref{subsection:ext_quivers}). Then, there exists an affine pointed scheme $(A,y)$ with a $\GL_{\mathbf{m}}\coloneqq\prod_{1\leq i\leq r}\GL_{m_i}$-action and a diagram with Cartesian squares and \'etale horizontal morphisms
\[\begin{tikzcd}
	{(\FM_{\Pi_Q,\mathbf{m}},0_{\mathbf{m}})} & {(A/G,y)} & {(\FM_{\CC},x)} \\
	{(\CM_{\Pi_Q,\mathbf{m}},0_{\mathbf{m}})} & {(A\cms G,\overline{y})} & {(\CM_{\CC},\overline{x})}
	\arrow["{\JH_{\Pi_Q}}"', from=1-1, to=2-1]
	\arrow[from=1-2, to=1-1]
	\arrow[from=1-2, to=1-3]
	\arrow[from=1-2, to=2-2]
	\arrow["{\JH_{\CC}}", from=1-3, to=2-3]
	\arrow[from=2-2, to=2-1]
	\arrow[from=2-2, to=2-3]
\end{tikzcd}\,.\]
In particular, one can identify the fiber of $\JH_{\CC}$ over $\overline{x}$ and the fiber of $\JH_{\Pi_Q}$ over $0_{\mathbf{m}}$.

In \cite{davison2023bps}, the first author together with Davison and Schlegel Mejia exploited this theorem and the following comparison result (\Cref{proposition:comparison_local_neighbourhood}) to describe the BPS Lie algebra of $2$-Calabi--Yau categories by generators and relations.

We let $\underline{\CF}=\{\CF_1,\hdots,\CF_r\}$ be a set of pairwise non-isomorphic objects of $\CC$ and
\[
\begin{matrix}
 \imath_{\underline{\CF}}\colon&\BoN^{r}&\rightarrow&\CM_{\CC}\\
 &(m_i)_{1\leq i \leq r}&\mapsto&\bigoplus_{i=1}^r\CF_i^{m_i}\,.
\end{matrix}
\]
We also let
\[
 \begin{matrix}
  \imath_{\Nil}\colon&\BoN^r&\rightarrow&\CM_{\Pi_Q}\\
  &(m_i)_{1\leq i\leq r}&\mapsto&0_{\mathbf{m}}\,.
 \end{matrix}
\]

\begin{proposition}[{\cite[Corollary~7.3]{davison2022bps}}]
\label{proposition:comparison_local_neighbourhood}
There is a natural isomorphism of algebras $\imath_{\underline{\CF}}^!\SA_{\CC}\cong \imath^!_{\Nil}\SA_{\Pi_Q}$. After taking the perverse degenerations, there is an isomorphism of algebras $\imath_{\underline{\CF}}^!\pH(\SA_{\CC})\cong \imath^!_{\Nil}\pH(\SA_{\Pi_Q})$.
\end{proposition}

\begin{lemma}
\label{lemma:abs_monoid_morphism}
There exists a monoid homomorphism $\abs{-}\colon \pi_0(\FM_{\CC})\rightarrow\BoN$ such that $\abs{a}\neq 0$ for any $a\neq 0$ and such that, for any $a\in\pi_0(\FM_{\CC})$ and $x\in \FM_{\CC,a}$, the composition $\mathrm{B}\BoG_{\mathrm{m}}\rightarrow\FM_{\CC}\rightarrow\mathrm{B}\BoG_{\mathrm{m}}$ given by the inclusion of scalar automorphisms of $x$ and the determinant line bundle $\CL$ over $\FM_{\Pi_Q}$ is described by the one-dimensional representation of $\BoG_{\mathrm{m}}$ of weight $\abs{a}$.
\end{lemma}
\begin{proof}
Let $a\in \pi_0(\FM_{\CC})$ and $x\in\FM_{\CC,a}$ be a closed $\BoC$-point. Then, the composition $\rmB\BoG_{\rmm}\rightarrow\rmB\BoG_{\rmm}$ of the lemma is given by a one-dimensional representation of $\BoG_{\mathrm{m}}$ of weight $w_x\in\BoZ\setminus\{0\}$. It suffices to show that $w_x$ only depends on $a$, giving a well-defined integer $\abs{a}=w_x$. Then, the compatibility of $\CL$ with short exact sequences implies that $\pi_0(\FM_{\CC})\rightarrow \BoZ, a\mapsto\abs{a}$ is a monoid homomorphism.

To show that $w_x$ only depends on $a$, it suffices to show that it is locally constant. In an \'etale neighbourhood of $x$, $\FM_{\CC}$ is described by $\FM_{\Pi_Q}$ for some quiver $Q$, by the local neighbourhood theorem (\Cref{subsection:local_neighbourhood_theorem}). The result then follows from the result for $\FM_{\Pi_Q}$ (\Cref{lemma:weight_det_lb_preproj}).
\end{proof}

\subsection{The sheafified less perverse associated graded of the affinized BPS Lie algebra for $2$-Calabi--Yau categories}

We denote by $\psi\colon\pi_0(\CM_{\CC})\times\pi_0(\CM_{\CC})\rightarrow\BoZ/2\BoZ$ a bilinear form such that the symmetrization of $\psi$ coincides with the Euler form of $\SC$.

Since we work with the perversely degenerated CoHA products in this paper, we can drop the ad hoc \cite[Assumption~7]{davison2022bps} which is the associativity of the CoHA product on $\SA_{\CC}^{(\psi)}\coloneqq\JH_*\BD\BoQ_{\FM_{\CC}}^{\vir}$ thanks to the following proposition.

\begin{proposition}
\label{proposition:associativity_perv_deg}
The less perversely degenerated CoHA product $\pH(\frakm^{(\psi)})\colon \SA_{\CC}^{(\psi)}\boxdot\SA_{\CC}^{(\psi)}\rightarrow\SA_{\CC}^{(\psi)}$ is associative. The less perversely degenerated CoHA product on $\CA_{\CC}^{(\psi)}$ is associative.
\end{proposition}
\begin{proof}
The associativity of $\CA_{\CC}^{(\psi)}$ follows from that of $\SA_{\CC}^{(\psi)}$ by taking derived global sections; and the associativity of $\SA_{\CC}$ and $\SA_{\CC}^{\psi}$ can be proven in the exact same way. Therefore, it suffices to prove that $\SA_{\CC}$ with the product $\pH(\frakm)$ is associative.

The associativity is the equality of the morphisms
$\pH(\frakm)\circ (\id_{\SA_{\CC}}\boxdot\pH(\frakm))$ and $\pH(\frakm)\circ (\pH(\frakm)\boxdot \id_{\SA_{\CC}})$ from $\SA_{\CC}^{\boxdot 3}\rightarrow\SA_{\CC}$. We let $D$ be the difference of these two morphisms. Since it is a morphism between $\BoZ$-graded semisimple perverse sheaves, the morphism $D$ vanishes if and only if it vanishes at any point $x\in\CM_{\CC}$, i.e. if $\imath_x^!D$ vanishes for any $x\in\CM_{\CC}$. The point $x$ corresponds to a semisimple object $\CF=\bigoplus_{j=1}^r\CF_j^{\oplus m_j}$ in $\CC$ where the $\CF_i$'s are pairwise non-isomorphic simple objects of $\CC$. We let $\overline{Q}$ be the Ext-quiver of $\{\CF_1,\hdots,\CF_r\}$ -- \Cref{subsection:ext_quivers}. This is the double of a quiver $Q$. By \Cref{proposition:comparison_local_neighbourhood}, there is an isomorphism of algebras
\begin{equation}
\label{equation:fiber_Cohas}
 (\imath_{\underline{\CF}}^!\SA_{\CC},\imath_{\underline{\CF}}^!\pH(\frakm_{\CC}))\cong (\imath_{\Nil}^!\SA_{\Pi_Q},\imath_{\Nil}^!\pH(\frakm_{\Pi_Q}))
\end{equation}
Since $ (\SA_{\Pi_Q},\frakm_{\Pi_Q})$ is associative (by \cite{davison2022bps}), $\frakm_{\Pi_Q}\circ (\frakm_{\Pi_Q}\boxdot\id_{\SA_{\Pi_Q}})-\frakm_{\Pi_Q}\circ(\id_{\SA_{\Pi_Q}}\boxdot\frakm_{\Pi_Q})=0$. Therefore, we also have
\[
 \pH(\frakm_{\Pi_Q})\circ (\pH(\frakm_{\Pi_Q}))\boxdot\id_{\pH(\SA_{\Pi_Q})})-\pH(\frakm_{\Pi_Q})\circ(\id_{\pH(\SA_{\Pi_Q})}\boxdot\pH(\frakm_{\Pi_Q}))=0\,.
\]
After applying $\imath_{\Nil}^!$, we obtain
\[
 \imath_{\Nil}^!\pH(\frakm_{\Pi_Q})\circ (\imath_{\Nil}^!\pH(\frakm_{\Pi_Q}))\boxdot\id_{\imath_{\Nil}^!\pH(\SA_{\Pi_Q})})-\imath_{\Nil}^!\pH(\frakm_{\Pi_Q})\circ(\id_{\imath_{\Nil}^!\pH(\SA_{\Pi_Q})}\boxdot\imath_{\Nil}^!\pH(\frakm_{\Pi_Q}))=0\,.
\]
By the isomorphism of algebra \eqref{equation:fiber_Cohas}, we obtain $\imath_{\underline{\CF}}^!D=0$ and in particular, $\imath_x^!D=0$ since $x$ belongs to the image of $\imath_{\underline{\CF}}$.
\end{proof}

\begin{theorem}
\label{theorem:less_perverse_2CY}
Let $\CC$ be a $2$-Calabi--Yau abelian category satisfying the assumptions in \Cref{subsection:2CY_Abelian_categories}. Then, the less perversely degenerated Lie bracket on $\SA_{\CC}^{\psi}$ induces a Lie bracket on $\BPS_{\CC}\otimes\rmH^*_{\BoC^*}(\pt)$. This Lie bracket is the $\abs{-}$-twisted $\rmH^*_{\BoC^*}(\pt)$-linear trivial extension of the Lie bracket on $\BPS_{\CC}$.
\end{theorem}
\begin{proof}
We consider the natural morphism $\imath\colon\BPS_{\CC}\otimes\rmH^*_{\BoC^*}(\pt)\rightarrow\SA_{\CC}^{\psi}$ induced by the morphism $\BPS_{\CC}\rightarrow\SA_{C}$ and the $\rmH^*_{\BoC^*}(\pt)$-action on $\SA_{\CC}^{\psi}$ coming from the choice of a positive determinant line bundle $\CL$. After taking the associated graded with respect to the less perverse filtration, it gives a monomorphism $\BPS_{\CC}\otimes\rmH^*_{\BoC^*}(\pt)\rightarrow\pH(\SA_{\CC}^{\psi})$ in the category $\Perv^{\BoZ}(\CM_{\CC})$. The $\BoZ$-graded perverse sheaf $\pH(\SA_{\CC}^{\psi})$ has the associative algebra structure given by $\pH(\frakm^{\psi})$ (the so-called \emph{less perversely degenerated CoHA product}), \Cref{proposition:associativity_perv_deg}. We would like first to prove that induced Lie bracket ${^{\frakp}\![-,-]}$ on $\pH(\SA_{\CC}^{\psi})$ preserves the subobject $\BPS_{\CC}\otimes\rmH^*_{\BoC^*}(\pt)$. We consider the composition
\[
 E\colon(\BPS_{\CC}\otimes\rmH^*_{\BoC^*}(\pt))^{\boxdot 2}\xrightarrow{\imath^{\boxdot 2}}\pH(\SA_{\CC}^{\psi})^{\boxdot2}\xrightarrow{{^{\frakp}\![-,-]}}\pH(\SA_{\CC}^{\psi})\rightarrow\coker(\imath)\,.
\]
We have to show that this composition vanishes. Since this is a morphism between $\BoZ$-graded semisimple perverse sheaves, it vanishes if and only if it vanishes at any point $x\in\CM_{\CC}$. If it does not vanish, we let $\CG\subset \coker(\imath)$ be a simple direct summand of the image of $E$. We fix $x\in\CM_{\CC}$ a point in the support of $\CG$ so that the stalk of $\CG$ at $x$ is non-zero. The point $x$ corresponds to a semisimple object $\CF=\bigoplus_{j=1}^r\CF_j^{\oplus m_j}$ in $\CC$. We let $\underline{\CF}=\{\CF_j\colon 1\leq j\leq r\}$ be the corresponding set of simple direct summands of $\CF$. Then, $\imath_{\underline{\CF}}^!E$ coincides with the analogous morphism $E_{\Pi_Q}$ obtained for $\CC={\Pi_Q}$, where $Q$ is half of the Ext-quiver of $\underline{\CF}$, by \Cref{proposition:comparison_local_neighbourhood}. Since the latter does not vanish, there exists a direct summand $\CH\subset \coker(\imath_{\Pi_Q})$ such that $0$ lies in the locus where $\CH$ is given by a local system so that $\imath_{\Nil}^!\CH\neq 0$. We also have $\imath_{\Nil}^*\CH\neq 0$. However, by $\BoC^*$-equivariance, $\imath_{\Nil}^*\CH\cong\rmH^*(\CH)$ and we obtain a contradiction with \Cref{theorem:main_degeneration_affinized_BPS} which states that $\rmH^*(E_{\Pi_Q})$ vanishes (since $\rmBPS_{\Pi_Q}\otimes\rmH^*_{\BoC^*}(\pt)$ is stable under the less perversely degenerated Lie bracket). Therefore, $\BPS_{\CC}\otimes\rmH^*_{\BoC^*}(\pt)$ is preserved by the less perversely degenerated Lie bracket ${^{\frakp}[-,-]}$.

The Lie bracket ${^{\frakp}[-,-]}$ on $\BPS_{\CC}\otimes\rmH^*_{\BoC^*}(\pt)$ induces for each $m,n\in\BoN$ and $a,b\in\pi_0(\CM_{\CC})$ a morphism
\[
 C_{a,b,m,n}[-m-n]\colon\pH^m(\BPS_{\CC,a}\otimes\rmH^*_{\BoC^*}(\pt))\boxdot\pH^n(\BPS_{\CC,b}\otimes\rmH^*_{\BoC^*}(\pt))\rightarrow\pH^{m+n}(\BPS_{\CC,a+b}\otimes\rmH^*_{\BoC^*}(\pt))
\]
and therefore a morphism $C_{a,b,m,n}\colon\BPS_{\CC,a}\boxdot\BPS_{\CC,b}\rightarrow\BPS_{\CC,a+b}$.

To identity the Lie bracket ${^{\frakp}[-,-]}$ on $\BPS_{\CC}\otimes\rmH^*_{\BoC^*}(\pt)$ with the $\abs{-}$-twisted $\rmH^*_{\BoC^*}(\pt)$-linear extension of the Lie bracket of $\BPS_{\CC}$, we consider for each $m,n\in\BoN$ and $a,b\in\pi_0(\FM_{\CC})$ the difference
\[
 D_{a,b,m,n}\coloneqq C_{a,b,m,n}-\frac{\abs{a}^m\abs{b}^n}{\abs{a+b}^{m+n}}[-,-]\colon\BPS_{\CC,a}\boxdot\BPS_{\CC,b}\rightarrow\BPS_{\CC,a+b}
\]
where $[-,-]\colon\BPS_{\CC}\boxdot\BPS_{\CC}\rightarrow\BPS_{\CC}$ denotes the Lie bracket of the sheafified BPS Lie algebra $\BPS_{\CC}$. The goal is to prove that it vanishes. Again, we prove this vanishing using a local argument based on the local neighbourhood theorem of \Cref{subsection:local_neighbourhood_theorem}. If this morphism does not vanish, then there exists a point $x\in\CM_{\CC}$ such that, denoting by $\imath_x\colon \{x\}\rightarrow \CM_{\CC}$ the inclusion, $\imath_x^!D_{a,b,m,n}\neq 0$. We pick a simple direct summand $\CG$ of $\BPS_{\CC}\boxdot\BPS_{\CC}$ such that the restriction of $\imath_x^!D_{a,b,m,n}$ to $\imath_x^!\CG$ is non-zero. The point $x$ corresponds to a semisimple object $\CF\cong\bigoplus_{i=1}^r\CF_i^{m_i}$ of $\CC$, where the $\CF_i$'s are pairwise non-isomorphic simple objects of $\CC$. We let $\overline{Q}$ be the Ext-quiver of $\underline{\CF}\coloneqq\{\CF_1,\hdots,\CF_r\}$. It is the double of a quiver $Q$. We let $\dd,\ee\in\BoN^{Q_0}$ be such that $\imath_{\Nil}(\dd)=a$, $\imath_{\Nil}(\ee)=b$.

By \Cref{proposition:comparison_local_neighbourhood}, $\imath_{\Nil}^!D_{\dd,\ee,m,n}\neq 0$ and so $D_{\dd,\ee,m,n}\neq 0$ does not vanish on a simple direct summand $\CG'$ of $\BPS_{\Pi_Q}\boxdot\BPS_{\Pi_Q}$ such that $\imath_0^!\CG'\neq0$. Then, $\imath_0^*\CG'\neq 0$ and by taking derived global sections, we obtain a contradiction with \Cref{theorem:main_degeneration_affinized_BPS}. This concludes and provides the formula for the less perversely degenerated Lie bracket on $\BPS_{\CC}\otimes\rmH^*_{\BoC^*}(\pt)$.
\end{proof}

A special case of \Cref{theorem:less_perverse_2CY} is the situation $\CC=\rep(\Pi_Q)$ for some quiver $Q$, which we write explicitly for later use. It provides a sheaf-theoretic enhancement of \Cref{theorem:main_degeneration_affinized_BPS}.

\begin{corollary}
\label{corollary:affinizedBPSsheaf_preproj}
Let $Q$ be a quiver and $\Pi_Q$ be the preprojective algebra of $Q$. Then, there is a split inclusion $\BPS_{\Pi_Q}\otimes\rmH^*_{\BoC^*}(\pt)\rightarrow \pH(\SA_{\Pi_Q}^{\psi})$ and the less perversely degenerated Lie bracket on $\pH(\SA_{\Pi_Q}^{\psi})$ induces a Lie bracket on $\BPS_{\Pi_Q}\otimes\rmH^*_{\BoC^*}(\pt)$. This Lie bracket is the $\abs{-}$-twisted $\rmH^*_{\BoC^*}(\pt)$-linear trivial extension of the Lie bracket on $\BPS_{\Pi_Q}$.\qed
\end{corollary}

\begin{corollary}
 In the situation of \Cref{theorem:less_perverse_2CY}, the subspace $\rmBPS_{\CC}[u]\coloneqq\rmBPS_{\CC}\otimes\rmH^*_{\BoC^*}(\pt)$ is stable under the less perversely degenerated Lie bracket on $\rmH^*(\SA_{\CC}^{\psi})=\CA_{\CC}^{\psi}$. The Lie bracket on $\rmBPS_{\CC}[u]$ is the $\abs{-}$-twisted $\rmH^*_{\BoC^*}(\pt)$-linear trivial extension of the Lie bracket on $\rmBPS_{\CC}$.
\end{corollary}
\begin{proof}
This follows by taking derived global sections of \Cref{theorem:less_perverse_2CY}.
\end{proof}

\begin{corollary}
\label{corollary:enveloping_affinizedBPS}
Let $\CC$ be a $2$-Calabi--Yau abelian category satisfying the assumptions \Cref{subsection:2CY_Abelian_categories}. There is an isomorphism of algebra objects in $\Perv^{\BoZ}(\CM_{\CC})$:
\[
 \bfU(\BPS_{\CC}\otimes\rmH^*_{\BoC^*}(\pt))\rightarrow \pH(\SA_{\CC}^{\psi})\,.
\]
There is an isomorphism of algebras
\[
 \bfU(\frakn_{\CC}^{+,\rmBPS}[u])\rightarrow\rmH^*(\SA_{\CC}^{\psi}),
\]
where $\rmH^*(\SA_{\CC}^{\psi})$ has the less perversely degenerated CoHA product.
\end{corollary}
\begin{proof}
Clearly, the first statement implies the second by taking derived global sections and the fact that $\mathcal{BPS}_{\mathcal{C}}$ is pure \cite{davison2021purity}. By \Cref{theorem:less_perverse_2CY}, there is a Lie algebra homomorphism $\BPS_{\CC}\otimes\rmH^*_{\BoC^*}(\pt)\rightarrow \pH(\SA_{\CC}^{\psi})$. By the universal property of the universal enveloping algebra of a Lie algebra object on the category $\Perv^{\BoZ}(\CM_{\CC})$, this morphism can be extended to a morphism of algebras $\bfU(\BPS_{\CC}\otimes\rmH^*_{\BoC^*}(\pt))\rightarrow\pH(\SA_{\CC}^{\psi})$. By the respective PBW morphism for $\bfU(\BPS_{\CC}\otimes\rmH^*_{\BoC^*}(\pt))$ and $\pH(\SA_{\CC}^{\psi})$, we obtain a commutative triangle of morphisms in $\Perv^{\BoZ}(\CM_{\CC})$ in which the diagonal arrows are isomorphisms:
\[\begin{tikzcd}
	& {\Sym(\BPS_{\CC}\otimes\rmH^*_{\BoC^*}(\pt))} \\
	{\bfU(\BPS_{\CC}\otimes\rmH^*_{\BoC^*}(\pt))} & {} & {\pH(\SA_{\CC}^{\psi})}
	\arrow["{\mathrm{PBW}}"', from=1-2, to=2-1]
	\arrow["{\mathrm{PBW}}", from=1-2, to=2-3]
	\arrow[from=2-1, to=2-3]
\end{tikzcd}\]
Therefore, the horizontal map is also an isomorphism. This concludes.
\end{proof}

\subsection{Deformed dimensional reduction and (affinized) BPS Lie algebra for deformed canonical cubic potential}
In this section, we explain how to generalize our results from triple quivers with the canonical cubic potential to triple quiver with deformed canonical cubic potential via deformed dimensional reduction.
\subsubsection{BPS Lie algebra for deformed canonical cubic potentials}
\label{subsubsection:BPSLiedefcanonicalpotential}
Let $Q$ be a quiver, $\tilde{Q}$ the triple quiver and $\tilde{W}$ its canonical cubic potential. We let $W_0\in\BoC[\overline{Q}]$ be a potential for $\overline{Q}$ and $W\coloneqq\tilde{W}+W_0$ be the deformed cubic potential. We let $\SA_{\tilde{Q},W}\in\CD_{\rmc}^{+}(\CM_{\tilde{Q}})$ be the sheafified CoHA of $(\tilde{Q},W)$. We let $\pi'\colon\CM_{\tilde{Q}}\rightarrow \CM_{\overline{Q}}$ be the projection forgetting the extra loops of $\tilde{Q}$.

We define $\BPS_{\Pi_Q,W_0}\coloneqq \varphi_{\Tr(W_0)}\BPS_{\Pi_Q}$, $\rmBPS_{\Pi_Q,W_0}\coloneqq\rmH^*(\BPS_{\Pi_Q,W_0})$. They have the Lie bracket induced by the Lie bracket on $\BPS_{\Pi_Q}$ after taking the vanishing cycle sheaf functor and derived global sections.

The potential $W$ is \emph{quasi-homogeneous} if for some torus action of $\BoG_{\rmm}^{\tilde{Q}_1}$ on the arrows of $\tilde{Q}$, the potential $W$ has positive weight. Deformed dimensional reduction \cite{davison2022deformed} yields the following proposition.
\begin{proposition}
\label{proposition:deformed_dim_red}
 We assume that $W$ is a quasi-homogeneous potential for $\tilde{Q}$. Then, there is a canonical isomorphism $\pi'_*\SA_{\tilde{Q},W}\cong\varphi_{\Tr(W_0)}\SA_{\Pi_Q}$ of constructible complexes in $\CD^+_{\rmc}(\CM_{\Pi_Q})$.
\end{proposition}
We define the \emph{dimensionally reduced} cohomological Hall algebra of $(\tilde{Q},W)$ as $\SA_{\tilde{Q},W}^{\red}\coloneqq \pi^{\prime}_*\SA_{\tilde{Q},W}$. We also define the dimensionally reduced BPS sheaf by $\mathcal{BPS}^{\mathrm{red}}_{\tilde{Q},W}=\pi^{\prime}_*\mathcal{BPS}_{\tilde{Q},W} [-1]$. By \Cref{lemma:dimredBPS_deformedpotential}, it is perverse.

\subsubsection{Structure of the BPS Lie algebra for deformed canonical cubic potential}
In this section, we determine the structure of the BPS Lie algebra for a triple quiver $\tilde{Q}$ with deformed canonical cubic potential $W=\tilde{W}+W_0$, where $W_0$ is a potential for the double quiver $\overline{Q}$.

\begin{lemma}
\label{lemma:dimredBPS_deformedpotential}
Let $\tilde{Q}$ be the triple of a quiver $Q$ and $W=\tilde{W}+W_0$ be a quasi-homogeneous potential where $\tilde{W}$ is the canonical cubic potential of $\tilde{Q}$ and $W_0\in\BoC[\overline{Q}]$. The complex of constructible sheaves $\pi'_*\BPS_{\tilde{Q},W}[-1]\in\CD^{\rmb}_{\rmc}(\CM_{\overline{Q}})$ is perverse. It is isomorphic to $\varphi_{\Tr(W_0)}\BPS_{\Pi_Q}$.
\end{lemma}
\begin{proof}
The cohomological integrality isomorphism for $\SA_{\Pi_Q}$ reads
\[
 \SA_{\Pi_Q}\cong\Sym_{\boxdot}\left(\BPS_{\Pi_Q}\otimes\rmH^*_{\BoC^*}(\pt)\right)\,.
\]
By the cohomological integrality isomorphism for $\SA_{\tilde{Q},W}$, $\pi'_*\BPS_{\tilde{Q},W}[-1]$ is a direct factor of $\pi'_*\SA_{\tilde{Q},W}\cong\varphi_{\Tr(W_0)}\SA_{\Pi_Q}$, where the isomorphism comes from \Cref{proposition:deformed_dim_red}. Since the latter complex is isomorphic to its total perverse cohomology (by perverse t-exactness of the vanishing cycle sheaf functor), the same is true for $\pi'_*\BPS_{\tilde{Q},W}[-1]$. Moreover, $\varphi_{\Tr(W_0)}\SA_{\Pi_Q}$ is concentrated in nonnegative perverse degrees since $\varphi_{\Tr(W_0)}$ is perverse t-exact, and therefore so is $\pi'_*\BPS_{\tilde{Q},W}[-1]$. We obtain $\pi'_*\BPS_{\tilde{Q},W}[-1]\cong\bigoplus_{j\geq 0}\pH^j(\pi'_*\BPS_{\tilde{Q,W}}[-1])[-j]$.

The cohomological integrality for $\SA_{\tilde{Q},W}$ gives, after applying the strict monoidal functor $\pi'_*\colon\CD^+_{\rmc}(\CM_{\tilde{Q}})\rightarrow\CD^+_{\rmc}(\CM_{\overline{Q}})$:
\[
 \pi'_*\SA_{\tilde{Q},W}\cong \Sym_{\boxdot}\left(\pi'_*\BPS_{\tilde{Q},W}[-1]\otimes\rmH^*_{\BoC^*}(\pt)\right)\,.
\]
The perverse degree $0$ parts of the cohomological integrality isomorphisms for $\SA_{\Pi_Q}$ and $\pi'_*\SA_{\tilde{Q},W}$ give
\[
 \Sym(\varphi_{\Tr(W_0)}\BPS_{\Pi_Q})\cong\Sym(\pH^0(\pi'_*\BPS_{\tilde{Q},W}[-1]))\,.
\]
By induction on the dimension vector, we obtain the existence of an abstract isomorphism $\varphi_{\Tr(W_0)}\BPS_{\Pi_Q}\cong\pH^0(\pi'_*\BPS_{\tilde{Q},W}[-1])$. Now, all of $\varphi_{\Tr(W_0)}\BPS_{\Pi_Q}$, $\pi_*\BPS_{\tilde{Q},W}[-1]$ and $\pH^0(\pi'_*\BPS_{\tilde{Q},W}[-1])$ satisfy the cohomological integrality isomorphism for $\pi_*\SA_{\tilde{Q},W}$ which concludes.
\end{proof}

\begin{proposition}
\label{proposition:sheaf_env_algebras}
 The enveloping algebras in the category $\Perv(\CM_{\overline{Q}})$ of $\varphi_{\Tr(W_0)}\BPS_{\Pi_Q}$ and $\pi'_*\BPS_{\tilde{Q},W}[-1]$ are both isomorphic to $\pH^0(\varphi_{\Tr(W_0)}\SA_{\Pi_Q}^{\psi})$.
\end{proposition}
\begin{proof}
By \cite[Theorem~1.1]{davison2023bps}, $\pH^0(\SA_{\Pi_Q}^{\psi})$ is isomorphic to the enveloping algebra of $\BPS_{\Pi_Q}$ in $\Perv(\CM_{\Pi_Q})$. By applying the perverse exact functor $\varphi_{\Tr(W_0)}$, it follows that $\varphi_{\Tr(W_0)}\pH^0(\SA_{\Pi_Q}^{\psi})\cong \pH^0(\varphi_{\Tr(W_0)}\SA_{\Pi_Q}^{\psi})$ is the universal enveloping algebra of $\varphi_{\Tr(W_0)}\BPS_{\Pi_Q}$ in $\Perv(\CM_{\Pi_Q})$.

Moreover, by \Cref{lemma:dimredBPS_deformedpotential}, there is an inclusion of a direct factor $\pi'_*\BPS_{\tilde{Q},W}[-1]\rightarrow \pH^0(\pi'_*\SA_{\tilde{Q},W}^{\psi})$ that is a Lie algebra homomorphism. By the universal property of the universal enveloping algebra, we obtain an algebra homomorphism $\bfU(\pi'_*\BPS_{\tilde{Q},W}[-1])\rightarrow \pH^0(\varphi_{\Tr(W_0)}\SA_{\Pi_Q}^{\psi})$ which must be an isomorphism by the respective PBW isomorphisms for $\varphi_{\Tr(W_0)}\BPS_{\Pi_Q}$ and $\pi_*\BPS_{\tilde{Q},W}[-1]$ and \Cref{lemma:dimredBPS_deformedpotential}.
\end{proof}

\begin{corollary}
\label{corollary:env_alg_def_dim_red}
There is an isomorphism of algebras $\bfU(\rmBPS_{\tilde{Q},W})\cong \bfU(\rmBPS_{\Pi_Q,W_0})$. Both are isomorphic to $\rmH^*(\varphi_{\Tr(W_0)}\pH^0(\SA_{\Pi_Q}^{\psi}))$.
\end{corollary}
\begin{proof}
It follows by taking the derived global sections of \Cref{proposition:sheaf_env_algebras}.
\end{proof}

\begin{remark}
 It is not immediate to decide from \Cref{corollary:env_alg_def_dim_red} whether $\rmBPS_{\tilde{Q},W}$ and $\rmBPS_{\Pi_Q,W_0}$ are isomorphic as Lie algebras or not. In general, given two Lie algebras $\frakg$ and $\frakh$, it is not true that an isomorphisnm $\bf{U}(\frakg)\cong\bfU(\frakh)$ implies an isomorphism $\frakg\cong\frakh$. This is called the \emph{isomorphism problem of universal enveloping algebras}, see \cite{schneider2011isomorphism} for some references. In the rest of this paper, we use the algebra structure on $\CA_{\tilde{Q},W}$ given by the isomorphism $\CA_{\tilde{Q},W}^{\psi}\cong\rmH^*(\varphi_{\Tr(W_0)}\SA_{\Pi_Q}^{\psi})$.
\end{remark}

As in \cite{davison2023bps}, we define the set of \emph{primitive simple positive roots} $\Sigma_{\Pi_Q}\subset\BoN^{Q_0}$ as the subset of dimension vectors $\dd\in\BoN^{Q_0}$ for which there exists a simple representation of $\Pi_Q$ of dimension $\dd$. We define the set of simple positive roots $\Phi_{\Pi_Q}^+\coloneqq \Sigma_{\Pi_Q}\sqcup\{\ell\dd\colon \ell\geq 2, \dd\in\Sigma_{\Pi_Q}, (\dd,\dd)_Q=0\}$. For $\dd\in\Sigma_{\Pi_Q}$, we define $\SG_{\dd}\coloneqq \ICS(\CM_{\Pi_Q,\dd})$ and $\SG_{\ell\dd}=(\Delta_{\ell})_*\ICS(\CM_{\Pi_Q,\dd})$ where $\Delta_{\ell}\colon\CM_{\Pi_Q,\dd}\rightarrow\CM_{\Pi_Q,\ell\dd}$, $x\mapsto x^{\oplus\ell}$ is the $\ell$-th direct sum power of $x$.

\begin{theorem}
\label{theorem:BPS_deformed}
The Lie algebra object $\BPS_{\Pi_Q,W_0}=\varphi_{\Tr(W_0)}\BPS_{\Pi_Q}$ may be described by generators and relations as the quotient of the free Lie algebra generated by $\varphi_{\Tr(W_0)}\SG_{\dd}$ for $\dd\in \Phi_{\Pi_Q}^+$ by the Serre relations
\[
 \ad(\varphi_{\Tr(W_0)}\SG_{\dd})^{1-(\dd,\dd')}(\varphi_{\Tr(W_0)}\SG_{\dd'}) \quad\text{for $\dd,\dd'\in\Phi_{\Pi_Q}^+$ such that $(\dd,\dd)_Q=2$ or $(\dd,\dd')=0$}.
\]
\end{theorem}
\begin{proof}
The presentation of $\BPS_{\Pi_Q,W_0}$ by generators and relations follows from that of $\BPS_{\Pi_Q}$ given in \cite{davison2023bps} since $\varphi_{\Tr(W_0)}\colon\Perv(\CM_{\Pi_Q})\rightarrow\Perv(\CM_{\Pi_Q})$ is a strict monoidal functor.
\end{proof}

\begin{corollary}
\label{corollary:inclusion_real_subalgebra}
There is an inclusion of Lie algebras $\frakn_{Q^{\re}}^+\subset \rmH^0(\BPS_{\tilde{Q},W}^{\red})\subset\rmBPS_{\tilde{Q},W}$.
\end{corollary}
\begin{proof}
The explanation from \cite[\S7.2.1]{davison2020bps} in the case of affine type A quivers works \emph{mutatis mutandis}. Namely, for $\dd$ concentrated at a single real vertex, the function induced by $\Tr(W_0)$ on $\FM_{\Pi_Q,\dd}$ vanishes and therefore, $\BPS_{\Pi_Q,\dd}=\BPS_{\Pi_Q,W_0,\dd}$. By \Cref{theorem:BPS_deformed}, the subalgebra of $\rmBPS_{\Pi_Q,W_0}$ generated by $\rmBPS_{\Pi_Q,W_0,\dd}$ for $\dd$ concentrated at a single real vertex is isomorphic to $\frakn_{Q^{\re}}^+$.
\end{proof}

In \cite[\S7.2.1]{davison2020bps}, Davison asks the question whether on can always find a potential $W_0\in\BoC[\overline{Q}]$ such that $\rmBPS_{\Pi_Q,W_0}=\frakn_Q^+$ and positively answered the question of affine type A quivers \cite[\S7.2.1]{davison2020bps}. The same strategy gives a positive answer in types $\mathrm{D}_n$ and $\mathrm{E}_6$.

\begin{lemma}
\label{lemma:function_vanishing_imaginary_roots}
 Let $\tilde{S}$ be the minimal resolution of the type $\mathrm{D}_n$ singularity with equation $x^2+y^{n-1}+yz^2=0$. Then, the function $y$ on $\tilde{S}$ is such that $\rmH^*(\varphi_y\BoQ_{\tilde{S}})\cong\BoQ^n$.
\end{lemma}
\begin{proof}
The strategy is the same for all $n$. The reduced subscheme of $y=0$ inside $\tilde{S}$ is the union of the exceptional curves over $0$ and the line parametrized by $z$. The equation $y=1$ defines the curve $x^2+z^2=-1$ which is isomorphic to $\BoA^1$. Therefore, the long exact sequence in cohomology involving vanishing and nearby cycles splits (by purity) and gives the result. Namely, by $\BoC^*$-equivariance, one may identify the nearby cycle cohomology with the cohomology of $\{y=1\}$ and we obtain as in \cite[]{davison2020bps}
\[
0\rightarrow\rmH^i(\varphi_y\BoQ)\rightarrow\rmH^i(\{y=0\})\rightarrow\rmH^i(\{y=1\})\rightarrow 0\,.
\]
\end{proof}

\begin{remark}
 For the type $\mathrm{E}_8$ singularity given by $x^2+y^3+z^5=0$, it is much harder to find an explicit function $g$ on the minimal resolution $\tilde{S}$ such that $\rmH^*(\varphi_g\BoQ_{\tilde{S}})\cong\BoQ^8$. Typically, we obtain a $\rmH^2(\varphi_g\BoQ_{\tilde{S}})$ with too high dimension. For example, if $g=z^5$, then the reduced subscheme of $g^{-1}(0)\subset \tilde{S}$ is the union the the $8$ exceptional curves and the affine curve $x^2+y^3=0$ and the nearby fiber $C'$ is an affine curve of genus one. Therefore, the long-exact sequence gives
 \[
  \rmH^1(g^{-1}(0))=0\rightarrow\rmH^1(C')\cong\BoQ^2\rightarrow   \rmH^2(\varphi_{z^3}\BoQ_{\tilde{S}})\rightarrow \rmH^2(g^{-1}(0))\rightarrow\rmH^2(C')\,.
 \]
Since $C'$ is an affine curve, $\rmH^2(C')=0$. Therefore, $\rmH^2(\varphi_{z^3}\BoQ_{\tilde{S}})$ has dimension $2+8=10>8$.
\end{remark}

\begin{proposition}
\label{proposition:potentialBPSalgebra}
Let $Q$ be a quiver of type $\mathrm{D}_n$. Then, there exists a potential $W_0\in\BoC[\overline{Q}]$ such that $\rmBPS_{\Pi_Q,W_0}\cong\frakn_{Q}^+$.
\end{proposition}
\begin{proof}
By \Cref{theorem:BPS_deformed} and \Cref{corollary:inclusion_real_subalgebra}, it suffices by the arguments of \cite[\S7.2.1]{davison2020bps} to find a potential $W_0$ such that the function $\Tr(W_0)$ induced on the resolution of the minimal singularity of type $\mathrm{D}_n$ is the function given by \Cref{lemma:function_vanishing_imaginary_roots}. This is clear, since a potential is a linear combination of cycles and any function on $\CM_{\Pi_Q,\delta}$ where $\delta$ is the indivisible imaginary root of $Q$ is given by the trace along a linear combination of cycles \cite{le1990semisimple}.
\end{proof}

\begin{remark}
 By studying the generators of $\CM_{\Pi_Q,\delta}$ (the surface singularity of Dynkin type $Q$) for an affine quiver $Q$ and the indivisible imaginary root $\delta$, one can provide explicit choices of potentials satisfying $\rmBPS_{\Pi_Q,W_0}\cong\frakn_Q^+$. Namely, we consider the affine quiver of type $\mathrm{D}_n$ with the labeling of vertices:
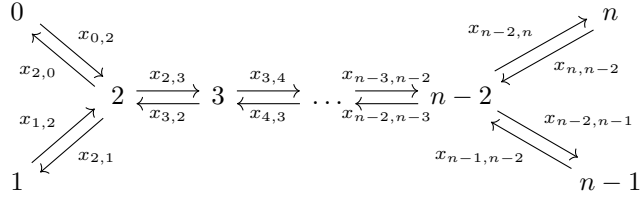
\[\begin{tikzcd}
	0 &&&&& n \\
	& 2 & 3 & \hdots & {n-2} \\
	1 &&&&& {n-1}
	\arrow["{x_{0,2}}", shift left, from=1-1, to=2-2]
	\arrow["{x_{n,n-2}}", shift left, from=1-6, to=2-5]
	\arrow["{x_{2,0}}", shift left, from=2-2, to=1-1]
	\arrow["{x_{2,3}}", shift left, from=2-2, to=2-3]
	\arrow["{x_{2,1}}", shift left, from=2-2, to=3-1]
	\arrow["{x_{3,2}}", shift left, from=2-3, to=2-2]
	\arrow["{x_{3,4}}", shift left, from=2-3, to=2-4]
	\arrow["{x_{4,3}}", shift left, from=2-4, to=2-3]
	\arrow["{x_{n-3,n-2}}", shift left, from=2-4, to=2-5]
	\arrow["{x_{n-2,n}}", shift left, from=2-5, to=1-6]
	\arrow["{x_{n-2,n-3}}", shift left, from=2-5, to=2-4]
	\arrow["{x_{n-2,n-1}}", shift left, from=2-5, to=3-6]
	\arrow["{x_{1,2}}", shift left, from=3-1, to=2-2]
	\arrow["{x_{n-1,n-2}}", shift left, from=3-6, to=2-5]
\end{tikzcd}\]
where $x_{i,j}$ are the elements in the path algebra of $\overline{Q}$.

Then, generators for the ring of functions on $\CM_{\Pi_Q,\delta}=\mu_{\delta}^{-1}(0)\cms \GL_{\delta}$ are given by
\[
y=\Tr(x_{2,0}x_{0,2}),\quad x=\Tr\bigl(
x_{0,2} x_{2,3} \cdots x_{n-3,n-2}
(x_{n-2,n-1} x_{n-1,n-2})
(x_{n-2,n} x_{n,n-2})
x_{n-2,n-3} \cdots x_{3,2} x_{2,0}
\bigr),
\]
and
\[
 z=u-v
\]
where
\[
 u=\Tr(x_{0,2}x_{2,3}\hdots x_{n-3,n-2}x_{n-2,n-1}x_{n-1,n-2}x_{n-2,n-3}\hdots,x_{3,2}x_{2,0}),
\]
\[
 v=\Tr(x_{0,2}x_{2,3}\hdots x_{n-3,n-2}x_{n-2,n}x_{n,n-2}x_{n-2,n-3}\hdots,x_{3,2}x_{2,0})\,.
\]
Therefore, a choice for $W_0$ is $x_{2,0}x_{0,2}\in\BoC[\overline{Q}]$.
\end{remark}

\subsubsection{The less perverse associated graded for deformed canonical cubic potential}

In this section, we extend our results regarding the affinized BPS Lie algebra for triple quivers with their canonical cubic potential to deformed canonical potentials.

\begin{theorem}
\label{theorem:perverse_degeneration_deformed_potential}
 Let $Q$ be a quiver, $\tilde{Q}$ the triple quiver and $\overline{Q}$ the double quiver. We let $\tilde{W}\in \BoC[\tilde{Q}]$ be the canonical cubic potential and $W_0\in \BoC[\overline{Q}]$ a potential for $\overline{Q}$. We let $W=\tilde{W}+W_0$ be the deformed potential. We assume that $W$ is quasi-homogeneous. Then, $\BPS_{\tilde{Q},W}^{\red}\otimes\rmH^*_{\BoC^*}(\pt)\in \CD^+_{\rmc}(\CM_{\Pi_Q})$ is stable under the less perversely degenerated Lie bracket which is given by the trivial $\abs{-}$-twisted extension of the Lie bracket on $\BPS_{\tilde{Q},W}^{\red}$.
\end{theorem}
\begin{proof}
The les perversely degenerated Lie bracket on $\SA_{\Pi_Q,W_0}^{\psi}$ is given by applying $\varphi_{\Tr(W_0)}$ to the less perversely degenerated Lie bracket on $\SA_{\Pi_Q}$. Since the later perserves $\BPS_{\Pi_Q}\otimes\rmH^*_{\BoC^*}(\pt)$ (\Cref{corollary:affinizedBPSsheaf_preproj}), the same is true after applying $\varphi_{\Tr(W_0)}$. The description of the Lie bracket on $\BPS_{\tilde{Q},W}^{\red}\otimes\rmH^*_{\BoC^*}(\pt)$ also follows from \Cref{corollary:affinizedBPSsheaf_preproj}.
\end{proof}

\begin{corollary}
In the situation of \Cref{theorem:perverse_degeneration_deformed_potential}, the subspace $\rmBPS_{\tilde{Q},W}^{\red}[u]\coloneqq\rmBPS_{\tilde{Q},W}^{\red}\otimes\rmH^*_{\BoC^*}(\pt)\subset \rmH^*\SA_{\tilde{Q},W}^{\red,\psi}$ is stable under the less perversely degenerated Lie bracket and the Lie bracket on $\rmBPS_{\tilde{Q},W}^{\red}[u]$ is the trivial extension of the Lie bracket on $\rmBPS_{\tilde{Q},W}^{\red}$.
\end{corollary}
\begin{proof}
This follows from \Cref{theorem:perverse_degeneration_deformed_potential} by taking derived global sections.
\end{proof}

\begin{corollary}
In the situation of \Cref{theorem:perverse_degeneration_deformed_potential}, there is an isomorphism of algebra objects in $\Perv^{\BoZ}(\CM_{\Pi_Q})$:
\[
 \bfU(\BPS_{\tilde{Q},W}^{\red}\otimes\rmH^*_{\BoC^*}(\pt))\rightarrow\pH\SA_{\tilde{Q},W}^{\red,\psi}\,.
\]
There is an isomorphism of algebras
\[
 \bfU(\rmBPS_{\tilde{Q},W}[u])\rightarrow\rmH^*(\SA_{\tilde{Q},W}^{\psi})
\]
where $\rmH^*(\SA_{\CC})$ is given the perversely degenerated CoHA product.
\end{corollary}
\begin{proof}
The second statement follows from the first by taking derived global sections and the fact that the less perverse filtration on $\pH(\SA_{\tilde{Q},W}^{\red,(\psi)})$ is split.
\end{proof}

\section{The deformed cohomological Hall algebra}
In this section, we turn on an extra torus action acting on the arrows of the quiver. We explain how to describe the associated graded of the affinized BPS Lie algebra in this deformed situation.
\subsection{Equivariant parameters}


Let $Q$ be a quiver and $W\in\BoC[Q]$ a potential for $Q$. We let $T'$ be an auxiliary torus acting by rescaling the arrows of $Q$ so that the potential is $T'$-invariant. Then, we may consider the quotient by $T'$ of the good moduli space morphism $\JH^{T'}\colon \FM_{Q}^{T'}\rightarrow \CM_{Q}^{T'}$, where $\FM_{Q}^{T'}$ and $\CM_{Q}^{T'}$ are the quotient stacks $\FM_{Q}^{T'}\coloneqq \FM_{Q}/T'$ and $\CM_{Q}^{T'}\coloneqq \CM_{Q}/T'$. The construction of the CoHA product gives a product on $(\JH^{T'})_*\varphi_{\Tr(W)}\BoQ_{\FM_Q}^{\vir}\in\CD^+_{\rmc,T'}(\CM_{Q})$ (\Cref{subsection:cohas}). If $Q$ is a symmetric quiver, we may define the BPS Lie algebra sheaf $\BPS_{Q,W}^{T'}\coloneqq \pH^1(\SA_{Q,W}^{T',\psi})[-1]$. We let $\rmBPS_{Q,W}^{T'}\coloneqq\rmH^*(\BPS_{Q,W}^{T'})$. By \cite[Theorem~3.4]{davison2022affine}, it carries a Lie algebra structure.

We now let $\tilde{Q}$ be the triple of some quiver and $\tilde{W}$ be the canonical cubic potential on $\tilde{Q}$. Then, the torus $T'$ acts on the arrows of $\overline{Q}$ and the preprojective relation is homogeneous. We let $\pi\colon\CM_{\tilde{Q}}^{T'}\rightarrow\CM_{\overline{Q}}^{T'}$ be the natural morphism forgetting the loops of $\tilde{Q}$. We define the dimensionally reduced BPS Lie algebra $\BPS_{\Pi_Q}^{T'}\coloneqq\pi_*\BPS_{\tilde{Q},W}^{T'}[-1]$. We let $\rmBPS_{\Pi_Q}^{T'}\coloneqq \rmH^*(\BPS_{\Pi_Q}^{T'})$. It is a trivial deformation of the BPS Lie algebra $\BPS_{\Pi_Q}$:

\begin{proposition}[{\cite[Proposition~11.13]{davison2023bps}}]
\label{proposition:dhm2023_deformedBPS}
There is an isomorphism of Lie algebras $\rmBPS_{\Pi_Q}^{T'}\cong\rmBPS_{\Pi_Q}\otimes\rmH^*_{T'}(\pt)$: the Lie bracket on $\rmBPS_{\Pi_Q}^{T'}$ is the trivial $\rmH^*_{T'}(\pt)$-linear extension of the Lie bracket on $\rmBPS_{\Pi_Q}$.
\end{proposition}
The purity of $\rmBPS_{\Pi_Q}$ is used to prove the isomorphism $\rmBPS_{\Pi_Q}^{T'}\cong\rmBPS_{\Pi_Q}\otimes\rmH^*_{T'}(\pt)$.

\subsection{Associated graded of the affinized BPS Lie algebra of the deformed cohomological Hall algebra}
\label{subsection:associated_graded_deformedCoHA}


\begin{theorem}
\label{theorem:Tprime_deformed_sheaf}
 Let $Q$ be a quiver and $\Pi_Q$ its preprojective algebra. We let $T'$ be extra torus acting on the arrows of $\overline{Q}$ so that the preprojective relation is homogeneous. Then, the perversely degenerated Lie bracket on the deformed CoHA $\SA_{\Pi_Q}^{T',\psi}$ induces a Lie bracket on $\BPS_{\Pi_Q}^{T'}\otimes\rmH^*_{\BoC^*}(\pt)$. It is the trivial $\abs{-}$-twisted extension of the Lie bracket on $\BPS_{\Pi_Q}^{T'}$.
\end{theorem}
\begin{proof}
We consider the natural morphism $\imath_{T'}\colon\BPS_{\Pi_Q}^{T'}\otimes\rmH^*_{\BoC^*}(\pt)\rightarrow\SA_{\Pi_Q}^{T',\psi}$. After taking the perverse perverse degeneration with respect to $\JH_{\Pi_Q}^{T'}$, it becomes a monomorphism in the category $\Perv^{\BoZ}(\CM_{\Pi_Q}^{T'})$ of $T'$-equivariant $\BoZ$-graded perverse sheaves on $\CM_{\Pi_Q}$. After forgetting the $T'$-equivariance, we recover the natural morphism $\imath\colon\BPS_{\Pi_Q}\otimes\rmH^*_{\BoC^*}(\pt)\rightarrow\SA_{\Pi_Q}^{\psi}$. We consider the composition
\[
 (\BPS_{\Pi_Q}^{T'}\otimes\rmH^*_{\BoC^*}(\pt))\boxdot (\BPS_{\Pi_Q}^{T'}\otimes\rmH^*_{\BoC^*}(\pt))\xrightarrow{\imath_{T'}^{\boxdot 2}}\SA_{\Pi_Q}^{T',\psi}\boxdot \SA_{\Pi_Q}^{T',\psi}\xrightarrow{{^{\frakp}\![-,-]}}\SA_{\Pi_Q}^{T',\psi}\rightarrow \coker(\imath_{T'})\,.
\]
If this composition does not vanish, then the same composition after applying the faithful forgetful functor $\Perv^{\BoZ}(\CM_{\Pi_Q}^{T'})\rightarrow\Perv^{\BoZ}(\CM_{\Pi_Q})$ does not vanish either. This contradicts \Cref{corollary:affinizedBPSsheaf_preproj}.

The Lie bracket on $\BPS_{\Pi_Q}^{T'}\otimes\rmH^*_{\BoC^*}(\pt)\cong \bigoplus_{j\geq 0}\BPS_{\Pi_Q}^{T'}[-2j]$ induced by the perversely degenerated Lie bracket is given by morphisms of $T'$-equivariant perverse sheaves
\[
 \BPS_{\Pi_Q}^{T'}[-2i]\boxdot\BPS_{\Pi_Q}^{T'}[-2j]\rightarrow\BPS_{\Pi_Q}^{T'}[-2(i+j)]
\]
for any $i,j\geq 0$. Applying the forgetful functor again, we recover the components of the Lie bracket on $\BPS_{\Pi_Q}\otimes\rmH^*_{\BoC^*}(\pt)$. Therefore, this morphism is given by the shift by $-2(i+j)$ of the Lie bracket on $\BPS_{\Pi_Q}^{T'}$ by \Cref{corollary:affinizedBPSsheaf_preproj}. The proves the last statement of the theorem.
\end{proof}

\begin{corollary}
\label{corollary:TprimeCoHA}
The perversely degenerated $T'$-deformed affinized BPS Lie algebra of $\Pi_Q$ is the $\rmH^*_{T'}(\pt)$-module $\rmBPS_{\Pi_Q}^{T'}\otimes\rmH^*_{\BoC^*}(\pt)\cong \rmBPS_{\Pi_Q}\otimes\rmH^*_{T'\times\BoC^*}(\pt)$ with the $\abs{-}$-twisted trivial extension of the Lie bracket on $\rmBPS_{\Pi_Q}$.
\end{corollary}
\begin{proof}
By taking derived global sections of \Cref{theorem:Tprime_deformed_sheaf}, the Lie bracket on $\rmBPS_{\Pi_Q}^{T'}\otimes\rmH^*_{\BoC^*}(\pt)$ is the trivial $\abs{-}$-twisted $\rmH^*_{\BoC^*}(\pt)$-linear extension of the Lie bracket on $\rmBPS_{\Pi_Q}^{T'}$. By combining this with \Cref{proposition:dhm2023_deformedBPS}, we obtain further that $\rmBPS_{\Pi_Q}^{T'}\otimes\rmH^*_{\BoC^*}(\pt)\cong \rmBPS_{\Pi_Q}\otimes\rmH^*_{T'\times\BoC^*}(\pt)$ has the Lie bracket from $\BPS_{\Pi_Q}$ linearly $\rmH^*_{T'\times\BoC^*}$ extended with the $\abs{-}$-twist.
\end{proof}

\begin{corollary}
\label{corollary:Tprimeenveloping}
There is an isomorphism of algebras in the category $\Perv(\CM_{\Pi_Q}^{T'})$:
\[
 \bfU(\BPS_{\Pi_Q}^{T'}\otimes\rmH^*_{\BoC^*}(\pt))\rightarrow\pH(\SA_{\Pi_Q}^{T',\psi})\,.
\]
There is an isomorphism of algebras
\[
 \bfU(\rmBPS_{\Pi_Q}^{T'}[u])\rightarrow\rmH^*(\SA_{\Pi_Q}^{T',\psi})
\]
where $\SA_{\Pi_Q}^{T'}$ has the perversely degenerated CoHA product.
\end{corollary}
\begin{proof}
The second statement follows from the first by applying the derived global sections functor. By \Cref{theorem:Tprime_deformed_sheaf}, there is a morphism of Lie algebras $\BPS_{\Pi_Q}^{T'}\otimes\rmH^*_{\BoC^*}(\pt)\rightarrow\pH(\SA_{\Pi_Q}^{T',\psi})$. By the universal property of universal enveloping algebras, it can be extended to a morphism of algebras $\bfU(\BPS_{\Pi_Q}^{T'}\otimes\rmH^*_{\BoC^*}(\pt))\rightarrow\pH(\SA_{\Pi_Q}^{T',\psi})$. This is a morphism between cohomologically graded $T'$-equivariant perverse sheaves which becomes an isomorphism after forgetting the $T'$-equivariance by \Cref{corollary:enveloping_affinizedBPS}. Therefore, by it is an isomorphism, which concludes the proof.
\end{proof}

\subsection{Associated graded for the BPS Lie algebra of deformed cohomological algebra of the triple quiver with deformed canonical cubic potential}

We may now generalize the results of \Cref{subsection:associated_graded_deformedCoHA} to CoHAs of triple quivers with deformed canonical cubic potential.

Let $Q$ be a quiver, $\tilde{Q}$ the triple quiver and $\tilde{W}$ the canonical cubic potential. We let $W_0\in\BoC[\overline{Q}]\subset\BoC[\tilde{Q}]$ be a potential only involving the arrows of $\overline{Q}$. We let $T'$ be a torus acting on the arrows of $\tilde{Q}$ so that the potential $W\coloneqq \tilde{W}+W_0$ is invariant and quasi-homogeneous. All constructions of \Cref{subsubsection:BPSLiedefcanonicalpotential} work $T'$-equivariantly, which gives the BPS Lie algebra sheaf $\BPS_{\tilde{Q},W}^{T'}\in \Perv(\CM_{\tilde{Q}}^{T'})$, and the dimensionally reduced BPS Lie algebra sheaf $\BPS_{\tilde{Q},W}^{T',\red}\in\Perv(\CM_{\overline{Q}}^{T'})$. We define similarly the less perverse filtration on $\SA_{\tilde{Q},W}^{T',\red,(\psi)}\cong\varphi_{\Tr(W_0)}\SA_{\Pi_Q}^{T',(\psi)}$.

By combining \Cref{theorem:Tprime_deformed_sheaf} with the perverse t-exactness of the vanishing sheaf functor, we obtain the following.

\begin{theorem}
\label{theorem:afinizedT'deformedpotential}
The affinized BPS sheaf $\BPS_{\tilde{Q},W}^{T',\red}\otimes\rmH^*_{\BoC^*}$ is stable under the perversely degenerated Lie bracket on $\SA_{\Pi_Q}^{T',\red}$. The Lie bracket on $\BPS_{\tilde{Q},W}^{T',\red}\otimes\rmH^*_{\BoC^*}$ is the $\abs{-}$-twisted trivial extension of the Lie bracket on $\BPS_{\tilde{Q},W}^{T',\red}$.
\end{theorem}
\begin{proof}
The Lie bracket on $\SA_{\Pi_Q,W_0}$ is obtained by applying $\varphi_{\Tr(W_0)}$ to the Lie bracket on $\SA_{\Pi_Q}$. The latter preserves $\BPS_{\Pi_Q}\otimes\rmH^*_{\BoC^*}(\pt)$ and so, the former preserves $\varphi_{\Tr(W_0)}\BPS_{\Pi_Q}\otimes\rmH^*_{\BoC^*}(\pt)$. Moreover, the Lie bracket on $\BPS_{\Pi_Q}\otimes\rmH^*_{\BoC^*}(\pt)$ is the $\abs{-}$-twisted trivial extension of the Lie bracket on $\BPS_{\Pi_Q}$, and therefore so is the Lie bracket on $\BPS_{\Pi_Q,W_0}\otimes\rmH^*_{\BoC^*}(\pt)$.
\end{proof}

\begin{corollary}
In the situation of \Cref{theorem:afinizedT'deformedpotential}, the affinized BPS Lie algebra $\rmBPS_{\tilde{Q},W}^{T',\red}\otimes\rmH^*_{\BoC^*}(\pt)\subset\rmH^*(\SA_{\tilde{Q},W}^{T',\red,\psi})$ is stable under the perversely degenerated Lie bracket. The Lie induced Lie bracket on $\rmBPS_{\tilde{Q},W}^{T',\red}\otimes\rmH^*_{\BoC^*}(\pt)$ is the trivial $\abs{-}$-twisted $\rmH^*_{\BoC^*}(\pt)$-linear extension of the Lie bracket on $\rmBPS_{\tilde{Q},W}^{T',\red}$.
\end{corollary}
\begin{proof}
This follows from \Cref{theorem:afinizedT'deformedpotential} by applying the derived global section functor.
\end{proof}

If we assume further that $\rmH^*(\varphi_{\Tr(W_0)}\BPS_{\Pi_Q})=\rmBPS_{\tilde{Q},W}$ is pure, then there we obtain the $T'$-equivariant formality of $\rmH^*_{T'}(\BPS_{\Pi_Q,W_0})$ and so, there is an isomorphism $\rmBPS_{\tilde{Q},W}^{T'}\cong\rmBPS_{\Pi_Q,W_0}\otimes\rmH^*_{T'}(\pt)$.

\begin{corollary}
In the situation of \Cref{theorem:afinizedT'deformedpotential}, if $\rmBPS_{\tilde{Q},W}$ carries a pure mixed Hodge structure, then there is an isomorphism of Lie algebras $\rmBPS_{\tilde{Q},W}^{T'}\otimes\rmH^*_{\BoC^*}(\pt)\cong \rmBPS_{\tilde{Q},W}\otimes\rmH^*_{T'\times\BoC^*}(\pt)$ where the right-hand-side has the $\abs{-}$-twisted $\rmH^*_{T'\times\BoC^*}(\pt)$-linear extension of the Lie bracket on $\rmBPS_{\tilde{Q},W}$.
\end{corollary}

\begin{corollary}
In the situation of \Cref{theorem:afinizedT'deformedpotential}, there an isomorphism of algebra objects in $\Perv(\CM_{\Pi_Q}^{T'})$
\[
 \bfU(\BPS_{\tilde{Q},W}^{T',\red}\otimes\rmH^*_{\BoC^*}(\pt))\rightarrow\pH(\SA_{\tilde{Q},W}^{T',\red,\psi})\,.
\]
There is an isomorphism of algebras
\[
 \bfU(\rmBPS_{\tilde{Q},W}^{T'}[u])\rightarrow\rmH^*(\SA_{\tilde{Q},W}^{T',\red,\psi})
\]
where $\rmH^*(\SA_{\tilde{Q},W}^{T',\red,\psi})$ has the perversely degenerated CoHA product.
\end{corollary}
\begin{proof}
By \Cref{theorem:afinizedT'deformedpotential}, there is a morphism of Lie algebras $\BPS_{\Pi_Q,W_0}^{T'}\otimes\rmH^*_{\BoC^*}(\pt)\rightarrow\pH(\SA_{\Pi_Q,W_0}^{T',\psi})$, which extends to a morphism of algebras
\[
 \bfU(\BPS_{\Pi_Q,W_0}^{T'}\otimes\rmH^*_{\BoC^*}(\pt))\rightarrow\pH(\SA_{\Pi_Q,W_0}^{T',\psi})\,.
\]
By the perversely degenerated PBW theorem for $\pH(\SA_{\Pi_Q,W_0}^{T'})$ and the PBW theorem for enveloping algebras for $\BPS_{\Pi_Q,W_0}^{T'}\otimes\rmH^*_{\BoC^*}(\pt)$, there is a commutative diagram in which diagonal arrows are isomorphism of cohomologically graded perverse sheaves:
\[\begin{tikzcd}
	& {\Sym(\BPS_{\Pi_Q,W_0}^{T'}\otimes\rmH^*_{\BoC^*}(\pt))} \\
	{\bfU(\BPS_{\Pi_Q,W_0}^{T'}\otimes\rmH^*_{\BoC^*}(\pt))} && {\pH(\SA_{\Pi_Q,W_0}^{T',\psi})}
	\arrow["{\mathrm{PBW}}"', from=1-2, to=2-1]
	\arrow["{\mathrm{PBW}}", from=1-2, to=2-3]
	\arrow[from=2-1, to=2-3]
\end{tikzcd}\]
Therefore, the horizontal morphism is an isomorphism of algebra objects in $\Perv^{\BoZ}(\CM_{\Pi_Q}^{T'})$.

\end{proof}

\section{Degenerations of nilpotent CoHAs}
\label{section:degenerations-nilpotent}
It has proven useful in the literature to consider \emph{nilpotent} versions of CoHAs in order to study the structure of non-nilpotent CoHAs, see \cite{schiffmann2020cohomological,schiffmann2023cohomological,davison2023bps}. Nilpotent CoHAs, obtained by restricting to substacks of nilpotent, or semi-nilpotent or strictly semi-nilpotent representations of a quiver, are dual (in a sense that can be made more precise, but irrelevant for this paper) to CoHAs without nilpotent support. They also enjoy their own structural results, such as tautological generation. Our degeneration results apply to the nilpotent versions of CoHAs.

Let $\CC$ be a two Calabi--Yau category satisfying the assumptions of \Cref{subsection:2CY_Abelian_categories}. We consider the good moduli space morphism $\JH\colon\FM_{\CC}\rightarrow\CM_{\CC}$. We define CoHAs associated to submonoids $\CN\subset\CM_{\CC}$ as in \cite[\S7.1.2]{davison2022bps}. For any saturated submoid $\imath\colon\CN\rightarrow\CM_{\CC}$, which means that the square
\[\begin{tikzcd}
	{\CN\times\CN} & \CN \\
	{\CM_{\CC}\times\CM_{\CC}} & {\CM_{\CC}}
	\arrow["\oplus", from=1-1, to=1-2]
	\arrow[from=1-1, to=2-1]
	\arrow["\lrcorner"{anchor=center, pos=0.125}, draw=none, from=1-1, to=2-2]
	\arrow[from=1-2, to=2-2]
	\arrow[from=2-1, to=2-2]
\end{tikzcd}\]
is Cartesian, we may define the $\CN$-CoHA $\CA_{\CC}^{\CN}$ and the sheafified $\CN$-CoHA $\SA_{\CC}^{\CN}$ as follows. If $(\SA_{\CC}^{(\psi)},\mathfrak{m}^{(\psi)})$ denotes the $\psi$-twisted or not sheafified CoHA of $\CC$, which is an algebra object in the category of constructible sheaves over $\CD^+_{\rmc}(\CM_{\CA})$, the $\CN$-sheafified CoHA is defined by $\SA_{\CC}^{\CN,(\psi)}\coloneqq\imath^!\SA_{\CC}^{(\psi)}$, and it has the product given by  $\imath^!\mathfrak{m}^{(\psi)}$. By applying the derived global sections functor, we obtain the $\CN$-CoHA $\CA_{\CC}^{\CN,(\psi)}\coloneqq\rmH^*(\SA_{\CC}^{\CN,(\psi)})$ with multiplication $\rmH^*(\mathfrak{m}^{(\psi)})$.

The nilpotent, semi-nilpotent and strictly semi-nilpotent CoHAs of the preprojective algebra of a quiver $Q$ are defined in way via appropriate choices of saturated submonoids $\CN\subset\CM_{\Pi_Q}$. Namely, the nilpotent CoHA $\CA_{\Pi_Q}^{\nil}$ is obtained for the choice $\imath\colon\BoN^{Q_0}\rightarrow\CM_{\Pi_Q}$, $\dd\mapsto 0_{\dd}$, sending a dimension vector $\dd\in\BoN^{Q_0}$ to the $\dd$-dimensional representation of $\Pi_Q$ on which all arrows of $\overline{Q}$ act by the zero morphism. The semi-nilpotent CoHA $\CA_{\Pi_Q}^{\mathcal{SN}}$ is define using the saturated submonoid $\CM_{Q}\subset\CM_{\Pi_Q}$ extending $Q$-representations to $\Pi_Q$-representations by setting all opposite arrows to $0$. Last, we denote by $Q^{\mathrm{loops}}\subset Q$ the full subquiver of $Q$ consisting of vertices of $Q$ together with loops. The strictly seminilpotent CoHA is obtained for the saturated submonoid $\CM_{Q^{\mathrm{loops}}}\subset\CM_{\Pi_Q}$ extending a $Q^{\mathrm{loops}}$-representation to a $\Pi_Q$-representation by setting all loops in $\overline{Q}$ but not $Q^{\mathrm{loops}}$ to $0$. Note that in the case when $Q$ has no loops and no oriented cycles, all three versions of nilpotent CoHAs $\CA_{\Pi_Q}^{\nil,(\psi)},\CA_{\Pi_Q}^{\mathcal{SN},(\psi)}$ and $\CA_{\Pi_Q}^{\mathcal{SSN},(\psi)}$ coincide.

For all sorts of nilpotent CoHAs and more generally $\CN$-CoHAs of $2$-Calabi--Yau categories, we define the less perverse filtration on $\CA_{\CC}^{\CN,(\psi)}$ by $!$-pullback of the less perverse filtration on $\SA_{\CC}^{(\psi)}$. Namely, we let $\mathfrak{L}^i\CA_{\CC}^{\CN,(\psi)}\coloneqq\rmH^*(\imath^!\ptau^{\leq i}\SA_{\CC}^{(\psi)})$. Then, it gives an increasing filtration of $\CA_{\CC}^{\CN,(\psi)}$ starting in degree $0$. The $\CN$-BPS Lie algebra of $\CC$ is defined as $\rmBPS_{\CC}^{\CN}\coloneqq\rmH^*(\imath^!\BPS_{\CC})$. It has the Lie bracket induced by the CoHA multiplication on $\CA_{\CC}^{\CN,(\psi)}$.

\begin{corollary}
There is an isomorphism of algebras
\[
 \Gr^{\FL}\CA_{\CC}^{\CN,\psi}\cong \bfU(\rmBPS_{\CC}^{\CN}\otimes\rmH^*_{\BoC^*}(\pt)).
\]
The Lie bracket on $\rmBPS_{\CC}^{\CN}\otimes\rmH^*_{\BoC^*}(\pt)$ is given by the $\abs{-}$-twisted $\rmH^*_{\BoC^*}(\pt)$-linear extension of the Lie bracket on $\rmBPS_{\CC}^{\CN}$.
\end{corollary}
\begin{proof}
The corollary follows by applying the strict monoidal functor $\rmH^*(\imath^!-)$ to \Cref{corollary:enveloping_affinizedBPS}.
\end{proof}

One may obtain versions of this theorem for deformed CoHA, given an action of a torus $T'$ on the arrow of the quiver $\overline{Q}$ such that the preprojective relation is homogeneous and the submonoid $\CN\subset\CM_{\Pi_Q}$ is preserved by the $T'$-action. In particular, the associated graded of nilpotent CoHAs of quivers is fully understood. While a full description of the nilpotent CoHAs of quivers is missing, it is described in \cite{diaconescu2025cohomological} in terms of Yangians in the case of affine quivers.

\section{Comparison with the filtration on the Maulik--Okounkov Yangians}
In \cite{maulik2019quantum}, Maulik and Okounkov construct Yangians associated with quivers using geometric methods. In particular, they construct a Yangian $\mathbf{Y}_Q$ associated with the preprojective algebra of a quiver $Q$. It is proved in \cite{botta2023okounkov} that there is an isomorphism of algebras between the CoHA \[ \Phi: \mathcal{A}^{T',\psi}_{\Pi_Q} \rightarrow \mathbf{Y}_Q^{T'} \] and the Yangian $\mathbf{Y}_Q^{T'}$ for suitably choosen torus $T'$.  The algebra on the right has a filtration defined by the degree of $u$ operators \cite[Section 5.2.10]{maulik2019quantum}. We show that this morphism respects the filtrations on both sides and induces an isomorphism on the associated graded algebras.

We follow \cite{botta2023okounkov} for quickly summarizing the definition of the Maulik--Okounkov Yangian and its filtration. Let $\NN_{Q}(\ff,\dd)$ be the Nakajima quiver variety associated with the quiver $Q$ with framing $\ff$ and dimension $\dd$. Let $T'$ be a torus acting on the arrows $a, a^{*}\in \overline{Q}$ such that $\sum [a,a^{*}]$ is homogeneous of equivariant weight $\hbar$. Let $T_{\ff} = T' \times A_{\ff}$ be the torus acting on $\NN_{Q}(\ff,\dd)$ where $A_{\ff}$ is the framing torus which acts on framing arrows $\alpha_{i,m}$ with weight $1$ and $\alpha_{i,m}^{*}$ with weight $-1$. Let \[ \NN^{T_{\ff}}_{Q,\ff} = \bigoplus_{\bd \in \mathbb{N}^{Q_0}} \NN^{T_{\ff}}_{Q,\ff,\dd} \ \ \mathrm{where } \ \  \NN^{T_{\ff}}_{Q,\ff,\dd} := \rmH^*_{T_{\ff}}(\NN_Q(\ff,\dd),\mathbb{Q}^{\vir}) \] be the $T_{\ff}$-equivariant cohomology of the Nakajima quiver variety, where $\mathbb{Q}^{\vir}: = \mathbb{Q}[\mathrm{dim}(\NN_Q(\ff,\dd))]$ is the constant sheaf shifted by the virtual dimension of $\NN_Q(\ff,\dd)$. Similarly, let $\NN^{T'}_Q = \bigoplus_{\bd \in \mathbb{N}^{Q_0}} \rmH^*_{T'}(\NN_Q(\ff,\dd),\mathbb{Q}^{\vir})$ and $\NN^{T}_{Q,\ff,\dd}$ be the $T'$-equivariant cohomology.

Then given two framing vectors $\ff^{\prime}$ and $\ff^{\prime \prime}$ such that $\ff = \ff^{\prime}+\ff^{\prime \prime}$, certain elements \[R_{\ff^{\prime},\ff^{\prime \prime}} \in \mathrm{End}_{\rmH^{*}_{T_{\ff}}}(\NN^{T_{\ff^{\prime}}}_{Q,\ff^{\prime}} \otimes_{\rmH^{*}_{T^{\prime}}} \NN^{T_{\ff^{\prime \prime}}}_{Q,\ff^{\prime \prime}}) \otimes_{\rmH^{*}_{T_{\ff}}} \mathrm{Frac}(\rmH^{*}_{T_{\ff}})\] are constructed using stable envelopes. Owing to their geometric definition, these elements satisfy the Yang--Baxter equation. More precisely, for any three framing vectors $\ff_1, \ff_2, \ff_3$ such that $\ff = \ff_1+\ff_2+\ff_3$ we have
\[ R^{12}_{\ff_1,\ff_2} R^{13}_{\ff_1,\ff_3}  R^{23}_{\ff_2,\ff_3} = R^{23}_{\ff_2,\ff_3}  R^{13}_{\ff_1,\ff_3}  R^{12}_{\ff_1,\ff_2} \] in $\mathrm{End}_{\rmH^{*}_{T_{\ff}}}(\NN^{T_{\ff_1}}_{Q,\ff_1} \otimes_{\rmH^{*}_{T^{\prime}}} \NN^{T_{\ff_2}}_{Q,\ff_2} \otimes_{\rmH^{*}_{T^{\prime}}} \NN^{T_{\ff_3}}_{Q,\ff_3} ) \otimes_{\rmH^{*}_{T_{\ff}}} \mathrm{Frac}(\rmH^{*}_{T_{\ff}})$. 

Note that if the framing vector is $\ff=\delta_i$ for some vertex $i \in Q_0$, where $\delta_i$ is the vector with $1$ at the $i$-th vertex and $0$ elsewhere, then, since the action of $A_{\delta_i}$ is trivial on the quiver variety $N_{Q}(\delta_i,\dd)$, we have \[ \NN^{T_{\delta_i}}_{Q,\delta_i} = \NN^{T'}_{Q,\delta_i} [a] \] where we are identifing $\rmH^{*}_{A_{\delta_i}} = \mathbb{Q}[a]$. Let \[ \NN^{T',\vee}_{Q,\delta_i} =  \bigoplus_{-\bd \in \mathbb{N}^{Q_0}} \mathrm{Hom}_{\rmH^{*}_{T'}}(\NN^{T'}_{Q,\delta_i,-\bd},\rmH^{*}_{T'}) \] be the graded dual of $\NN^{T'}_{Q,\delta_i}$ as $\rmH^{*}_{T'}$-module. This defines a trace morphism \[ \mathrm{Tr} : \NN^{T'}_{Q,\delta_i} \otimes_{\rmH^{*}_{T'}} \NN^{T',\vee}_{Q,\delta_i} \rightarrow \rmH^{*}_{T'} \] by taking the composition on the graded pieces. 

For any $i \in Q_0$, let $m_i(a) \in \NN^{T'}_{Q,\delta_i} \otimes_{\rmH^{*}_{T'}} \NN^{T',\vee}_{Q,\delta_i}[a]$ be a polynomial in $a$ with coefficients in $\NN^{T'}_{Q,\delta_i} \otimes_{\rmH^{*}_{T'}} \NN^{T',\vee}_{Q,\delta_i}$. Let $(\NN^{T'}_{Q,\delta_i}[a] \otimes_{{\rmH^{*}_{T'}} } \NN^{T',\vee}_{Q,\delta_i} \otimes_{{\rmH^{*}_{T'}} } \NN^{T_{\ff}}_{Q,\ff})_{\mathrm{Loc}}=   (\NN^{T'}_{Q,\delta_i}[a] \otimes_{{\rmH^{*}_{T'}} } \NN^{T',\vee}_{Q,\delta_i} \otimes_{{\rmH^{*}_{T'}} } \NN^{T_{\ff}}_{Q,\ff}) \otimes_{\rmH^{*}_{T_{\ff+\delta_i}}} \mathrm{Frac}(\rmH^{*}_{T_{\ff+\delta}})$. We consider the composition \[\begin{tikzcd}
	{ (\NN^{T'}_{Q,\delta_i}[a] \otimes_{{\rmH^{*}_{T'}} } \NN^{T',\vee}_{Q,\delta_i} \otimes_{{\rmH^{*}_{T'}} } \NN^{T_{\ff}}_{Q,\ff})} & { (\NN^{T'}_{Q,\delta_i}[a] \otimes_{{\rmH^{*}_{T'}} } \NN^{T',\vee}_{Q,\delta_i} \otimes_{{\rmH^{*}_{T'}} } \NN^{T_{\ff}}_{Q,\ff})_{\mathrm{Loc}}} \\
	& { (\NN^{T'}_{Q,\delta_i}[a] \otimes_{{\rmH^{*}_{T'}} } \NN^{T',\vee}_{Q,\delta_i} \otimes_{{\rmH^{*}_{T'}} } \NN^{T_{\ff}}_{Q,\ff})_{\mathrm{Loc}}} \\
	{ \NN^{T_{\ff}}_{Q,\ff}} & {  (\NN^{T_{\ff}}_{Q,\ff})_{\mathrm{Loc}}}
	\arrow["{{R_{\delta_i,\ff}^{(13)}}}", from=1-1, to=1-2]
	\arrow["{m_i(a) \otimes 1}", from=1-2, to=2-2]
	\arrow["{\mathrm{Tr}}", from=2-2, to=3-2]
	\arrow["{1 \otimes 1 \otimes \mathrm{Id}}", from=3-1, to=1-1]
	\arrow[dotted, from=3-1, to=3-2]
\end{tikzcd}\]
where the second arrow is given by applying $R_{\delta_i,\ff}$ to the first and third factors, the third arrow is given by taking the tensor product with $m_i(a)$ and the forth arrow is given by applying the trace morphism $\mathrm{Tr}$ to the first two factors. Then taking the residue of the above compostion, we obtain 
\[ E(m_i(a)) := \frac{1}{\hbar} \sum  \mathrm{Res}_a( \mathrm{Tr}((m_i(a) \otimes 1) R_{\delta_i,\ff}))\]

It is proved in \cite{maulik2019quantum} that the operator $E(m_i(a))$ is well-defined and lives in $\mathrm{End}_{\rmH^{*}_{T_{\ff}}}(\NN^{T_{\ff}}_{Q,\ff})$.

\begin{definition}\cite[Section 5.2.6]{maulik2019quantum}
	The Maulik--Okounkov Yangian $\mathbf{Y}^{T'}_Q$ is the subalgebra of $\prod_{\ff} \mathrm{End}_{\rmH^{*}_{T_{\ff}}}(\NN^{T_{\ff}}_{Q,\ff})$ generated by the operators $E(m_i(a))$ for all $i \in Q_0$ and all polynomials $m_i(a)$.
\end{definition} 

\begin{theorem}\cite[Section 5.3, Theorem 5.5.1]{maulik2019quantum}\label{theorem:MO_Yangian_properties}
The algebra $\mathbf{Y}^{T'}_Q$ enjoys the following properties:

	\begin{itemize}
		\item Let $\mathfrak{g}^{\MO,T'}_Q$ be the span of the operators $E(m_i)$ for all constant polynomials $m_i$. Then $\mathfrak{g}^{\MO,T'}_Q$ is closed under Lie bracket and is called the Maulik-Okounkov Lie algebra. Furthermore, the algebra $\mathbf{Y}^{T'}_Q$ is generated by $\mathfrak{g}^{\MO,T'}_Q$ and the operators of multiplication by the Chern classes of the tautological bundles on the Nakajima quiver variety.
		\item The algebra $\mathbf{Y}^{T'}_Q$ admits a natural increasing filtration $F^{\bullet}$ defined by setting degree of the generator $E(m_i(a))$ to be the degree of $a$ in $m_i(a)$, i.e \[deg(E(m_i(a)))=\deg_a(m_i(a)). \] The associated graded algebra is isomorphic to the universal enveloping algebra of $\mathfrak{g}^{\MO,T'}_Q[a]$, i.e, \[\bfU(\mathfrak{g}^{\MO,T'}_Q[a]) \cong \mathrm{Gr}^{F}(\mathbf{Y}^{T'}_Q) \] where $\mathfrak{g}^{\MO,T'}_Q[a]$ is the trivial $\BoQ[a]$-linear extension of Lie algebra structure on $\mathfrak{g}^{\MO,T'}_Q$ by the variable $a$ and any element $a^n \alpha$ for $\alpha \in \mathfrak{g}^{\MO,T'}_Q$ is in degree $n$.
		\item The Lie algebra $\mathfrak{g}^{\MO,T'}_Q$ admits a triangular decomposition \[ \mathfrak{g}^{\MO,T'}_Q \simeq \mathfrak{n}^{\MO,T',+}_Q \oplus \mathfrak{h}^{\MO,T'}_Q \oplus \mathfrak{n}^{\MO,T',-}_Q \] such that $\mathfrak{n}^{\MO,T',+}_Q$ is $\mathbf{N}^{Q_0}$-graded \[ \mathfrak{n}^{\MO,T',+}_Q  = \bigoplus_{\dd \in \mathbf{N}^{Q_0} \backslash 0} \mathfrak{n}^{\MO,T'}_{Q,\dd} \] where any operator $\alpha \in \mathfrak{n}^{\MO,T',+}_{Q,\dd}$ acts by \[ \alpha \colon \NN^{T_{\ff}}_{Q,\ff,\dd^{\prime}} \rightarrow \NN^{T_{\ff}}_{Q,\ff, \dd^{\prime}+\dd} \] for any $\dd, \dd^{\prime}$.
	\end{itemize}
\end{theorem}

Let $\mathbf{Y}^{T',+}_Q$ be the subalgebra of $\mathbf{Y}^{T'}_Q$ generated by $\mathfrak{n}^{\MO,T',+}_Q$ and multiplication by tautological classes. It is shown in \cite{davison2020cohomological} that for any framing $\ff$, there is an $\rmH^{*}_{T_{\ff}}$-linear action of the cohomological Hall algebra $\mathcal{A}^{T',\psi}_{\Pi_Q}$ on $\NN^{T_{\ff}}_{Q,\ff}$, thus giving a morphism of algebras \[ \mathcal{A}^{T',\psi}_{\Pi_Q} \rightarrow \prod_{\ff} \mathrm{End}_{\rmH^{*}_{T_{\ff}}}(\NN^{T_{\ff}}_{Q,\ff}). \] Then it is proved in \cite{schiffmann2023cohomological} and \cite{botta2023okounkov} that this morphism is an injection and the image of this morphism is exactly $\mathbf{Y}^{T',+}_Q$ and thus we have an isomorphism of algebras \[ \Phi: \mathcal{A}^{T',\psi}_{\Pi_Q} \rightarrow \mathbf{Y}^{T',+}_Q. \]  Note that since the Yangian is a $\mathbf{Z}^{Q_0}$-graded algebra, the filtration $F^{\bullet}$ restricts to the positive Half $\mathbf{Y}^{T',+}_Q$ of the Yangian. We then have

\begin{theorem}
	The morphism $\Phi$ induces an isomorphism between the less perverse filtration on the cohomological Hall algebra and the degree filtration on the Maulik--Okounkov Yangian, i.e, \[ \Phi(\mathfrak{L}^{\leq i}(\mathcal{A}^{T',\psi}_{\Pi_Q})) =  F^{\leq [i/2]}(\mathbf{Y}^{T',+}_Q) \] for all $i \geq 0$.
\end{theorem}

\begin{proof} We first check that the morphism $\Phi$ respects the filtrations in this way. The cohomological Hall algebra $\mathcal{A}^{T',\psi}_{\Pi_Q}$ is generated by $\mathfrak{n}^{+, \mathrm{BPS},T'}_{\Pi_Q} \otimes \rmH^*_{\BoC^*}(\pt)$, thus degree of elements in $\mathfrak{n}^{+, \mathrm{BPS},T'}_{\Pi_Q} \otimes \rmH^{*}_{\mathbb{C}^{*}}(\pt)$ defines the less perverse filtration on $\CA_{\Pi_Q}^{T'}$. By definition, any $\alpha \in \mathfrak{n}^{+, \mathrm{BPS},T'}_{\Pi_Q}$ is in $\mathfrak{L}^{0}(\mathcal{A}^{T',\psi}_{\Pi_Q})$ (the less perverse filtration is coarser than the perverse filtration). On the other hand, it is proved in \cite{botta2023okounkov} the the morphism $\Phi$ sends $\alpha \in \mathfrak{n}^{+, \mathrm{BPS},T'}_{\Pi_Q}$ to $\mathfrak{n}^{\MO,T',+}_{Q}$ and thus $\Phi(\alpha) \in F^0(\mathbf{Y}^{T',+}_Q)$ for any $\alpha\in\frakn_{\Pi_Q}^{+,\rmBPS,T'}$. The action of the cohomological Hall algebra on $\NN^{T_{\ff}}_{Q,\ff}$ is defined by convolution with tautological classes, thus the morphism $\Phi$ sends multiplication by tautological classes to multiplication by tautological classes. In particular, for any $\alpha\in\frakn_{\Pi_Q}^{+,\rmBPS,T'}$ and $n\geq 0$, $\Phi(u^n \alpha) = (\sum_{i\in Q_0} c_1(\mathcal{V}_i))^n \Phi(\alpha)$ where $\mathcal{V}_i$ is the tautological bundle on the $i$th vertex. But it is proved in \cite[Proposition 5.5.3]{maulik2019quantum} that the cup product with $c_1(\mathcal{V}_i)$ increases $F$ degree by 1, thus $\Phi(u^n \alpha) \in F^n(\mathbf{Y}^{T',+}_Q)$. This shows that $\Phi(\mathfrak{L}^{\leq 2n}(\mathcal{A}^{T',\psi}_{\Pi_Q})) \subset F^{\leq n}(\mathbf{Y}^{T',+}_Q)$ for all $n \geq 0$. Since $\mathfrak{L}^{\leq 2n}(\mathcal{A}^{T',\psi}_{\Pi_Q}) = \mathfrak{L}^{\leq 2n+1}(\mathcal{A}^{T',\psi}_{\Pi_Q})$, shows that $\Phi(\mathfrak{L}^{\leq i}(\mathcal{A}^{T',\psi}_{\Pi_Q})) \subset F^{\leq [i/2]}(\mathbf{Y}^{T',+}_Q)$ for all $i \geq 0$. Since the morphism $\Phi$ respect the filtrations, this induces a morphism on the associated graded algebras \[ \mathrm{Gr}(\Phi) \colon \mathrm{Gr}^{\mathfrak{L}}(\mathcal{A}^{T',\psi}_{\Pi_Q}) \rightarrow \mathrm{Gr}^F(\mathbf{Y}^{T',+}_Q).\]Since the morphism $\Phi$ is an isomorphism, the induced morphism on the associated graded algebras is surjective. But by Corollary \ref{corollary:PBW_affinized_BPS} and Theorem \ref{theorem:MO_Yangian_properties}, the associated graded algebras are isomorphic to $\bfU(\mathfrak{n}^{+, \mathrm{BPS},T'}_{\Pi_Q}[u])$ and $\bfU(\mathfrak{n}^{\MO,T',+}_Q[a])$ respectively. Thus the graded dimensions as a $\mathbf{N}^{Q_0} \times \mathbf{Z}$ graded vector space are the same under the morphism $\Phi$ where the $\mathbf{Z}$-grading is coming from the grading of the associated graded algebra. Thus $\mathrm{Gr}(\Phi)$ is surjective. Since each graded piece of the source and the target have the same dimension, $\mathrm{Gr}(\Phi)$ is an isomorphism. Since both filtration start in degree $0$, the isomorphism between the associated graded implies $\Phi(\mathfrak{L}^{\leq 0}(\mathcal{A}^{T',\psi}_{\Pi_Q})) =  F^{\leq 0}(\mathbf{Y}^{T',+}_Q)$. Now we may conclude by induction since we have pairs of exact sequences\[\begin{tikzcd}
	0 & {\mathfrak{L}^{\leq i-1}(\mathcal{A}^{T',\psi}_{\Pi_Q})} & {\mathfrak{L}^{\leq i}(\mathcal{A}^{T',\psi}_{\Pi_Q})} & {\mathrm{Gr}^{\mathfrak{L},i}(\mathcal{A}^{T',\psi}_{\Pi_Q})} & 0 \\
	0 & {F^{[(i-1)/2]}(\mathbf{Y}^{T',+}_Q)} & {F^{[i/2]}(\mathbf{Y}^{T',+}_Q)} & {\mathrm{Gr}^{F,[i/2]}(\mathbf{Y}^{T',+}_Q)} & 0
	\arrow[from=1-1, to=1-2]
	\arrow[from=1-2, to=1-3]
	\arrow[from=1-2, to=2-2]
	\arrow[from=1-3, to=1-4]
	\arrow[from=1-3, to=2-3]
	\arrow[from=1-4, to=1-5]
	\arrow[from=1-4, to=2-4]
	\arrow[from=2-1, to=2-2]
	\arrow[from=2-2, to=2-3]
	\arrow[from=2-3, to=2-4]
	\arrow[from=2-4, to=2-5]
\end{tikzcd}\]
for $i\geq 1$, where by induction, we may assume that first vertical arrow is an isomorphism. Since the third vertical arrow is an isomorphism, we conclude that the middle vertical arrow is an isomorphism as well by the snake lemma, which concludes the proof.
\end{proof}

\printbibliography

\end{document}